\documentclass[12pt]{article}
\usepackage[applemac]{inputenc}
\usepackage{graphics}
\usepackage{epic}
\usepackage{amsmath}
\usepackage{txfonts}
\usepackage[frenchb]{babel}
\usepackage{xcolor}

\usepackage{fancyhdr}

\usepackage{amsmath}
\usepackage{amsfonts}

\usepackage{txfonts}
\usepackage[frenchb]{babel}

\font\minititre=txss at 9 pt
\font\titresection=txbss at 12 pt

\font\texte = txr at 12pt
\font\textebf = txb at 12pt
\font \texteit=txi at 12 pt

\def\id{\mathrm{Id}}
\def\C{\mathbb{C}}
\def\H{\mathbb{H}}
\def\R{\mathbb{R}}
\def\Z{\mathbb{Z}}
\def\S{\mathbb{S}}

\def\I{\mathbb{I}}
\def\sn{\mathbb{SN}}
\def\tn{\mathbb{T}}
\def\tr{\mathrm{Tr}}
\def\hom{\mathrm{{Hom}_{X}}(\R^{2})}
\def\homplus{\mathrm{{Hom^{+}}_{\hspace{-1mm}X}}(\R^{2})}
\def\cp{\mathbb{CP}}
\def\rdeux{\mathrm{R}_{2}}
\def\mt{\mathrm{M}_{T}}
\def\ft{(f_{t})_{t\in [0;1]}}

\def\ov{\overrightarrow}

\def\GT{\color{mygreen}\uparrow\color{black}}
\def\GB{\color{mygreen}\downarrow\color{black}}
\def\RT{\color{red}\uparrow\color{black}}
\def\RB{\color{red}\downarrow\color{black}}

\definecolor{myred}{rgb}{0.97,0.15,0.04}
\definecolor{mygreen}{rgb}{0.1,0.62,0.33}
\definecolor{myblue}{rgb}{0.06,0.21,0.62}
\definecolor{myblack}{rgb}{1,1,1}

\newenvironment{remarks}{~\\ \hfill  \underline{\sc \large Remarks:} \hfill~\\~\begin{enumerate}}{\end{enumerate}}

\newtheorem{prop}{{\sc \large Proposition}}{}

\newenvironment{proof}{\underline{\sc \large Proof:}}{\hfill $\square$}

\newtheorem{theo}{\underline{\sc \large Theorem}}{}
\newtheorem{defi}{Definition}
\makeatletter
\renewcommand\section{\@startsection {section}{1}{\z@}%
        {-3.5ex \@plus -1ex \@minus -.2ex}%
        {2.3ex \@plus.2ex}%
        {\reset@font\titresection}}
\renewcommand{\fnum@figure}{\minititre{\figurename~\thefigure}} 
\makeatother 
\makeatletter 
\begin{document}
\texte
\thispagestyle{empty}
\enlargethispage{2cm}

\title{TURBULENT HOMEOMORPHISMS \\ AND THE TOPOLOGICAL SNAIL}
\author{Arnaud DEHOVE \\
\texteit{arnaud\underline{~~}dehove@yahoo.fr}}

\date{\today}

\maketitle

\section{Historical considerations}

The topological snail and the presence of turbulent fixed points appeared for the first time in public as the tittle of my conference at the dynamical systems seminar of the IMJ-PRG in Paris (Sorbonne-University, Jussieu), may 2024 the 31st. Nontheless, they had occupied my mind all along those years, since my first personnal reflexions about my doctoral thesis, when I perceived their presence. The counter-example to the main conjecture that founded my research already confronts turbulent fixed points and linking numbers (see \cite{ad1}, \cite{ad2}, \cite{ad3}). I imagined it during 1995 and published it with my thesis for the first time. Since then, it took me a long time to begin to understand and name those turbulent fixed points. Prior I had to convince myself that for both theoritical and practical reasons, not only the concept of a linking number between fixed points and periodic orbits of surface homeomorphisms had to be clarified, but also important basic tools of plane topology, as for example the Jordan curve theorem and the Schoenflies theorems. And this also took me years.

From this point of view, the project of a general elementary presentation of plane topology still seems important to me. Many ideas, that reveal fruitful in the present paper, directly come from a fundamental reflexion on plane topology. I have already begun to expose some of those ideas at the IMJ-PRG in Paris, back in november 2011. The spining skeleton theorem, the turbulent fixed points theorem and the topological snail must be considered not only as effective progress in topological dynamics, but also as steps and applications of this general project. And I also hold that the original formulation of the turbulent fixed points theorem is relevant. It is exposed here in the simplest manner, but has direct generalizations and many important illustrations that go back to my thesis, with related crucial theoritical and physical considerations. I will of course describe all them in this paper. They not only all contribute to enhance the topological snail as a new dynamical tool. They also give, in my view, the best possible justification for a foundational reflexion on plane topology.

\section{The original turbulent fixed points theorem.}

\begin{theo}\label{turbulent}

Three points, located on the axis $\Delta$ of a plane axial symetry $\sigma$, move during an isotopy to a homeomorphism $f$ : from $\sigma$ in the orientation reversing case, or from the identity otherwise. And one supposes that during this displacement, the middle point always flies to the left or to the right of on of the other two on $\Delta$, this displacement being successively followed by one or a certain number of analogous displacements of the midle point progressively appearing (see figure \ref{figure1}).  

\begin{figure}[htpb]
\begin{picture}(45,100)(-110,-15)

\color{mygreen}

\put(0,0){\qbezier[300](0,0)(22,30)(48,2)}
\put(48, 2){\vector(1,-1){0}}

\color{black}

\put(45,17){\makebox(0,0)[bc]{$f_{B}$}}

\put(-30,0){\circle*{1}}
\put(0,0){\circle*{1}}
\put(-80,2){\makebox(0,0)[bc]{$\Delta$}}
\put(0,0){\circle*{1}}
\put(30,0){\circle*{1}}

\put(-80,62){\makebox(0,0)[bc]{$\Delta$}}
\put(-30,60){\circle*{1}}
\put(0,60){\circle*{1}}
\put(0,60){\circle*{1}}
\put(30,60){\circle*{1}}

\put(-80, 60){\line(1,0){160}}
\put(80,60){\vector(1,0){0}}
\put(-80, 0){\line(1,0){160}}
\put(80,0){\vector(1,0){0}}

\put(-30,58){\makebox(0,0)[tc]{$p_{1}$}}
\put(-0,58){\makebox(0,0)[tc]{$ p_{2}$}}
\put(30,58){\makebox(0,0)[tc]{$ p_{3}$}}

\put(-30,-2){\makebox(0,0)[tc]{$p_{1}$}}
\put(-0,-2){\makebox(0,0)[tc]{$p_{2}$}}
\put(30,-2){\makebox(0,0)[tc]{$p_{3}$}}

\put(-45,77){\makebox(0,0)[bc]{$f_{A}$}}

\color{mygreen}
\put(0, 60){\qbezier[300](0,0)(-25,30)(-48,2)}
\put(-48, 62){\vector(-1,-1){0}}

\end{picture}
\caption{Displacements of the middle point, to the left or to the right relatively to $X= \Big\{ p_{0}; p_{1}; p_{2} \Big\}$ on $\Delta$. \label{figure1}}
\end{figure}

One then supposes that the three points form an invariant set $X$ of the homeomorphism $f$ and set 
$$A = \left[\begin{array}{cc} 1 & 1 \\ 0 & 1 \end{array} \right]~~~\mathrm{and}~~~B = \left[\begin{array}{cc} 1 & 0 \\ 1 & 1 \end{array} \right]\, .$$

We call turbulence matrix of $f$ relative to $X$ and $\Delta$ the matrix $\mt$ obtained by successively multipliying on the left the matrices $A$ or $B$ according to the correspondant displacements, to the left or to the right, of the middle point of $X$. Let $\lambda$ be the bigger proprer value of $\mt$.

Then the number $N_{X}(f)$ of Nielsen classes of fixed points of $f$ relative to $X$ is at least the trace $\mathrm{Tr}(\mt)$ of the turbulence matrix and its topological entropy is a least $\ln(\lambda)$.
\end{theo}

\begin{remarks}
\item In the case of an orientation preserving homeomorphism, the turbulence matrix $M_{T}(f)$ is obtained from $\id$ by successively multipliying on the left the matrices $A$ or $B$ according to the correspondant displacements, to the left or to the right, of the middle point of $X$. In the case of an orientation reversing homeomorphism, it is obtained similarily but from the initial matrix $Y=\left[\begin{array}{cc} -1 & 0 \\ 0 & 1 \end{array}\right]$ that is associated to $\sigma$. See below theorem \ref{canonicalform} page \pageref{canonicalform} and theorem \ref{classification} page \pageref{classification} for the explanation of this additionnal fact.

\item Similarily, the matrix $M_{T}$ must be considered as an element of $\mathrm{PSL}_{2}(\Z)$ and its trace as the absolute value $\vert \tr(\mt) \vert$. See theorem \ref{canonicalform} page \pageref{canonicalform} and theorem \ref{classification} page \pageref{classification} .
\end{remarks}

\section{Homeomorphisms of the type $A$ and of the type $B$.}

Let us first fix the horizontal oriented axis $\Delta^{+}$ defined in $\R^{2}$ by the equation $y=0$, and its classical positive orientation. The axial symetry about $\Delta^{+}$ is therefore given by $\sigma(x;y) = (x;-y)$. There is a half-plane located on the left of $\Delta^{+}$ and defined by 
$$\H^{+} = \lbrace (x;y) \in \R^{2}~|~ y>0 \rbrace\, .$$
I call this half-plane the top half-plane, as opposed to the bottom half-plane equal to $\H^{-}~=~\sigma(\H^{+})$. I suppose that $X = \lbrace p_{1}; p_{2}; p_{3} \rbrace \subset \Delta^{+}$ contains three different points with coordinates $p_{1}(x_{1}; 0)$; $p_{2}(x_{2}; 0)$; $p_{3}(x_{3}; 0)$ and $x_{1} < x_{2} < x_{3}$. I will first define the notion of a top-displacement of the middle point, to the right or to the left and generally restrict myself to the case of homeomorphisms that preserve the orientation. More precisely, I will always keep in mind that an homeomorphism $f$ has always a conjugated one $f\circ \sigma$ and that they must be studied together. I will call positively turbulent a homeomorphism $f$ that is a finite product of homeomorphisms of both types, $A$ and $B$, and turbulent an homeomorphism that is conjugated to some positively turbulent homeomorphism. 

From a general point of view, I will use and connect two ways of considering homeomorphisms. First the cinematic point of view. In this point of view, a homeomorphism is defined as the deformation of another one during an isotopy. This isotopy is often classicaly given by $(f_{t})_{t\in [0;1]}$, with $f_{0}=\id$. Historically, the flow of differential equations gave birth to homeomorphisms of surfaces and this point of view is related with the suspension of the deformation that associates braids and braid types to the periodic orbits of homeomorphisms.
The second point of view is a static, graphic or geometric point of view. A homeomorphism is given by its direct action on points, curves or geometric objects. Of course, the iteration of a homeomorphism still induces dynamics and the cinematic and static points of view are closely related. Both are in fact necessary to understand the behaviour of homeomorphisms and the appearance of fixed points and periodic orbits.

Let me first give a dynamical definition of a top displacement to the left of the middle point and explain how this definition can be generally understood.

\begin{defi}
A top-displacement to the left of the middle point of $X$ relatively to $\Delta^{+}$, or simply a homeomorphism of the type $A$ relatively to $X$ and $\Delta^{+}$, is a homeomorphism $f_{A}$ of the plane $\R^{2}$ such that 
$$f_{A}(p_{1})=p_{2},~~~~ f_{A}(p_{2})=p_{1},~~~~ f(p_{3})=p_{3}$$

and such that there exists an isotopy $(f_{t})_{t\in [0;1]}$ with $f_{0} = \id$, $f_{1} = f_{A}$ and for any $t\in ]0;1[$, 
$$f_{t}(p_{1}) \in \Delta^{+},~~~~~f_{t}(p_{3}) \in \Delta^{+},~~~~\mathrm{and}~~ f_{t}(p_{2}) \notin \H^{-}\, .$$

Similarily, a top-displacement to the right of the middle point of $X$ relatively to $\Delta^{+}$, or simply a homeomorphism of the type $B$ relatively to $X$ and $\Delta^{+}$, is a homeomorphism $f_{B}$ of the plane $\R^{2}$ such that 
$$f_{B}(p_{1})=p_{1},~~~~ f_{B}(p_{2})=p_{3},~~~~ f_{B}(p_{3})=p_{2}$$
and
there exists an isotopy $(f_{t})_{t\in [0;1]}$ such that $f_{0} = \id$, $f_{1} = f$ and for any $t\in ]0;1[$, $f_{t}(p_{1}) \in \Delta^{+}$, $f_{t}(p_{3}) \in \Delta^{+}$ and $f_{t}(p_{2}) \notin \H^{-}$.
\label{definition1}
\end{defi}

\begin{figure}[htpb]

\begin{picture}(45,50)(-110,-15)

\color{black}
\put(-80,02){\makebox(0,0)[bc]{$\Delta$}}
\put(75,02){\makebox(0,0)[bc]{$\H^{+}$}}
\put(-30,0){\circle*{1}}
\put(0,0){\circle*{1}}
\put(0,0){\circle*{1}}
\put(30,0){\circle*{1}}

\put(-30,-2){\makebox(0,0)[tc]{$p_{1}$}}
\put(-0,-2){\makebox(0,0)[tc]{$p_{2}$}}
\put(30,-2){\makebox(0,0)[tc]{$p_{3}$}}

\put(-80, 0){\line(1,0){160}}
\put(80,0){\vector(1,0){0}}

\put(-15,12){\makebox(0,0)[bc]{$f_{A}$}}

\color{mygreen}
\put(0, 0){\qbezier[300](-2,3)(-15,15)(-28,2)}
\put(-28, 2){\vector(-1,-1){0}}
\put(-5, 0){\vector(1,0){0}}
\put(-15, 0){\vector(1,0){0}}
\put(-25, 0){\vector(1,0){0}}

\end{picture}

\caption{A top-displacement to the left or type $A$ homeomorphism relatively to $X= \Big\{ p_{0}; p_{1}; p_{2} \Big\}$ on $\Delta$.\label{figure2}}
\end{figure}

\begin{figure}[htpb]
\begin{picture}(45,50)(-110,-15)

\color{black}

\put(-80,02){\makebox(0,0)[bc]{$\Delta$}}

\put(-30,0){\circle*{1}}
\put(0,0){\circle*{1}}
\put(0,0){\circle*{1}}
\put(30,0){\circle*{1}}

\put(-80, 0){\line(1,0){160}}
\put(80,0){\vector(1,0){0}}

\put(-30,-2){\makebox(0,0)[tc]{$p_{1}$}}
\put(-0,-2){\makebox(0,0)[tc]{$p_{2}$}}
\put(30,-2){\makebox(0,0)[tc]{$p_{3}$}}

\put(80,02){\makebox(0,0)[bc]{$\H^{+}$}}

\put(15,12){\makebox(0,0)[bc]{$f_{B}$}}

\color{mygreen}
\put(0, 0){\qbezier[300](2,3)(15,15)(28,2)}
\put(28, 2){\vector(1,-1){0}}
\put(5, 0){\vector(-1,0){0}}
\put(15, 0){\vector(-1,0){0}}
\put(25, 0){\vector(-1,0){0}}

\end{picture}

\caption{A top-displacement to the right or type $B$ homeomorphism relatively to $X= \Big\{ p_{0}; p_{1}; p_{2} \Big\}$ on $\Delta$.\label{figure3}}
\end{figure}

\begin{remarks}

\item The notions of a homeomorphism of the type $A$ and of the type $B$ play a symetric role and the general remarks relative to the homeomorphisms of the type $A$ are also true for the homeomorphisms of the type $B$ in the corresponding symetric formulation.
\item The set of homeomorphisms $f$ of the type $A$ relative to $\Delta^{+}$ and $X\subset \Delta$ form a unique isotopy class of homeomorphisms relative to $X$. Suppose that $f$ is of the type $A$ relatively to $X$ on $\Delta^{+}$ and that $g$ is isotopic to $f$ relatively to $X$. There is an isotopy $(\varphi_{t})_{t\in [0;1]}$ such that $\varphi_{0} = f$ and $\varphi_{1} = g$ and $\varphi_{t}(X) = X$ for all $t\in [0;1]$. And if  $(f_{t})_{t\in [0;1]}$ is an isotopy with the required conditions (see definition \ref{definition1}), one can start from $f_{0} = \id$, follow $f_{2t}$ for $t\in [0;\frac{1}{2}]$ to $f_{1} = f$, and then $(\varphi_{2t-1})$ for $t\in [\frac{1}{2};1]$ to $g$ to define an isotopy with the required properties. One can therefore always consider a particular homeomorphism of the type $A$, and use the fact that all the other homeomorphisms of the type $A$ constitute its homotopy class relatively to $X$.

\item Once the condition to exchange $p_{1}$ and $p_{2}$ and to fix $p_{3}$ is given, one can also compose an isotopy $(f_{t})_{t\in [0;1]}$ from $f_{0} = \id$ to $f_{1} = f$ with a unique isotopy of affine homeomorphisms to satisfy (see figure \ref{figure2}) the conditions  
$$f_{t}(p_{1}) = p_{1} + t (p_{2} - p_{1})~~\mathrm{and}~~ f_{t}(p_{3})= p_{3}\, .$$ 

The condition to be of the type $A$ is then equivalent to the fact that the continuous determination 
$\theta(t)$ of the argument of $\displaystyle \frac{f_{t}(p_{2}) - f_{t}(p_{1})}{f_{t}(p_{3}) - f_{t}(p_{1})}$ such that $\theta(0) = 0$ also satisfy $\theta(1) = \pi$. Once one simply set $p_{2} \in [p_{1}: p_{3}] \subset \Delta$, one can always consider this general condition as a criterium for a homeomorphism to be of the type $A$ relatively to a set $X=\big\{ p_{1}; p_{2}; p_{3} \big\}$ on a line $\Delta^{+}$ that contains $X$. 

\item Consider for example the homeomorphism $f$ defined in complex coordinates by  
$$f(z) = z \times e^{\frac{i \pi}{3} \times (4 - \vert z \vert^{2})}\, , $$
and the points $p_{1} = -1$; $p_{2} = 1$ and $p_{3} = 2$. 

Check that one has $f(p_{1}) = p_{2}$, $f(p_{2})= p_{1}$ and $f(p_{3})= p_{3}$ and that during the isotopy defined for $t\in [0;1]$ by 
$$f_{t}(z) = z \times e^{\frac{i \pi}{3} \times t \times (4-  \vert z \vert^{2})}\, ,$$

 the argument of the complex number $\displaystyle \frac{f_{t}(p_{2}) - f_{t}(p_{1})}{f_{t}(p_{3}) - f_{t}(p_{1})}$ is constantly increasing from $0$ to $\pi$. This homeomorphism is of the type $A$ relative to $p_{1}$; $p_{2}$; $p_{3}$.
 
 \item  If generally a homeomorphism $f_{\varphi}$ is of the form~:
 $$f_{\varphi}(z) = z \times e^{i \varphi(\vert z \vert)}\, , $$
then I call this homeomorphism a {\texteit rotary homeomorphism} around the center $O$. 

Such an homeomorphism $f_{\varphi}$ is of the type $A$ relatively to $p_{1} = -1$; $p_{2} = 1$ and $p_{3} = 3$ if and only if $\varphi(1)=(2k+1) \pi$ and $\varphi(2) = 2k \pi$ for some $k\in \Z$.  The homeomorphisms that preserve the distance to a given point $C$ are called the {\texteit rotary homeomorphisms around $C$}. 
\item It is a good idea to consider the case of points $p_{1}$; $p_{2}$, $p_{3}$ such that in complex coordinates $p_{1} = -2$; $p_{2} = 0$ and $p_{3} = 2$, and rotary homeomophisms around the points $I = -1$ and $J=1$. If $S$ is the axial symetry about the line $x=0$, defined by $S(x; y) = (-x;y)$, then for any homeomorphism $f$ such that $f(X) =X$, the conjugated homeomorphism $g =S \circ f \circ S^{-1} = S \circ f \circ S$ also satisfies $g(X)=X$. Furthermore, $f$ is of the type $A$ if and only if $g$ is of the type $B$. If $\varphi$ is a continuous function such that $\varphi(1) = \pi$ and $\varphi(2) = 0$, then setting 
$$L_{\varphi}(z) = -1 + (z+1) e^{i\varphi(\vert z+1 \vert )}\, ,$$
one defines a type $A$ homeomorphism relatively to $X$. Similarily,
$$R_{\varphi}(z) = 1 + (z-1) e^{-i\varphi(\vert z-1 \vert )}\, , $$
defines a homeomorphism of the type $B$ relatively to $X$. Those homeomorphisms give the simplest examples of turbulent homeomorphisms, and some very nice geometric properties allow to describe their fixed points and periodic orbits (see the following section to be devoted to this case). 
\begin{figure}[htpb]

\begin{picture}(45,50)(-110,-20)

\color{black}
\put(-80,02){\makebox(0,0)[bc]{$\Delta$}}
\put(75,02){\makebox(0,0)[bc]{$\H^{+}$}}
\put(-30,0){\circle*{1}}
\put(0,0){\circle*{1}}
\put(0,0){\circle*{1}}
\put(30,0){\circle*{1}}

\put(-30,2){\makebox(0,0)[bc]{$p_{1}$}}
\put(0,2){\makebox(0,0)[bc]{$p_{2}$}}
\put(30,-2){\makebox(0,0)[tc]{$p_{3}$}}

\put(-80, 0){\line(1,0){160}}
\put(80,0){\vector(1,0){0}}
\put(-15,-12){\makebox(0,0)[tc]{$f_{A}^{-1}$}}

\color{mygreen}
\put(0, 0){\qbezier[300](-2,-3)(-15,-15)(-28,-2)}
\put(-28, - 2){\vector(-1,1){0}}
\put(-5, 0){\vector(1,0){0}}
\put(-15, 0){\vector(1,0){0}}
\put(-25, 0){\vector(1,0){0}}

\end{picture}

\caption{The inverse of a homeomorphism of the type $A$  relatively to $X= \Big\{ p_{0}; p_{1}; p_{2} \Big\}$ on $\Delta$.\label{figure4}}
\end{figure}

\item The inverse of a homeomorphism of the type $A$ is a bottom displacement of the middle point to the left (see figure \ref{figure4}). To understand this fact, one can for example consider an isotopy $(f_{t})_{t\in[0;1]}$ from $f_{0} = \id$ to $f_{1} = f$ with the properties required by the definition \ref{definition1} and set 
$$h_{t} = f_{1-t} \circ f^{-1}~~~~\mathrm{for}~~~~t\in [0;1]\, .$$ 
One then has $h_{0} = \id$, $h_{1} = f^{-1}$, $h_{t}(p_{1}) = f_{1-t}(p_{2}) \notin \H^{-}$, $h_{t}(p_{2}) = f_{1-t}(p_{1}) \in \Delta^{+}$ and $h_{t}(p_{3}) = f_{1-t}(p_{3}) \in \Delta^{+}$. Furthermore, if one composes $h_{t}$ with the isotopy $g_{t}$ of affine homeomorphisms such that $g_{t}(h_{t}(p_{1})) = p_{1} + t \ov{p_{1}p_{2}}$ and $g_{t}(h_{t}(p_{3}))=p_{3}$, necessarily $g_{t}\circ h_{t}(p_{2}) \notin \H^{+}$. Similarily, an homeomorphism $f$ such that there exists an isotopy $(f_{t})_{t\in [0;1]}$ with $f_{0} = \id$ and $f_{1}=f$, $f(p_{1}) = p_{2}$, $f(p_{2}) = p_{1}$, $f(p_{3})= p_{3}$; $f_{t}(p_{2}) \in \Delta^{+}$, $f_{t}(p_{3}) \in \Delta^{+}$ and $f_{t}(p_{1}) \notin \H^{+}$ is of the type $A$.    

\item  One can also give a graphical definition of a homeomorphism $f$ of the type $A$ by considering its action on homotopy classes of curves in $\R^{2} \setminus X$. In fact, they are many different ways to do so. From a theorical point of view, a curve in $\R^{2}\setminus X$ is a continuous image of an  interval $\I$ or of the circle $\S^{1}= \R / \Z$. One generally considers homotopy classes or isotopy classes of curves, and one can use the space $\S^{2} \setminus X = \overline{\C} \setminus X$ instead of $\R^{2} \setminus X$. If $\I$ is an open interval the considered curves must have limits in $X$ or $\infty \in \overline{\C}$. One can also join those curves together to define multicurves, glue them and associate them together using limit processes, identify the different classes of positive and negative parametrizations. Or consider the homotopy group of $\overline{\C} \setminus X$ with  $\infty$ as a base point. Note that the three points of view will be useful in this paper.

For example, if one considers properly embeded continuous lines in $\R^{2}$. Let $\Delta_{1}$ and $\Delta_{2}$ the respective homotopy classes of the mediatrices of the segments $[p_{1};p_{2}]$ and $[p_{2}; p_{3}]$. A homeomorphism $f$ is isotopic to $\id$ in $\R^{2} \setminus X$ if and only if $f(\Delta_{1}) \equiv \Delta_{1}$ and $f(\Delta_{2}) \equiv \Delta_{2}$. If $f$ is a homeomorphism of the type $A$ then $f(\Delta_{2})$ is homotopic to  $\Delta_{2}$ and $f(\Delta_{1})$ is homotopic to $\Delta_{1} \cup \Delta_{2}$. But this fact is also true of $f^{-1}$ and it is not an equivalence. To understand this situation (see figure \ref{figure5} and \ref{figure6}), on can imagine that the limit muticurve defined by $\Delta_{1} + \Delta_{2}$ is on the frontier of two different regions, one containing $f(\Delta_{1})$ and the other $f^{-1}(\Delta_{1})$. 

\begin{figure}[htpb]

\begin{picture}(45,110)(-110,-55)

\color{black}

\put(-40,0){\circle*{1}}
\put(0,0){\circle*{1}}
\put(40,0){\circle*{1}}

\put(-40,-2){\makebox(0,0)[tc]{$ p_{1}$}}
\put(0,-2){\makebox(0,0)[tc]{$p_{2}$}}
\put(40,-2){\makebox(0,0)[tc]{$p_{3}$}}

\put(-80,2){\makebox(0,0)[bc]{$\Delta$}}
\put(-80, 0){\line(1,0){160}}
\put(80,0){\vector(1,0){0}}

\color{mygreen}

\put(-20,-50){\line(0,1){100}}
\put(-22,50){\makebox(0,0)[tr]{$\Delta_{1}$}}
\put(20,-50){\line(0,1){100}}
\put(18,50){\makebox(0,0)[tr]{$\Delta_{2}$}}

\put(-40,0){\qbezier[100](-10,0)(-10,4.14)(-7.07,7.07)}
\put(-40,0){\qbezier[100](-7.07,7.07)(-4.14,10)(0,10)}
\put(-40,0){\qbezier[100](0,10)(4.14,10)(7.07,7.07)}
\put(-40,0){\qbezier[100](7.07,7.07)(10,4.14)(10,0)}

\put(-50,-50){\line(0,1){50}}
\put(-30,-50){\line(0,1){50}}
\put(-52,-50){\makebox(0,0)[tr]{$\Delta_{1}$}}

\put(40,0){\qbezier[100](-10,0)(-10,4.14)(-7.07,7.07)}
\put(40,0){\qbezier[100](-7.07,7.07)(-4.14,10)(0,10)}
\put(40,0){\qbezier[100](0,10)(4.14,10)(7.07,7.07)}
\put(40,0){\qbezier[100](7.07,7.07)(10,4.14)(10,0)}

\put(50,-50){\line(0,1){50}}
\put(52,-50){\makebox(0,0)[tr]{$\Delta_{2}$}}

\put(30,-50){\line(0,1){50}}
\put(18,50){\makebox(0,0)[tr]{$\Delta_{2}$}}

\put(20,-50){\line(0,1){100}}

\color{red}
\put(0,0){\qbezier[100](-10,0)(-10,4.14)(-7.07,7.07)}
\put(0,0){\qbezier[100](-7.07,7.07)(-4.14,10)(0,10)}
\put(0,0){\qbezier[100](0,10)(4.14,10)(7.07,7.07)}
\put(0,0){\qbezier[100](7.07,7.07)(10,4.14)(10,0)}

\put(-10,-50){\line(0,1){50}}
\put(-15,-50){\line(0,1){75}}
\put(-18,-50){\line(0,1){100}}
\put(10,-50){\line(0,1){50}}
\put(15,-50){\line(0,1){75}}

\put(18,-50){\line(0,1){100}}

\put(11,-50){\makebox(0,0)[tl]{$f(\Delta_{1})$}}

\put(0,25){\qbezier[100](-15,0)(-15,6.21)(-10.67,10.67)}
\put(0,25){\qbezier[100](-10.67,10.67)(-6.21,15)(0,15)}
\put(0,25){\qbezier[100](0,15)(6.21,15)(10.67,10.67)}
\put(0,25){\qbezier[100](10.67,10.67)(15,6.21)(15,0)}

\put(0,45){{\vector(0,1){5}}}

\put(0,15){{\vector(0,1){5}}}

\color{black}
\put(0, 0){\qbezier[300](0,0)(-40,40)(-60,2)}
\put(-60, 2){\vector(-1,-1){0}}
\put(-40,25){\makebox(0,0)[bc]{$f$}}

\put(-40, 0){\qbezier[300](0,0)(20,-15)(37,-2)}
\put(-3, -2){\vector(2,1){0}}
\put(-25,-15){\makebox(0,0)[bc]{$f$}}

\color{mygreen}

\end{picture}
\caption{A homeomorphism of the type $A$ acts on $\Delta_{1}$ and $\Delta_{2}$ and in some sense $f(\Delta_{1}) = \Delta_{1} + \Delta_{2}$.  \label{figure5}}
\end{figure}

\begin{figure}[htpb]

\begin{picture}(45,110)(-110,-55)

\color{black}
\put(-40,0){\circle*{1}}
\put(0,0){\circle*{1}}
\put(40,0){\circle*{1}}
\put(-40,-2){\makebox(0,0)[tc]{$p_{1}$}}
\put(0,-2){\makebox(0,0)[tc]{$p_{2}$}}
\put(40,-2){\makebox(0,0)[tc]{$p_{3}$}}
\put(-80,2){\makebox(0,0)[bc]{$\Delta$}}
\put(-80, 0){\line(1,0){160}}
\put(80,0){\vector(1,0){0}}

\color{mygreen}
\put(-20,-50){\line(0,1){100}}
\put(-22,50){\makebox(0,0)[tr]{$ \Delta_{1}$}}
\put(20,-50){\line(0,1){100}}

\put(-40,0){\qbezier[100](-10,0)(-10,-4.14)(-7.07,-7.07)}
\put(-40,0){\qbezier[100](-7.07,-7.07)(-4.14,-10)(0,-10)}
\put(-40,0){\qbezier[100](0,-10)(4.14,-10)(7.07,-7.07)}
\put(-40,0){\qbezier[100](7.07,-7.07)(10,-4.14)(10,0)}

\put(-50,50){\line(0,-1){50}}
\put(-30,50){\line(0,-1){50}}
\put(-51,50){\makebox(0,0)[tr]{$\Delta_{1}$}}

\put(40,0){\qbezier[100](-10,0)(-10,4.14)(-7.07,7.07)}
\put(40,0){\qbezier[100](-7.07,7.07)(-4.14,10)(0,10)}
\put(40,0){\qbezier[100](0,10)(4.14,10)(7.07,7.07)}
\put(40,0){\qbezier[100](7.07,7.07)(10,4.14)(10,0)}
\put(50,-50){\line(0,1){50}}
\put(51,-50){\makebox(0,0)[tr]{$\Delta_{2}$}}
\put(30,-50){\line(0,1){50}}
\put(21,50){\makebox(0,0)[tl]{$\Delta_{2}$}}
\put(20,-50){\line(0,1){100}}

\color{red}
\put(0,0){\qbezier[100](-10,0)(-10,-4.14)(-7.07,-7.07)}
\put(0,0){\qbezier[100](-7.07,-7.07)(-4.14,-10)(0,-10)}
\put(0,0){\qbezier[100](0,-10)(4.14,-10)(7.07,-7.07)}
\put(0,0){\qbezier[100](7.07,-7.07)(10,-4.14)(10,0)}

\put(-10,50){\line(0,-1){50}}
\put(-15,50){\line(0,-1){75}}
\put(-18,50){\line(0,-1){100}}
\put(10,50){\line(0,-1){50}}
\put(15,50){\line(0,-1){75}}

\put(18,50){\line(0,-1){100}}

\put(12,50){\makebox(0,0)[bl]{$f^{-1}(\Delta_{1})$}}

\put(0,-25){\qbezier[100](-15,0)(-15,-6.21)(-10.67,-10.67)}
\put(0,-25){\qbezier[100](-10.67,-10.67)(-6.21,-15)(0,-15)}
\put(0,-25){\qbezier[100](0,-15)(6.21,-15)(10.67,-10.67)}
\put(0,-25){\qbezier[100](10.67,-10.67)(15,-6.21)(15,0)}

\put(0,-45){{\vector(0,-1){5}}}

\put(0,-15){{\vector(0,-1){5}}}

\color{black}
\put(0, 0){\qbezier[300](0,0)(-40,-40)(-60,-2)}
\put(-60, -2){\vector(-1,1){0}}
\put(-40,-23){\makebox(0,0)[tc]{$f^{-1}$}}

\put(-40, 0){\qbezier[300](0,0)(20,15)(37,2)}
\put(-3, 2){\vector(2,-1){0}}
\put(-25,10){\makebox(0,0)[bc]{$f^{-1}$}}

\end{picture}
\caption{A homeomorphism of the type $A$ acts on $\Delta_{1}$ and $\Delta_{2}$ and in some sense $f^{-1}(\Delta_{1}) = \Delta_{1} + \Delta_{2}$.  \label{figure6}}
\end{figure}

On can also look at classes of curves drawn in $\R^{2} \setminus X$, with an end in $X$ and the other at $\infty$. Let $H_{1}^{-}$, $H_{2}^{-}$ and $H_{3}^{-}$ be the respective homotopy classes of the vertical half-lines of respective origins $p_{1}$, $p_{2}$ and $p_{3}$ and direction $\ov{v}(-1;0)$. Similarily let $H_{1}^{+}\equiv \sigma(H_{1}^{-})$,$H_{2}^{+}= \sigma(H_{2}^{-})$ and $H_{3}^{+}= \sigma(H_{3}^{-})$. Since $H_{1}^{-}\equiv H_{1}^{+}$ one can set $H_{1} \equiv H_{1}^{-} \equiv H_{1}^{+}$ and similarily $H_{3} \equiv H_{3}^{-} \equiv H_{3}^{+}$. But $H_{2}^{-} \neq H_{2}^{+}$ and this difference is essential. A homeomorphism is of the type $A$ if and only if $f(H_{3}) \equiv H_{3}$ and $f(H_{1}) \equiv H_{2}^{-}$ or equivalently if and only if $f(H_{3}) \equiv H_{3}$  and $f(H_{2}^{+}) \equiv H_{1}$. And one has symetric results for the homeomorphisms of the type $B$ (see figure \ref{figure7}). 
 
\begin{figure}[htpb]

\begin{picture}(45,110)(-110,-55)

\color{black}
\put(-40,0){\circle*{1}}
\put(0,0){\circle*{1}}
\put(40,0){\circle*{1}}
\put(-40,2){\makebox(0,0)[bc]{$p_{1}$}}
\put(0,-2){\makebox(0,0)[tc]{$p_{2}$}}
\put(40,2){\makebox(0,0)[bc]{$p_{3}$}}

\color{mygreen}

\put(-40,-50){\line(0,1){50}}
\put(-42,-50){\makebox(0,0)[br]{$H_{1}$}}


\put(-0,0){\line(0,1){50}}
\put(-2,50){\makebox(0,0)[tr]{$H_{2}^{+}$}}

\put(40,-50){\line(0,1){50}}
\put(38,-50){\makebox(0,0)[br]{$H_{3}$}}

\color{black}

\put(-85,0){\makebox(0,0)[cr]{$f_{A}$}}

\put(-80, 0){\vector(0,-1){0}}

\put(0, 0){\qbezier[300](0,30)(-80,50)(-80,0)}
\put(0, 0){\qbezier[300](-80,0)(-80,-40)(-40,-40)}
\put(-40, -40){\vector(1,0){0}}

\put(-20,47){\makebox(0,0)[bc]{$f_{A}$}}

\end{picture}
\caption{An homeomorphism of type $A$ acting on $H_{1}$, $H_{2}^{-}$, $H_{2}^{+}$ and $H_{3}$.  \label{figure7}}
\end{figure}

The third point of view is to consider homotopy classes of arcs relating points in $X$ : $I_{1}$ is the homotopy class relatively to $X$ of the segment $[p_{1}; p_{2}]$, $I_{2}$ the homotopy class relative to $X$ of $[p_{2}; p_{3}]$, $C^{+}$ the homotopy class relative to $X$ of the half circle whose diameter is $[p_{1}; p_{3}]$ and that is contained in $\H^{+}$, and $C^{-} \equiv \sigma(C^{+})$. Then an homeomorphism $f$ that preserves the orientation is of the type $A$ if and only if $f(I_{1}) \equiv I_{1}$ and $f(I_{2}) \equiv C^{+}$ (see figure \ref{figure9}). 

\begin{figure}[htpb]

\begin{picture}(45,90)(-110,-10)

\color{black}
\put(-40,70){\circle*{1}}
\put(0,70){\circle*{1}}
\put(40,70){\circle*{1}}
\put(-40,68){\makebox(0,0)[tc]{$p_{1}$}}
\put(0,68){\makebox(0,0)[tc]{$p_{2}$}}
\put(40,68){\makebox(0,0)[tc]{$p_{3}$}}

\put(-40,0){\circle*{1}}
\put(0,0){\circle*{1}}
\put(40,0){\circle*{1}}
\put(-40,-2){\makebox(0,0)[tc]{$p_{1}$}}
\put(0,-2){\makebox(0,0)[tc]{$p_{2}$}}
\put(40,-2){\makebox(0,0)[tc]{$p_{3}$}}

\color{mygreen}
\put(-40, 70){\line(1,0){40}}
\color{black}
\put(-20, 70){\vector(1,0){0}}
\put(20, 70){\vector(1,0){0}}

\put(-20,72){\makebox(0,0)[bc]{$I_{1}$}}
\put(20,72){\makebox(0,0)[bc]{$I_{2}$}}

\put(-20,2){\makebox(0,0)[bc]{$f(I_{1})$}}
\put(0, 42){\makebox(0,0)[bc]{$f(I_{2})$}}

\put(-50,0){\qbezier[300](0,60)(-30,35)(0,10)}

\put(-50,10){\vector(1,-1){0}}

\put(-82,35){\makebox(0,0)[bc]{$f_{A}$}}

\color{mygreen}
\put(-40, 0){\line(1,0){40}}
\color{black}
\put(-20, 0){\vector(-1,0){0}}
\put(0, 40){\vector(1,0){0}}

\color{red}
\put(0, 70){\line(1,0){40}}

\put(0,0){\qbezier[300](-40,0)(-40,16.56)(-28.28,28.28)}
\put(0,0){\qbezier[300](-28.28,28.28)(-16.56,40)(0,40)}
\put(0,0){\qbezier[300](0,40)(16.56,40)(28.28,28.28)}
\put(0,0){\qbezier[300](28.28,28.28)(40,16.56)(40,0)}

\end{picture}
\caption{An homeomorphism $f_{A}$ of type $A$ acting on $I_{1}$ and $I_{2}$.  \label{figure8}}
\end{figure}

\end{remarks}

\section{The spinning skeleton of a curve.\label{redressement}}

The theory of the topological snail can be founded on a general description of a curve relatively to a finite linear set $X$. This explicit process produces, from a given curve $\Gamma$,  a new simplified curve called its spinning skeleton $\Psi(\Gamma)$, obtained by replacing parts of the initial curve with possibly very complicated behaviour, by segments or arcs of circle, keeping the hotomopy type with fixed ends relative to $X$, to construct a final curve by the junction of those segments and arcs of circle, isotopic to the initial curve with fixed ends relatively to $X$.

The simplest non trivial spinning skeleton appears as a topological snail (see sections \ref{pictures}, \ref{simple}, \ref{sectiondouble}) in the case of a set $X$ with three points. It is is important for it allows one to locate the fixed points of a homeomorphism $f$ through the use of an isotopic one $F$. More explicitely, given a homeomorphism $f$, a segment $I$ with a set $X\subset I$, one considers the curve $\Gamma= f(I)$, its spinning skeleton $\Psi(\Gamma)$, and a homeomorphism $F$, isotopic to $f$ relatively to $X$ such that $F(I) = \Psi(\Gamma)$. Many fixed points of $F$ appear in the general case, and one obtain them by looking for some Smale's geometric configurations defined by the spinning skeleton $\Psi(\Gamma)$. One then shows that those fixed points belong to topologicaly stable or persistent Nielsen classes and that therefore they also belong to $f$ (their indicies are different from zero). A direct proof, reasoning on sub-arcs of $\Gamma$ associated to the elements of $\Psi(\Gamma)$, is also possible. One makes use of the an important topological fact, the proof of wich is out of the scope of the present paper. 

\begin{theo} \label{topo}
Let $X = \lbrace p_{0}; p_{1}; \cdots ; p_{n} \rbrace \subset \R^{2}$ be a subset of $n$ points on two curves $\Gamma_{1}= [a;b]_{\Gamma_{1}}$ and $\Gamma_{2}= [a;b]_{\Gamma_{2}}$ with for $i=1$ and $i=2$:
$$a= p_{0} <_{\Gamma_{i}}< p_{1} <_{\Gamma_{i}} \cdots <_{\Gamma_{i}} p_{n} =b$$

If for each $i\in \lbrace 1; \cdots ; n \rbrace$ the subarcs $\Gamma^{i}_{1}= [p_{i-1};p_{i}]_{\Gamma_{1}}$ and $\Gamma^{i}_{2}=[p_{i-1};p_{i}]_{\Gamma_{2}}$ are homotopic relatively to $X$, then $\Gamma_{1}$ and $\Gamma_{2}$ are isotopic relatively to $X$. Similarily, if $f$ and $g$ are homeomorphisms such that $\Gamma_{2} = f(\Gamma_{1})$, then $f$ is isotopic to the identity relatively to $X$, and if $X$ is an invariant set of a homeomorphism $f$ such that for some curve $\Gamma$ on has that  $f(\Gamma)= \Gamma_{1}$, then there exists a homeomorphism $F$ in the isotopy class of $f$ relative to $X$ such that $F(\Gamma) = \Gamma_{2}$.
\end{theo}

Let us now define the spinning skeleton $\Psi(\Gamma)$ of a curve $\Gamma$. Consider $a, b \in \Delta$, with a point $\gamma(t)$ moving continuously and injectively for $t\in [0;1]$, from the starting point $a=\gamma(0)$ to the ending point $b=\gamma(1)$, through the running points $\gamma(t)$ for $t\in ]0;1[$ and describing the simple curve 
$$\Gamma = \Big\{ \gamma(t)\in \R^{2};~~t\in [0;1] \Big\}\,.$$ 
For any $m \in \Gamma$, $m'\in \Gamma$, write $[m;m']_{\Gamma} \subset \Gamma$ for the subarc of $\Gamma$ between $m$ and $m'$. Also write $m<_{\Gamma} m'$ if the respective parameters $t$ and $t'$ of the points $m$ and $m'$ satisfy $t<t'$. Similarily, write $[m;m']_{\Delta}$ for the segment between $m$ and $m'$ on the line $\Delta$. Since $\gamma$ is uniformly continuous on $[0;1]$, for any $\varepsilon>0$ there is a positive $\delta>0$ such that if $m,m' \in \Gamma$ and $mm'<\delta$ then the diameter of the subarc $[m;m']_{\Gamma}$ is strictly less than $\varepsilon$.

If $[m;n]_{\Gamma} \cap \Delta = \big\{ m; n \big\}$, say that $[m;n]_{\Gamma}$ is a doubly branched sub-arc of $\Gamma$ on $\Delta$ between $m$ and $n$. In this case, define its side $S$ to be the positive $S=+1$ if the running points of $\Gamma$ all belong to $\H^{+}$, and else the negative $S=-1$, when $]m;n[_{\Gamma} \subset \H^{-}$. If $m <_{\Gamma} n$, call $m$ and $n$ respectively a leaving and a landing point of $\Gamma$ on the side $S$. In a similar manner, denote $[m;n]^{S=+1}_{\mathcal{C}}$ to be the half circle from $m$ to $n$ on the diameter $[m;n]$ inside $\H^{+}$, and $[m;n]^{S=-1}_{C}$  the symetric half-circle $\sigma([m;n]^{S=+1}_{\mathcal{C}})$. If $[m;m']_{\Gamma}$ is a doubly branched sub-arc of $\Gamma$ on $\Delta$ between $m$ and $n$ and if $]m;m'[_{\Delta} \cap X \neq \emptyset$ call $[m;m']_{\Gamma}$ a separating link or sub-arc of $\Gamma$ relatively to $X$ between $m$ and $m'$ outside $\Delta$. 

Let $X\subset \Delta$ be a finite subset of $\Delta$ such that $a\in X$ and $b\in X$ and $X$ contains at least three points. For any $m\in \Delta\cap \Gamma$, define $n_{X}(m)$ to be the next point of $X$ after $m$ on $\Gamma$~: $n_{X}(m)$ is the only point $k \in X$ such that $]m;k]_{\Gamma} \cap X = \Big\lbrace k \big\}$. For any $m\in \Delta$, let also $\Omega_{X}(m)$ be the set of points of $\Delta$ that are not separeted by $X$ in $\Delta$~:
$$\Omega_{X}(m) =  \big\{ m'\in \Delta~~|~~]m;m'[_{\Delta} \cap X = \emptyset \big\}\,.$$

If $m\in \Delta \setminus X$, then either $\Omega_{X}(m)$ is the closure of the connected component of $m$ in $\R^{2} \setminus X$, or else $m\in X$ and $\Omega_{X}(m)$ is the closure of the union of its two adjacent components in $\R^{2} \setminus X$.

Suppose that $m\in \Gamma \cap \Delta$ and let $H_{m}$ be the set of points $k\in [m; n_{X}(m)]_{\Gamma}$ such that $[m;k]_{\Gamma}$ is homotopic with fixed end to $[m;k]_{\Delta}$. 

If $n_{X}(m) \in H_{m}$, then $[m; n_{X}(m)]_{\Gamma}$ is homotopic with fixed ends to $[m;n_{X}(m)]_{\Delta}$ in $\R^{2} \setminus X$, and I define $\Psi_{\Gamma}(m)=n_{X}(m)$.

Else, since $\gamma$ is uniformly continuous, there is a $\delta>0$ such that for any point $l\in X$ and $m\in \Gamma \cap X$ such that $ml < \delta$, the diameter of the sub-arc $[m;l]_{\Gamma}$ is strictly less that the mutual distances of points in $X$, and in this case $[m;l]_{\Gamma}$ is necessarily homotopic with fixed ends to the segment $[m;l]_{\Delta}$. If $n_{X}(m) \notin H_{m}$ therefore the distance of a point of $H_{m}$ to $X$ must be at least $\delta$. By uniform continuity there is also a $\delta'>0$ such that if $mn< \delta'$ then the diameter of $[m;n]_{\Gamma}$ is less than $\frac{\delta}{2}$ : in this case, if $m$ and $n$ are points of $H_{m}$, then $[m;n]_{\Gamma}$ is homotopic with fixed ends to $[m;n]_{\Delta}$ in $\R^{2} \setminus X$. Any convergent sequence of points of $H_{m}$ will therefore converge in $H_{m}$. In this case the set $\gamma^{-1}(H_{m})$ of parameters $t\in [0;1]$ such that $\gamma(t)\in H_{m}$ has a unique maximum $t'$, defining a point $\gamma(t')=m' \notin X$, that is nessarily the leaving point of a separating doubly branched sub-arc $[m'; p]_{\Gamma}^{S=\pm1}$, located on a given side $S=\pm1$ of $\Delta$, for some uniquely determined landing point $p$. 

Since $p\notin \Omega_{X}(m)$, the sub-arc $[m;m']_{\Gamma}$ is homotopic with fixed ends to $[m;m']_{\Delta}$ in $\R^{2} \setminus X$, and $[m';p]_{\Gamma}^{S=\pm1}$ is homotopic to the corresponding half-circle $[m';p]_{\mathcal{C}}^{S}$ in $\R^{2} \setminus \Delta$. The hole sub-arc $[m;p]_{\Gamma}$ is therefore homotopic in $\R^{2} \setminus X$ and with fixed ends to the half-circle $[m;p]_{\mathcal{C}}^{S}$, and the set $K_{H}$ of points $p\in [m; n_{X}(m)]_{\Gamma}$ such that $[m;p]_{\Gamma}$ is homotopic with fixed ends to a separating half-circle $[m;k]_{\mathcal{C}}^{S}$ on $\Delta$ is not empty. As before with the set $H_{m}$ and for the same reason, either $K_{M}$ contains $n_{X}(m)$ and one defines $\Psi_{\Gamma}(m) = n_{X}(m)$, or it is compact in $\R^{2} \setminus X$ and has a unique last element $m"$. In this case one defines $\Psi_{\Gamma}(m)=m"$, and obtains for $\Psi_{\Gamma}(m)$ a leaving point on a separating arc $[m"; P"]_{\Gamma}$, located on the opposite side $-S$ of $\Delta$ (otherwise the maximality of $m"$ wouldn't hold).

 Consider now the sequence of points defined by $m_{0} = a$ and $m_{i+1} = \Psi_{\Gamma}(m_{i})$ while $m_{i} \neq b$. For any $i\geqslant 0$ either $m_{i}\in X$ or $m_{i}$ is the leaving point of a separating link outside $\Delta$, at a distance uniformly bigger than $\Delta$ from $X$ (otherwise $m_{i}$ would be in $X$). This implies that the sequence of points $m_{i}$ is finished and that the preceeding process ends for some $i=N$ with $b= m_{N}$. One therefore obtains the following spinning skeleton theorem~:

\begin{theo} \label{spinningskeleton}
Let $\Gamma= [a;b]_{\Gamma}$ be a simple curve with both ends one a line $\Delta$ and successive points 
$$a= p_{0} <_{\Gamma}< p_{1} <_{\Gamma} \cdots <_{\Gamma} p_{n} =b$$
 on a finite subset $X\subset \Delta$. Then there is unique sequence of non zero natural numbers $(n_{j})_{1 \leqslant j\leqslant n}$,  with points $(m_{i}^{j})_{0\leqslant i \leqslant n_{j}}$ on $\Gamma \cap \Delta$ such that 

 $$p_{j-1}= m^{j}_{0} <_{\Gamma}< m^{j}_{1} <_{\Gamma} \cdots <_{\Gamma} m^{j}_{n_{j}} =p_{j} $$
and such that for any $i$ and $j$ such that $1 \leqslant i < n_{j}$ there is a defined half-circle $[m^{j}_{i}; m^{j}_{i+1}]_{\mathcal{C}}^{s_{i;j}}$ with diameter $[m^{j}_{i}; m^{j}_{i+1}]_{\Delta}$, on a given side $s_{i} = \pm1$ of the line $\Delta$, and such that if 
$$X'= X \bigcup \Big\{ m^{j}_{i} ,~~~1\leqslant j \leqslant n;~~~~~0\leqslant i \leqslant n_{j} \Big\}$$
 then $[m^{j}_{i}; m^{j}_{i+1}]_{\Gamma}$ and $[m^{j}_{i}; m^{j}_{i+1}]_{\mathcal{C}}^{s_{i;j}}$ are homotopic with fixed ends in $\R^{2} \setminus X'$. Futhermore the sides $s_{i}^{j}$ alternate when for some fixed $j$ the index $i$ describes $\big\{ 0; \cdots ; n_{j} \big\}$.


\begin{figure}
\begin{picture}(100, 90)(-120,-30)





\color{black}
\put(-125,0){\line(1,0){230}}
\put(105,0){\vector(1,0){0}}

\put(-115,-2){\makebox(0,0)[tr]{$\Delta$}}

\put(-80,0){\circle*{1}}
\put(-80,-2){\makebox(0,0)[tc]{$p_{0} = a$}}
\put(-40,0){\circle*{1}}
\put(-41,-2){\makebox(0,0)[tr]{$p_{1}$}}
\put(-20,0){\circle*{1}}
\put(-21,-2){\makebox(0,0)[tl]{$p_{2}$}}
\put(0,0){\circle*{1}}
\put(0,-2){\makebox(0,0)[tc]{$q_{1}$}}
\put(40,0){\circle*{1}}
\put(40,2){\makebox(0,0)[bc]{$p_{4} = b$}}
\put(80,0){\circle*{1}}
\put(80,-2){\makebox(0,0)[tc]{$q_{2}$}}
\put(100,0){\circle*{1}}
\put(100,2){\makebox(0,0)[bc]{$p_{3}$}}

\color{mygreen}

\put(-30,0){\qbezier[300](-50,0)(-50,20.7)(-35.35,35.35)}
\put(-30,0){\qbezier[300](-35.35,35.35),(-20.7,50),(0,50)}
\put(-30,0){\qbezier[300](0,50),(20.7,50)(35.35,35.35)}
\put(-30,0){\qbezier[300](35.35,35.35)(50,20.7)(50,0)}

\put(-65.35, 35.35){\vector(1,1){0}}

\put(20,0){\circle*{1}}

\put(21,-1){\makebox(0,0)[tl]{$m_{1}^{1}$}}

\put(-10,0){\qbezier[300](-30,0)(-30,-12.42)(-21.21,-21.21)}
\put(-10,0){\qbezier[300](-21.21,-21.21)(-12.42,-30)(0,-30)}
\put(-10,0){\qbezier[300](0,-30)(12.42,-30)(21.21,-21.21)}
\put(-10,0){\qbezier[300](21.21,-21.21)(30,-12.42)(30,0)}

\put(11.21,-21.21){\vector(-1,-1){0}}


\put(-40,0){\line(1,0){20}}
\put(-30,0){\vector(1,0){0}}


\put(-40,0){\qbezier[300](-20,0)(-20,8.28)(-14.14,14.14)}
\put(-40,0){\qbezier[300](-14.14,14.14)(-8.28,20)(0,20)}
\put(-40,0){\qbezier[300](0,20)(8.28,20)(14.14,14.14)}
\put(-40,0){\qbezier[300](14.14,14.14)(20,8.28)(20,0)}

\put(-54.14,14.14){\vector(-1,-1){0}}

\put(-60,0){\circle*{1}}

\put(-61,1){\makebox(0,0)[br]{$m_{1}^{3}$}}

\put(-80,0){\qbezier[300](-20,0)(-20,-8.28)(-14.14,-14.14)}
\put(-80,0){\qbezier[300](-14.14,-14.14)(-8.28,-20)(0,-20)}
\put(-80,0){\qbezier[300](0,-20)(8.28,-20)(14.14,-14.14)}
\put(-80,0){\qbezier[300](14.14,-14.14)(20,-7,64)(20,0)}

\put(-94.14,-14.14){\vector(-1,1){0}}

\put(-100,0){\circle*{1}}
\put(-101,1){\makebox(0,0)[br]{$m_{2}^{3}$}}


\put(-20,0){\qbezier[300](-80,0)(-80,33.12)(-56.56,56.56)}
\put(-20,0){\qbezier[300](-56.56,56.56)(-33.12,80)(0,80)}
\put(-20,0){\qbezier[300](0,80)(33.12,80)(56.56,56.56)}
\put(-20,0){\qbezier[300](56.56,56.56)(80,33.12)(80,0)}

\put(36.56,56.56){\vector(1,-1){0}}

\put(60,0){\circle*{1}}

\put(59,-1){\makebox(0,0)[tr]{$m_{3}^{3}$}}


\put(80,0){\qbezier[300](-20,0)(-20,-8.28)(-14.14,-14.14)}
\put(80,0){\qbezier[300](-14.14,-14.14)(-8.28,-20)(0,-20)}
\put(80,0){\qbezier[300](0,-20)(8.28,-20)(14.14,-14.14)}
\put(80,0){\qbezier[300](14.14,-14.14)(20,-7,64)(20,0)}

\put(94.14,-14.14){\vector(1,1){0}}

\put(70,0){\qbezier[300](-30,0)(-30,-12.42)(-21.21,-21.21)}
\put(70,0){\qbezier[300](-21.21,-21.21)(-12.42,-30)(0,-30)}
\put(70,0){\qbezier[300](0,-30)(12.42,-30)(21.21,-21.21)}
\put(70,0){\qbezier[300](21.21,-21.21)(30,-12.42)(30,0)}

\put(48.79,-21.21){\vector(-1,1){0}}

\color{black}

\color{blue}


\end{picture}
\caption{A typical spinning skeleton on the set $X= \Big\{ p_{0}; p_{1}; p_{2}; p_{3} ; p_{4}; q_{1}; q_{2} \Big\}$ on $\Delta$.\label{figure9}}

\end{figure}

\end{theo}
  
~\\
~
  
\begin{remarks}
\item One can also give a very natural interpretation of the function $\Psi_{\Gamma}$ in the universal cover of $\R^{2} \setminus X$. If $X$ contains $n\geqslant 2$ points, a realisation of the universal cover of $\R^{2}\setminus X$ can be classicaly constructed. Define  $X= \big\{ x_{1}; x_{2}; \cdots ; x_{n} \big\}$ to be the points of $X \cap \Delta$, suppose that $x_{1} < x_{2} < \cdots < x_{n}$ and define $x_{0}= x_{n+1} = \infty$. Using a conformal uniformization of the half plane on the left of $\Delta$ to the interior of a given regular $n+1$-gone $p$ with associated vertices in circular order $s_{0},s_{1}, \ldots, s_{n+1} = s_{0}$, with $s_{i} = e^{\frac{2i\pi}{n+1}}$, lift this half plane to the interior of $p$. Extend this lift to a homeomorphism between the universal cover of $\R^{2} \setminus X$ and the hole disc, using the respective hyperbolic symetries $\sigma_{i}$ around the geodesics $]s_{i}; s_{i+1}[$ to progressively glue together the successive images of $p$ obtained by the crossing of the intervals $]x_{i}; x_{i+1}[$ by a classe of curve $\gamma$. The covering group $G$ is the free group generated by the transformations of the form $g_{i; j} = \sigma_{i} \circ \sigma_{j}$. Any simple curve $\Gamma \subset \R^{2}$ lifts to a disjoint family of simple curves, obtained from a single one of them by taking its images by the elements of the covering group $G$. If the initial curve $\Gamma$ begins and ends on points of $X$, then one of those curves can be taken to begin in some of the vertices of  $p$: this curve is either homotopic to one of the geodesic $[s_{i}; s_{i+1}]$, or a geodesic of the form $[s_{i}; s_{k}]$ for $k \notin \big\{ i-1; i+1 \big\}$  or a geodesic to an image $g(s_{k})$ of $s_{k}$ by an element $g$ of $G$. This geodesic crosses unique images of the sides of $p$ by elements of $G$, with on each of them a corresponding last point of the lift of the curve $\Gamma$ that corresponds to an image of $\Psi(\Gamma)$.

\item The spinning skeleton $\Psi(\Gamma)$ doesn't depend continuously on the deformation of the curve $\Gamma$ but  in the class of curves that contain a subset of $X$ and are made of the junction of a finite collection of half circles doubly branched on $\Delta$ and avoiding the other points of $X$, the skeletons of two isotopic curves relatively to $X$ are always also isotopic. The unique representant of $\Psi(\Gamma)$ with uniformly positionned points $m_{i}^{j}$ in the intervals of $\R^{2} \setminus X$ (see figure \ref{figure9}) defines a unique representant of the hole class of curve isotopic to $\Gamma$ relatively to $X$. Being constant on isotopy classes, it also depends continuously on $\Gamma$.   

\item A simple curve $\Gamma$, obtained by joining a finite serie of segments on a line $\Delta$ and a finite serie of half-circles with their diameters on this line, and such that a finite set $X\subset \Delta$ contains points on each diameter of those circles, all the extremities of thoses segments, and no point inside them, then this curve is its proper skeleton in both directions, which means that one has $\Psi(\Gamma) = \Gamma$ (see figure \ref{figure9}).

\item On can apply this procedure to a simple closed curve $\Omega$, for example if it contains at least two points on $X$ (see theorem \ref{canonicalform} and its proof).

\end{remarks}

\section{The simplest non trivial spinning skeleton.\label{snail}}

In the case of an invariant set $X= \Big\{ p_{1}; p_{2} ; p_{3} \Big\}$ made of three points on the line $\Delta$, the spinning skeleton $\Psi(\Gamma)$ of a curve $\Gamma$ doubly branched on two points of $X$ and which does not contain the third, can be easily and completely described as a simple topological snail. 

First, observe that if $\Psi(\Gamma)$ contains a segment, then its extremities are those of $\Gamma$ and this segment is necessarily $[p_{1}; p_{2}]_{\Delta}$ or $[p_{2}; p_{3}]_{\Delta}$. The spinning skeleton $\Psi(\Gamma)$ is therefore reduced to this segment, and $\Gamma$ itself is isotopic, with fixed ends, either to $[p_{1}; p_{2}]_{\Delta}$ or to $[p_{2}; p_{3}]_{\Delta}$ in $\R^{2} \setminus X$. Else, $\Psi(\Gamma)$ can be a unique half circle $[p_{1} ; p_{3}]_{\mathcal{C}^{S}}^{S=\pm 1}$, or obtained by the junction of several half circles on successive opposite sides of $\Delta$ (see theorem \ref{spinningskeleton} in section \ref{redressement}).

Since $\Psi(\Gamma) \cap \Delta$ is a compact set, it has a minimum $m$ and a maximum $n$ on $\Delta$, caracterized by $m<_{\Delta} n$ and 
$$ \Psi(\Gamma) \cap \Delta \subset [m; n ]_{\Delta}\, .$$
If $p_{1} <_{\Delta} m$, then the hole skeleton $\Psi(\Gamma)$ is located inside the convex open set formed by the points of $\R^{2}$ whose abscissa is strictly greater than the abscissa $x_{1}$ of $p_{1}$ and therefore $\Gamma$ is homotopic to $[p_{2}; p_{3}]$, an excluded case. So necessarily one has $m \leqslant p_{1}$ and similarily $p_{3} \leqslant n$. Let us show that in this case, the part $[m; n]_{\Psi(\Gamma)}$ of the skeleton of $\Gamma$ between the points $m$ and $n$ is necessarily reduced to a unique half circle on a given side $S = \pm1$ of $\Delta$~:

\begin{prop} If $m, n \in \Psi(\Gamma)$ are such that $m<_{\Delta} n$ and $ \Psi(\Gamma) \cap \Delta \subset [m; n ]_{\Delta}$, then the part $[m; n]_{\Psi(\Gamma)}$ of the spinning skeleton of $\Gamma$ between the points $m$ and $n$ is necessarily reduced to a unique half circle on a given side $S(\Gamma) = \pm1$ of $\Delta$. \label{proposition1}
\end{prop}
~

\begin{proof} If $m=p_{1}$ and $n=p_{3}$, then $\Psi(\Gamma) = [p_{1}; p_{3}]_{\mathcal{C}}^{S}$ is reduced to a half-circle and the proposition is statisfied.

Else let us first suppose that $m=p_{1}$ and consider $[m; p]_{\mathcal{C}}^{S}$ the half-circle of $\Psi(\Gamma)$ that begins with $p_{1}=m$. If $p\neq n$, then $p\in ]p_{3}; +\infty[$ is impossible for in this case necessarily $n$ is a running point of $\Psi(\Gamma)$, it is therefore a junction point for two half-circles of $\Psi(\Gamma)$ one of wich, say $[n;k]_{\mathcal{C}}^{S}$, is on the same side of $\Delta$ as $[m; p]_{\mathcal{C}}^{S}$, and therefore necessarily $k<_{\Delta} m$ which contradicts the definition of $m$ (see figure \ref{figure10}).

\begin{figure}[htpb]

\begin{picture}(45,85)(-110,-15)

\color{black}

\put(-40,0){\circle*{1}}
\put(0,0){\circle*{1}}
\put(-80,2){\makebox(0,0)[bc]{$ \Delta$}}
\put(0,0){\circle*{1}}
\put(40,0){\circle*{1}}
\put(-80, 0){\line(1,0){160}}
\put(80,0){\vector(1,0){0}}

\put(-40,-2){\makebox(0,0)[tc]{$m=p_{1}$}}
\put(-0,-2){\makebox(0,0)[tc]{$p_{2}$}}
\put(40,-2){\makebox(0,0)[tc]{$p_{3}$}}

\put(66,0){\circle*{1}}

\put(60,-2){\makebox(0,0)[tc]{$\scriptstyle  p$}}
\put(66,-2){\makebox(0,0)[tc]{$\scriptstyle  n$}}
\put(-54,-2){\makebox(0,0)[tc]{$\scriptstyle  k$}}

\color{blue}

\put(10,0){\qbezier[100](-50,0)(-50,20.7)(-35.35,35.35)}
\put(10,0){\qbezier[100](-35.35,35.35)(-20.7,50)(0,50)}
\put(10,0){\qbezier[100](0,50)(20.7,50)(35.35,35.35)}
\put(10,0){\qbezier[100](35.35,35.35)(50,20.7)(50,0)}

\put(60,0){\vector(0,-1){0}}

\put(48,30){\makebox(0,0)[tr]{$ [m;p]_{\mathcal{C}}^{S}$}}

\put(6,0){\qbezier[100](-60,0)(-60,24.85)(-42.42,42.42)}
\put(6,0){\qbezier[100](-42.42,42.42)(-24.85,60)(0,60)}
\put(6,0){\qbezier[100](0,60)(24.85,60)(42.42,42.42)}
\put(6,0){\qbezier[100](42.42,42.42)(60,24.85)(60,0)}

\put(-54,0){\vector(0,-1){0}}

\put(-48,30){\makebox(0,0)[br]{$ [n;k]_{\mathcal{C}}^{S}$}}

\end{picture}
\caption{If $m=p_{1}$ then $ p_{3}<_{\Delta}P$ is impossible.  \label{figure10}}
\end{figure}

Similarily, $p=p_{3}$ is impossible for then the hole $\Psi(\Gamma)$ would be $[m; p]_{\mathcal{C}}^{S}$ and $n=p=p_{3}$ would hold. 

The case $p\in ]p_{1}; p_{2}[$ is also impossible by the construction of $\Psi(\Gamma)$ (see section \ref{spinningskeleton}), and finally one must show that $p\in ]p_{2}; p_{3}[$ is impossible. In this case holds, the next half circle of $\Psi(\Gamma)$ after $[m; p]_{\mathcal{C}}^{S}$ is some $[p;q]_{\mathcal{C}}^{-S}$ for $Q\in \Delta$ (see figure \ref{figure11}). One cannot have $q\in ]p_{1}; p_{2}[$ for in this case the skeleton must directly finish spinning to $p_{2}$, which implies that $\Psi(\Gamma) =  [p_{1}; p_{2}]_{\Delta}$, an excluded case.

\begin{figure}[htpb]
\setlength{\unitlength}{0.7mm}
\begin{picture}(45,65)(-110,-25)

\color{black}
\put(-40,0){\circle*{1}}
\put(0,0){\circle*{1}}
\put(-80,2){\makebox(0,0)[bc]{$ \Delta$}}
\put(0,0){\circle*{1}}
\put(40,0){\circle*{1}}
\put(-80, 0){\line(1,0){160}}
\put(80,0){\vector(1,0){0}}

\put(-40,-2){\makebox(0,0)[tc]{$m=p_{1}$}}
\put(-0,-2){\makebox(0,0)[tc]{$p_{2}$}}
\put(40,-2){\makebox(0,0)[tc]{$p_{3}$}}

\put(16,27){\makebox(0,0)[tr]{$ [m;p]_{\mathcal{C}}^{S}$}}
\put(11, 1){\makebox(0,0)[bl]{$\scriptstyle  p$}}
\put(-11, -1){\makebox(0,0)[tr]{$\scriptstyle  q$}}

\put(10,-12){\makebox(0,0)[tr]{$ [p;q]_{\mathcal{C}}^{-S}$}}

\color{blue}

\put(-15,0){\qbezier[100](-25,0)(-25,10.35)(-17.67,17.67)}
\put(-15,0){\qbezier[100](-17.67,17.67)(-10.35,25)(0,25)}
\put(-15,0){\qbezier[100](0,25)(10.35,25)(17.67,17.67)}
\put(-15,0){\qbezier[100](17.67,17.67)(25,10.35)(25,0)}

\put(10,0){\vector(0,-1){0}}

\put(0,0){\qbezier[100](-10,0)(-10,-4.14)(-7.07,-7.07)}
\put(0,0){\qbezier[100](-7.07,-7.07)(-4.14,-10)(0,-10)}
\put(0,0){\qbezier[100](0,-10)(4.14,-10)(7.07,-7.07)}
\put(0,0){\qbezier[100](7.07,-7.07)(10,-4.14)(10,0)}

\put(-10,0){\vector(0,1){0}}

\put(0,0){\qbezier[100](-10,0)(-10,-4.14)(-7.07,-7.07)}
\put(0,0){\qbezier[100](-7.07,-7.07)(-4.14,-10)(0,-10)}
\put(0,0){\qbezier[100](0,-10)(4.14,-10)(7.07,-7.07)}
\put(0,0){\qbezier[100](7.07,-7.07)(10,-4.14)(10,0)}

\put(-10,0){\vector(0,1){0}}

\color{mygreen}

\put(-2,0){\qbezier[10](-8,0)(-8,3.31)(-5.65,5.65)}
\put(-2,0){\qbezier[10](-5.65,5.65)(-3.31,8)(0,8)}
\put(-2,0){\qbezier[10](0,8)(3.31,8)(5.65,5.65)}
\put(-2,0){\qbezier[10](5.65,5.65)(8,3.31)(8,0)}

\put(6,0){\vector(0,-1){0}}

\end{picture}
\caption{If $m=p_{1}$ and $p_{2} <_{\Delta} p <_{\Delta} p_{3}$ then $ p_{1}<_{\Delta} q <_{\Delta} p_{2}$ is impossible. \label{figure11}}
\end{figure}

If $q\in [p_{3}; +\infty[$ then for the same reason the skeleton must finish spinning to $p_{3}$ which means that $p=p_{3}$ and $\Psi(\Gamma) = [p_{1}; p_{3}]_{\Psi(\Gamma)}^{S}$ (see figure \ref{figure12}).

\begin{figure}[htpb]
\begin{picture}(45,65)(-110,-25)

\color{black}

\put(-40,0){\circle*{1}}
\put(0,0){\circle*{1}}
\put(-80,2){\makebox(0,0)[bc]{$ \Delta$}}
\put(0,0){\circle*{1}}
\put(40,0){\circle*{1}}
\put(-80, 0){\line(1,0){160}}
\put(80,0){\vector(1,0){0}}

\put(-40,-2){\makebox(0,0)[tc]{$ m=p_{1}$}}
\put(-0,-2){\makebox(0,0)[tc]{$p_{2}$}}
\put(40,-2){\makebox(0,0)[tc]{$ p_{3}$}}

\put(5,22){\makebox(0,0)[tl]{$ [m;p]_{\mathcal{C}}^{S}$}}
\put(11, 1){\makebox(0,0)[bl]{$\scriptstyle p$}}

\put(46,-12){\makebox(0,0)[tl]{$ [p;q]_{\mathcal{C}}^{-S}$}}
\put(50, 1){\makebox(0,0)[bl]{$\scriptstyle q$}}

\color{blue}
\put(-15,0){\qbezier[100](-25,0)(-25,10.35)(-17.67,17.67)}
\put(-15,0){\qbezier[100](-17.67,17.67)(-10.35,25)(0,25)}
\put(-15,0){\qbezier[100](0,25)(10.35,25)(17.67,17.67)}
\put(-15,0){\qbezier[100](17.67,17.67)(25,10.35)(25,0)}
\put(10,0){\vector(0,-1){0}}

\put(30,0){\qbezier[100](-20,0)(-20,-8.28)(-14.14,-14.14)}
\put(30,0){\qbezier[100](-14.14,-14.14)(-8.28,-20)(0,-20)}
\put(30,0){\qbezier[100](0,-20)(8.28,-20)(14.14,-14.14)}
\put(30,0){\qbezier[100](14.14,-14.14)(20,-8.28)(20,0)}
\put(50,0){\vector(0,1){0}}

\color{mygreen}
\put(40,0){\qbezier[50](-10,0)(-10,4.14)(-7.07,7.07)}
\put(40,0){\qbezier[50](-7.07,7.07)(-4.14,10)(0,10)}
\put(40,0){\qbezier[50](0,10)(4.14,10)(7.07,7.07)}
\put(40,0){\qbezier[50](7.07,7.07)(10,4.14)(10,0)}

\put(30,0){\vector(0,-1){0}}

\put(38,0){\qbezier[10](-8,0)(-8,-3.31)(-5.65,-5.65)}
\put(38,0){\qbezier[10](-5.65,-5.65)(-3.31,-8)(0,-8)}
\put(38,0){\qbezier[10](0,-8)(3.31,-8)(5.65,-5.65)}
\put(38,0){\qbezier[10](5.65,-5.65)(8,-3.31)(8,0)}

\put(46,0){\vector(0,1){0}}

\end{picture}
\caption{If $m=p_{1}$ and $p_{2} <_{\Delta} p <_{\Delta} p_{3}$ then $ p_{3}<_{\Delta} q $ is impossible. \label{figure12}}
\end{figure}

The only possible left case is therefore $p=n$, wich proves the proposition in this case. And this proposition naturally also holds under the symetric hypothesis $n=p_{3}$. 

Now suppose that $m<_{\Delta} p_{1}$ and $p_{3} <_{\Delta} n$ : $m$ and $n$ are both junction points of two half-circles in the spinning skeleton $\Psi(\Gamma)$. Let us write those arcs $[m; m_{1}]_{\mathcal{C}}^{S}$; $[m; m_{2}]_{\mathcal{C}}^{-S}$; $[n_{1}; n]_{\mathcal{C}}^{S}$; $[n; n_{2}]_{\mathcal{C}}^{-S}$, with $S=+1$. Let us then suppose that that none of them connects $m$ and $n$. 

Il $m_{1} \geqslant_{\Delta} p_{3}$, since $m_{1} \leqslant_{\Delta} m$ and $m\neq m_{1}$, one has either the impossible $n_{1} <_{\Delta} m$ or the impossible $n_{2}<_{\Delta} m$ (in the manner of figure \ref{figure10}). Similarily, if both $p_{1} <_{\Delta} m_{1}  <_{\Delta} p_{2}$ and $p_{1} <_{\Delta} m_{2}  <_{\Delta} p_{2}$ then $\Psi(\Gamma)$ must directly finish spinning to $p_{1}$ from $m_{1}$ or $m_{2}$, according to the fact that $m_{1}<_{\Delta} m_{2}$ or  $m_{2}<_{\Delta} m_{1}$ respectively (see figure \ref{figure13}). But this implies that  the half circle $[m_{1};p_{1}]_{\mathcal{\C}}^{S}$ and the arc $[m;p_{1}]_{\Psi(\Gamma)}$ are isotopic with fixed ends in $\R^{2} \setminus \Big\{ p_{1}; p_{2}; p_{3} \Big\}$ and this contradicts the construction of $\Psi(\Gamma)$ and the hypothesis that $m<_{\Delta} p_{1}$ .

\begin{figure}[htpb]
\begin{picture}(45,65)(-110,-25)
\color{black}
\put(-40,0){\circle*{1}}
\put(0,0){\circle*{1}}
\put(-80,2){\makebox(0,0)[bc]{$ \Delta$}}
\put(0,0){\circle*{1}}
\put(40,0){\circle*{1}}
\put(-80, 0){\line(1,0){160}}
\put(80,0){\vector(1,0){0}}

\put(-40,-2){\makebox(0,0)[tc]{$p_{1}$}}
\put(-0,-2){\makebox(0,0)[tc]{$p_{2}$}}
\put(40,-2){\makebox(0,0)[tc]{$p_{3}$}}

\put(10,16){\makebox(0,0)[br]{$ [m;m_{1}]_{\mathcal{C}}^{S}$}}
\put(-55,1){\makebox(0,0)[br]{$\scriptstyle m$}}
\put(-6,1){\makebox(0,0)[br]{$\scriptstyle m_{1}$}}

\put(-16, -1){\makebox(0,0)[tr]{$\scriptstyle m_{2}$}}
\put(-20,-12){\makebox(0,0)[tl]{$ [m;m_{2}]_{\mathcal{C}}^{-S}$}}
\put(-20,-12){\makebox(0,0)[tl]{$ [m;m_{2}]_{\mathcal{C}}^{-S}$}}

\color{blue}

\put(-30,0){\qbezier[100](-25,0)(-25,10.35)(-17.67,17.67)}
\put(-30,0){\qbezier[100](-17.67,17.67)(-10.35,25)(0,25)}
\put(-30,0){\qbezier[100](0,25)(10.35,25)(17.67,17.67)}
\put(-30,0){\qbezier[100](17.67,17.67)(25,10.35)(25,0)}

\put(-5,0){\vector(0,-1){0}}

\put(-35,0){\qbezier[100](-20,0)(-20,-8.28)(-14.14,-14.14)}
\put(-35,0){\qbezier[100](-14.14,-14.14)(-8.28,-20)(0,-20)}
\put(-35,0){\qbezier[100](0,-20)(8.28,-20)(14.14,-14.14)}
\put(-35,0){\qbezier[100](14.14,-14.14)(20,-8.28)(20,0)}
\put(-15,0){\vector(0,1){0}}

\color{mygreen}
\put(-31,0){\qbezier[30](-16,0)(-16,6.62)(-11.31,11.31)}
\put(-31,0){\qbezier[30](-11.31,11.31)(-6.62,16)(0,16)}
\put(-31,0){\qbezier[30](0,16)(6.62,16)(11.31,11.31)}
\put(-31,0){\qbezier[30](11.31,11.31)(16,6.62)(16,0)}
\put(-47,0){\vector(0,-1){0}}

\put(-37,0){\qbezier[15](-10,0)(-10,-4.14)(-7.07,-7.07)}
\put(-37,0){\qbezier[15](-7.07,-7.07)(-4.14,-10)(0,-10)}
\put(-37,0){\qbezier[15](0,-10)(4.14,-10)(7.07,-7.07)}
\put(-37,0){\qbezier[15](7.07,-7.07)(10,-4.14)(10,0)}
\put(-27,0){\vector(0,1){0}}

\end{picture}
\caption{If $m <_{\Delta} p_{1}$  then $p_{1} <_{\Delta} m_{1}  <_{\Delta} p_{2}$ and $p_{1} <_{\Delta} m_{2}  <_{\Delta} p_{2}$ is impossible. \label{figure13}}
\end{figure}

Now if both $p_{2} <_{\Delta} m_{1}  <_{\Delta} p_{3}$ and $p_{2} <_{\Delta} m_{2}  <_{\Delta} p_{3}$ then both $p_{2} <_{\Delta} n_{1}  <_{\Delta} p_{3}$ and $p_{2} <_{\Delta} n_{2}  <_{\Delta} p_{3}$ and this case is symetric from the previous one (see figure \ref{figure13}) and also impossible. So necessarily one has : $p_{1} <_{\Delta} m_{1}  <_{\Delta} p_{2}$;  $p_{2} <_{\Delta} m_{2}  <_{\Delta} p_{3}$, $p_{1} <_{\Delta} m_{1}  <_{\Delta} p_{2}$ and  $p_{2} <_{\Delta} m_{2}  <_{\Delta} p_{3}$. As in the previous cases, the landing point of half circle following $[n_{1}; n]_{\Psi(\Gamma)}$ in $\Psi(\Gamma)$ cannot be in $]p_{3}; n[$; neither can the landing point of the half circle following $[m; m_{2}]_{\Psi(\Gamma)}$ in $\Psi(\Gamma)$ be in $]m; p_{1}[$. Those points must therefore respectively be in $]n_{2}; p_{2}[_{\Delta}$ and $]p_{2}; m_{1}[_{\Delta}$. But in this case the four subarcs of $\Psi(\Gamma)$ following respectively $[n; n_{1}]_{\mathcal{C}}^{S}$; $[n; n_{2}]_{\mathcal{C}}^{S}$; $[m; m_{1}]_{\mathcal{C}}^{S}$ and $[m; m_{2}]_{\mathcal{C}}^{S}$ must converge to $p_{2}$, an impossible situation (see figure \ref{figure14}).

\begin{figure}[htpb]

\begin{picture}(45,75)(-110,-35)
\color{black}
\put(-40,0){\circle*{1}}
\put(0,0){\circle*{1}}
\put(-80,2){\makebox(0,0)[bc]{$ \Delta$}}
\put(0,0){\circle*{1}}
\put(40,0){\circle*{1}}
\put(-80, 0){\line(1,0){160}}
\put(80,0){\vector(1,0){0}}

\put(-40,-2){\makebox(0,0)[tc]{$ p_{1}$}}
\put(-0,-2){\makebox(0,0)[tc]{$ p_{2}$}}
\put(40,-2){\makebox(0,0)[tc]{$ p_{3}$}}

\put(-45,25){\makebox(0,0)[br]{$ [m;m_{1}]_{\mathcal{C}}^{S}$}}
\put(-56,1){\makebox(0,0)[br]{$ m$}}
\put(-55,0){\circle*{1}}
\put(15,-1){\makebox(0,0)[tc]{$\scriptstyle m_{1}$}}

\put(-16, 1){\makebox(0,0)[br]{$\scriptstyle m_{2}$}}
\put(-45,-18){\makebox(0,0)[tr]{$ [m;m_{2}]_{\mathcal{C}}^{-S}$}}

\put(-9, 1){\makebox(0,0)[bc]{$\scriptstyle n_{2}$}}
\put(55,-18){\makebox(0,0)[tl]{$ [n;n_{2}]_{\mathcal{C}}^{-S}$}}

\put(45,19){\makebox(0,0)[bl]{$ [n;n_{1}]_{\mathcal{C}}^{S}$}}
\put(60,0){\circle*{1}}
\put(61,1){\makebox(0,0)[bl]{$\scriptstyle n$}}
\put(21,-1){\makebox(0,0)[tl]{$\scriptstyle n_{1}$}}

\color{blue}

\put(-20,0){\qbezier[100](-35,0)(-35,14.5)(-24.75,24.75)}
\put(-20,0){\qbezier[100](-24.75,24.75)(-14.5,35)(0,35)}
\put(-20,0){\qbezier[100](0,35)(14.5,35)(24.75,24.75)}
\put(-20,0){\qbezier[100](24.75,24.75)(35,14.5)(35,0)}

\put(15,0){\vector(0,-1){0}}

\put(-35,0){\qbezier[100](-20,0)(-20,-8.28)(-14.14,-14.14)}
\put(-35,0){\qbezier[100](-14.14,-14.14)(-8.28,-20)(0,-20)}
\put(-35,0){\qbezier[100](0,-20)(8.28,-20)(14.14,-14.14)}
\put(-35,0){\qbezier[100](14.14,-14.14)(20,-8.28)(20,0)}
\put(-15,0){\vector(0,1){0}}

\put(25,0){\qbezier[100](-35,0)(-35,-14.5)(-24.75,-24.75)}
\put(25,0){\qbezier[100](-24.75,-24.75)(-14.5,-35)(0,-35)}
\put(25,0){\qbezier[100](0,-35)(14.5,-35)(24.75,-24.75)}
\put(25,0){\qbezier[100](24.75,-24.75)(35,-14.5)(35,0)}

\put(-10,0){\vector(0,1){0}}

\put(40,0){\qbezier[100](-20,0)(-20,8.28)(-14.14,14.14)}
\put(40,0){\qbezier[100](-14.14,14.14)(-8.28,20)(0,20)}
\put(40,0){\qbezier[100](0,20)(8.28,20)(14.14,14.14)}
\put(40,0){\qbezier[100](14.14,14.14)(20,8.28)(20,0)}
\put(20,0){\vector(0,-1){0}}

\color{mygreen}

\put(8,0){\qbezier[30](-12,0)(-12,-4.97)(-8.48,-8.48)}
\put(8,0){\qbezier[30](-8.48,-8.48)(-4.97,-12)(0,-12)}
\put(8,0){\qbezier[30](0,-12)(4.97,-12)(8.48,-8.48)}
\put(8,0){\qbezier[30](8.48,-8.48)(12,-4.97)(12,0)}
\put(-4,0){\vector(0,1){0}}

\put(-5,0){\qbezier[40](-10,0)(-10,4.14)(-7.07,7.07)}
\put(-5,0){\qbezier[40](-7.07,7.07)(-4.14,10)(0,10)}
\put(-5,0){\qbezier[40](0,10)(4.14,10)(7.07,7.07)}
\put(-5,0){\qbezier[40](7.07,7.07)(10,4.14)(10,0)}
\put(5,0){\vector(0,-1){0}}

\end{picture}
\caption{If $m <_{\Delta} p_{1}$  then $p_{1} <_{\Delta} m_{1}  <_{\Delta} p_{2}$ and $p_{1} <_{\Delta} m_{2}  <_{\Delta} p_{2}$ is impossible. \label{figure14}}
\end{figure}

\end{proof}

According to the proposition \ref{proposition1}, let us say that the curve $\Gamma$ is in top position if $S=+1$ and in bottom position if $S=-1$. Since obviously $\Psi\left(\sigma(\Gamma)\right) = \sigma\left(\Psi(\Gamma)\right)$, the symetry $\sigma$ relative to $\Delta$ changes the position of a curve $\Gamma$ to its opposite. Let us call {\texteit visible emerging arc} of $\Psi(\Gamma)$ the half circle $[m;n]_{C}^{S}$ in $\Psi(\Gamma)$ for a given side $S(\Gamma)$, and similarily let us define the {\texteit emerging arcs} of $\Psi(\Gamma)$ to be the half circles of $\Psi(\Gamma)$ that are on the side $S$ of the visible emerging arc, as opposed to the {\texteit sustaining arcs} of $\Psi(\Gamma)$ that are on the opposite side. Let us define the correponding sub-arcs $[m;m_{1}]_{\Psi(\Gamma)}$ and $[n;n_{1}]_{\Psi(\Gamma)}$ of $\Psi(\Gamma)$ to be respectively the {\texteit left visible sustaining arc} and the {\texteit right visible sustaining arc} of $\Psi(\Gamma)$. If $m=p_{1}$, then let us consider that $m_{1} = p_{1}=m$ and say that $[m;m_{1}]_{\Psi(\Gamma)}$ is reduced to a point ~: in this case that the left sustaining arc of $\Psi(\Gamma)$ is trivial. Else, let us call $m_{1}$ the {\textit second extremity of the left visible sustaining arc} of $\Psi(\Gamma)$, and let us speak similarily of $n_{1}$. If both sustaining arcs of $\Psi(\Gamma)$ are trivial, then $\Psi(\Gamma)$ is reduced to $[p_{1}; p_{2}]_{\mathcal{C}}^{S}$. Else, let us suppose that $S(\Gamma)=+1$ and show the following~:

\begin{prop}
Suppose that $\Psi(\Gamma)$ contains more than one segment or half circle and let $m_{1}$ and $n_{1}$ be the second extremities respectively  of the left and of the right visible sustaining arc of $\Psi(\Gamma)$. Then either 
$$ p_{1} \leqslant_{\Delta} m_{1} <_{\Delta} n_{1} \leqslant p_{2} $$
or 
$$ p_{2} \leqslant_{\Delta} m_{1} <_{\Delta} n_{1} \leqslant p_{3} $$

Furthermore, the extremities of each sustaining arc of $\Psi(\Gamma)$ are all either separated by $p_{1}$ or separated by $p_{3}$ exclusively, and the extremities of each emerging arc of $\Psi(\Gamma)$ are all separated by $p_{2}$.
\label{proposition2}
\end{prop}

\begin{proof} It is sufficient to prove that $p_{2} \notin ]m_{1}; n_{1}[_{\Delta}$, and there are three possibilities for this to happen : the two similar cases  $\left( m_{1} = m; n_{1} \neq n\right) $ or 
 $\left( m_{1} \neq m; n_{1} = n \right)$; and the case $\left( m_{1} \neq m ; n_{1} \neq n \right)$. 

Let us first suppose that $p_{2} \in ]m_{1}; n_{1}[_{\Delta}$ with the condition $m_{1} = m=p_{1}$ and $n_{1} \neq n$. Then, following $n_{1}$ from $n$ the spinning skeleton $\Psi(\Gamma)$ must take on its right to the right of $p_{3}$ or on its left to the left of $p_{2}$. In the first case it must finish spinning directly to $p_{3}$, which implies that $n_{1}= p_{3}$ and then $n_{1}=n$, and excluded situation. In the second case it must  finish spinning directly to $p_{2}$, and similarily one has $n_{1}=p_{2}$ (see figure \ref{figure15}).

\begin{figure}[htpb]

\begin{picture}(45,85)(-110,-20)

\color{black}

\put(-40,0){\circle*{1}}
\put(0,0){\circle*{1}}
\put(-80,2){\makebox(0,0)[bc]{$ \Delta$}}
\put(0,0){\circle*{1}}
\put(40,0){\circle*{1}}
\put(-80, 0){\line(1,0){160}}
\put(80,0){\vector(1,0){0}}

\put(-40,-2){\makebox(0,0)[tc]{$m=p_{1}=m_{1}$}}
\put(-0,-2){\makebox(0,0)[tc]{$p_{2}$}}
\put(40,-2){\makebox(0,0)[tc]{$p_{3}$}}

\put(55,30){\makebox(0,0)[tl]{$ [m;n]_{\mathcal{C}}^{S=+1}$}}

\put(62,1){\makebox(0,0)[bl]{$\scriptstyle  n$}}
\put(19,-1){\makebox(0,0)[tr]{$\scriptstyle n_{1}$}}

\put(52,-23){\makebox(0,0)[bl]{$ [n;n_{1}]_{\mathcal{C}}^{S}$}}

\color{blue}
\put(10,0){\qbezier[100](-50,0)(-50,20.7)(-35.35,35.35)}
\put(10,0){\qbezier[100](-35.35,35.35)(-20.7,50)(0,50)}
\put(10,0){\qbezier[100](0,50)(20.7,50)(35.35,35.35)}
\put(10,0){\qbezier[100](35.35,35.35)(50,20.7)(50,0)}

\put(60,0){\vector(0,-1){0}}

\put(40,0){\qbezier[100](-20,0)(-20,-8.28)(-14.14,-14.14)}
\put(40,0){\qbezier[100](-14.14,-14.14)(-8.28,-20)(0,-20)}
\put(40,0){\qbezier[100](0,-20)(8.28,-20)(14.14,-14.14)}
\put(40,0){\qbezier[100](14.14,-14.14)(20,-8.28)(20,0)}
\put(20,0){\vector(0,1){0}}

\put(60,0){\circle*{1}}

\color{mygreen}

\put(32,0){\qbezier[50](-12,0)(-12,4.97)(-8.48,8.48)}
\put(32,0){\qbezier[50](-8.48,8.48)(-4.97,12)(0,12)}
\put(32,0){\qbezier[50](0,12)(4.97,12)(8.48,8.48)}
\put(32,0){\qbezier[50](8.48,8.48)(12,4.97)(12,0)}
\put(44,0){\vector(0,-1){0}}

\put(0,0){\qbezier[100](-20,0)(-20,8.28)(-14.14,14.14)}
\put(0,0){\qbezier[100](-14.14,14.14)(-8.28,20)(0,20)}
\put(0,0){\qbezier[100](0,20)(8.28,20)(14.14,14.14)}
\put(0,0){\qbezier[100](14.14,14.14)(20,8.28)(20,0)}
\put(-20,0){\vector(0,-1){0}}

\end{picture}
\caption{If $m=p_{1}$ and $n\neq p_{3}$, then $p_{2}<_{\Delta} n_{1}<_{\Delta} <p_{3}$ is impossible.  \label{figure15}}
\end{figure}

Now suppose that $p_{2} \in ]m_{1}; n_{1}[_{\Delta}$ with $m_{1} \neq m$ and $n_{1} \neq n$. Then following the spinning skeleton $\Psi(\Gamma)$ from $m$ after $m_{1}$, there are four possible intervals for the landing point $m_{2}$ of the arc $[m_{1};m_{2}]_{\mathcal{C}}^{S=+1}$ following $m_{1}$ (see figure \ref{figure16}). If $m_{2} \in [m;p_{1}]_{\Delta}$ then $\Psi(\Gamma)$ must finish spinning to $p_{1}$ and necessarily $m=p_{1}=m_{1}$ which is excluded. If  $m_{2} \in [p_{2};n_{1}]_{\Delta}$ then $\Psi(\Gamma)$ must finish spinning to $p_{2}$ and necessarily $m_{1}=p_{2}$ which is excluded. Now if $m_{2} \in [n_{1};p_{3}]_{\Delta}$ or $m_{2} \in [p_{3};n]_{\Delta}$, then one can follow $\Psi(\Gamma)$ from $n$ after $n_{1}$ and consider the landing point $n_{2}$ of the arc $[n_{1};n_{2}]_{\Psi(\Gamma)}^{S=+1}$ : as in figure \ref{figure16} it must necessarily finish spinning to $p_{2}$ or $p_{3}$ which is excluded. This proves the first part of the proposition.

\begin{figure}[htpb]

\begin{picture}(45,85)(-110,-20)

\color{black}

\put(-40,0){\circle*{1}}
\put(0,0){\circle*{1}}
\put(-80,2){\makebox(0,0)[bc]{$ \Delta$}}
\put(0,0){\circle*{1}}
\put(40,0){\circle*{1}}
\put(-80, 0){\line(1,0){160}}
\put(80,0){\vector(1,0){0}}

\put(-40,-2){\makebox(0,0)[tc]{$p_{1}$}}
\put(-0,-2){\makebox(0,0)[tc]{$p_{2}$}}
\put(40,-2){\makebox(0,0)[tc]{$p_{3}$}}

\put(66,0){\circle*{1}}

\put(67,1){\makebox(0,0)[bl]{$\scriptstyle n$}}
\put(-57,1){\makebox(0,0)[bc]{$\scriptstyle m$}}

\put(-48,-16){\makebox(0,0)[tr]{$ [m;m_{1}]_{\mathcal{C}}^{S=-1}$}}
\put(-13,-1){\makebox(0,0)[tl]{$\scriptstyle m_{1}$}}
\put(-14,0){\makebox(0,0)[tl]{\vector(0,1){0}}}

\put(58,-16){\makebox(0,0)[tl]{$ [n;n_{1}]_{\mathcal{C}}^{S=-1}$}}
\put(25,-1){\makebox(0,0)[tr]{$\scriptstyle n_{1}$}}

\color{blue}

\put(-34,0){\qbezier[100](-20,0)(-20,-8.28)(-14.14,-14.14)}
\put(-34,0){\qbezier[100](-14.14,-14.14)(-8.28,-20)(0,-20)}
\put(-34,0){\qbezier[100](0,-20)(8.28,-20)(14.14,-14.14)}
\put(-34,0){\qbezier[100](14.14,-14.14)(20,-8.28)(20,0)}

\put(46,0){\qbezier[100](-20,0)(-20,-8.28)(-14.14,-14.14)}
\put(46,0){\qbezier[100](-14.14,-14.14)(-8.28,-20)(0,-20)}
\put(46,0){\qbezier[100](0,-20)(8.28,-20)(14.14,-14.14)}
\put(46,0){\qbezier[100](14.14,-14.14)(20,-8.28)(20,0)}

\put(26,0){\vector(0,1){0}}

\put(6,0){\qbezier[150](-60,0)(-60,24.85)(-42.42,42.42)}
\put(6,0){\qbezier[150](-42.42,42.42)(-24.85,60)(0,60)}
\put(6,0){\qbezier[150](0,60)(24.85,60)(42.42,42.42)}
\put(6,0){\qbezier[150](42.42,42.42)(60,24.85)(60,0)}

\put(-54,0){\circle*{1}}

\put(-48,30){\makebox(0,0)[br]{$ [m;n]_{\mathcal{C}}^{S=+1}$}}

\color{mygreen}

\put(-2,0){\qbezier[50](-12,0)(-12,4.97)(-8.48,8.48)}
\put(-2,0){\qbezier[50](-8.48,8.48)(-4.97,12)(0,12)}
\put(-2,0){\qbezier[50](0,12)(4.97,12)(8.48,8.48)}
\put(-2,0){\qbezier[50](8.48,8.48)(12,4.97)(12,0)}
\put(10,0){\vector(0,-1){0}}

\put(-29,0){\qbezier[100](-15,0)(-15,6.21)(-10.6,10.6)}
\put(-29,0){\qbezier[100](-10.6,10.6)(-6.21,15)(0,15)}
\put(-29,0){\qbezier[100](0,15)(6.21,15)(10.6,10.6)}
\put(-29,0){\qbezier[100](10.6,10.6)(15,6.21)(15,0)}
\put(-44,0){\vector(0,-1){0}}

\put(21,0){\qbezier[100](-35,0)(-35,14.5)(-24.75,24.75)}
\put(21,0){\qbezier[100](-24.75,24.75)(-14.5,35)(0,35)}
\put(21,0){\qbezier[100](0,35)(14.5,35)(24.75,24.75)}
\put(21,0){\qbezier[100](24.75,24.75)(35,14.5)(35,0)}

\put(56,0){\vector(0,-1){0}}

\put(11,0){\qbezier[100](-25,0)(-25,10.35)(-17.67,17.67)}
\put(11,0){\qbezier[100](-17.67,17.67)(-10.35,25)(0,25)}
\put(11,0){\qbezier[100](0,25)(10.35,25)(17.67,17.67)}
\put(11,0){\qbezier[100](17.67,17.67)(25,10.35)(25,0)}

\put(36,0){\vector(0,-1){0}}

\end{picture}
\caption{If $m\neq p_{1}$ and $n\neq p_{3}$ then $ m_{1}<_{\Delta} p_{2} < n_{1}$ is impossible.  \label{figure16}}
\end{figure}

To prove the second part of this proposition, let us first consider that one of the two symetric possibilities given by the first part holds : for example that 
$$ p_{1} \leqslant_{\Delta} m_{1} <_{\Delta} n_{1} \leqslant p_{2}\, . $$

\begin{figure}[htpb]

\begin{picture}(45,90)(-110,-30)

\color{black}

\put(-40,0){\circle*{1}}
\put(-80,2){\makebox(0,0)[bc]{$\Delta$}}
\put(9,0){\circle*{1}}
\put(41,0){\circle*{1}}
\put(-80, 0){\line(1,0){160}}
\put(80,0){\vector(1,0){0}}

\put(-40,-1){\makebox(0,0)[tc]{$p_{1}$}}
\put(9,-1){\makebox(0,0)[tc]{$p_{2}$}}
\put(41,-1){\makebox(0,0)[tc]{$p_{3}$}}

\put(60,0){\circle*{1}}
\put(-60,0){\circle*{1}}

\put(61,1){\makebox(0,0)[bl]{$\scriptstyle n$}}
\put(-61,1){\makebox(0,0)[br]{$\scriptstyle m$}}

\put(-52,-16){\makebox(0,0)[tr]{$ [n;n_{1}]_{\mathcal{C}}^{S=-1}$}}
\put(-20,1){\makebox(0,0)[bc]{$\scriptstyle m_{1}$}}
\put(-20,0){\makebox(0,0)[tl]{\vector(0,1){0}}}

\put(58,-16){\makebox(0,0)[tl]{$ [n;n_{1}]_{\mathcal{C}}^{S=-1}$}}
\put(-8,1){\makebox(0,0)[bc]{$\scriptstyle n_{1}$}}

\put(-54,30){\makebox(0,0)[br]{$ [m;n]_{\mathcal{C}}^{S=+1}$}}
\color{blue}

\put(-40,0){\qbezier[100](-20,0)(-20,-8.28)(-14.14,-14.14)}
\put(-40,0){\qbezier[100](-14.14,-14.14)(-8.28,-20)(0,-20)}
\put(-40,0){\qbezier[100](0,-20)(8.28,-20)(14.14,-14.14)}
\put(-40,0){\qbezier[100](14.14,-14.14)(20,-8.28)(20,0)}

\put(26,0){\qbezier[100](-34,0)(-34,-14.08)(-24.04,-24.04)}
\put(26,0){\qbezier[100](-24.04,-24.04)(-14.08,-34)(0,-34)}
\put(26,0){\qbezier[100](0,-34)(14.08,-34)(24.04,-24.04)}
\put(26,0){\qbezier[100](24.04,-24.04)(34,-14.08)(34,0)}

\put(-8,0){\vector(0,1){0}}

\put(0,0){\qbezier[150](-60,0)(-60,24.85)(-42.42,42.42)}
\put(0,0){\qbezier[150](-42.42,42.42)(-24.85,60)(0,60)}
\put(0,0){\qbezier[150](0,60)(24.85,60)(42.42,42.42)}
\put(0,0){\qbezier[150](42.42,42.42)(60,24.85)(60,0)}

\put(12,0){\qbezier[50](-12,0)(-12,-4.97)(-8.48,-8.48)}
\put(12,0){\qbezier[50](-8.48,-8.48)(-4.97,-12)(0,-12)}
\put(12,0){\qbezier[50](0,-12)(4.97,-12)(8.48,-8.48)}
\put(12,0){\qbezier[50](8.48,-8.48)(12,-4.97)(12,0)}
\put(24,0){\vector(0,1){0}}

\put(0,0){\vector(0,1){0}}

\color{black}
\put(12,-13){\makebox(0,0)[tc]{$\Lambda$}}
\put(1,-1){\makebox(0,0)[tl]{$\scriptstyle q_{1}$}}
\put(25,-1){\makebox(0,0)[tl]{$\scriptstyle r_{1}$}}

\color{mygreen}

\put(16,0){\qbezier[100](-16,0)(-16,6.62)(-11.3,11.3)}
\put(16,0){\qbezier[100](-11.3,11.3)(-6.62,16)(0,16)}
\put(16,0){\qbezier[100](0,16)(6.62,16)(11.3,11.3)}
\put(16,0){\qbezier[100](11.3,11.3)(16,6.62)(16,0)}

\put(32,0){\vector(0,-1){0}}


\put(10,0){\qbezier[50](-10,0)(-10,4.14)(-7.07,7.07)}
\put(10,0){\qbezier[50](-7.07,7.07)(-4.14,10)(0,10)}
\put(10,0){\qbezier[50](0,10)(4.14,10)(7.07,7.07)}
\put(10,0){\qbezier[50](7.07,7.07)(10,4.14)(10,0)}
\put(20,0){\vector(0,-1){0}}

\put(-27,0){\qbezier[100](-27,0)(-27,11.18)(-19.09,19.09)}
\put(-27,0){\qbezier[100](-19.09,19.09)(-11.18,27)(0,27)}
\put(-27,0){\qbezier[100](0,27)(11.18,27)(19.09,19.09)}
\put(-27,0){\qbezier[100](19.09,19.09)(27,11.18)(27,0)}

\put(-54,0){\vector(0,-1){0}}

\end{picture}
\caption{The extremities of any sustaining arc must be separated either by $ p_{1}$ or by $ p_{3}$ (I).  \label{figure17}}
\end{figure}

The segment $]m_{1};n_{1}[_{\Delta}$ doesn't contain any point of $\Psi(\Gamma)$, for such a point would be an extremity of a sustaining arc whose second end would be either before $m$ or after $n$ on $\Delta$, contrary to the definition of $m$ and $n$. The sustaining arcs of $\Psi(\Gamma)$ are therefore of two types~: the left ones, whose extremities both belong to $[m; m_{1}]_{\Delta}$, and the right ones, whose extremities both belong to  $[n_{1}; n]_{\Delta}$. Obviously the extremities of the left ones are separated by $p_{1}$ according to the spinning skeleton theorem (theorem \ref{spinningskeleton} page \pageref{spinningskeleton}). 

\begin{figure}[htpb]

\begin{picture}(45,90)(-110,-30)

\color{black}

\put(-40,0){\circle*{1}}
\put(-80,2){\makebox(0,0)[bc]{$\Delta$}}
\put(9,0){\circle*{1}}
\put(41,0){\circle*{1}}
\put(-80, 0){\line(1,0){160}}
\put(80,0){\vector(1,0){0}}

\put(-40,-1){\makebox(0,0)[tc]{$p_{1}$}}
\put(9,-1){\makebox(0,0)[tc]{$p_{2}$}}
\put(41,-1){\makebox(0,0)[tc]{$p_{3}$}}

\put(60,0){\circle*{1}}
\put(-60,0){\circle*{1}}

\put(61,1){\makebox(0,0)[bl]{$\scriptstyle n$}}
\put(-61,1){\makebox(0,0)[br]{$\scriptstyle m$}}

\put(-52,-16){\makebox(0,0)[tr]{$ [m;m_{1}]_{\mathcal{C}}^{S=-1}$}}
\put(-21,-1){\makebox(0,0)[tr]{$\scriptstyle m_{1}$}}

\put(58,-16){\makebox(0,0)[tl]{$ [n;n_{1}]_{\mathcal{C}}^{S=-1}$}}
\put(-9,-1){\makebox(0,0)[tr]{$\scriptstyle n_{1}$}}

\color{blue}
\put(-40,0){\qbezier[100](-20,0)(-20,-8.28)(-14.14,-14.14)}
\put(-40,0){\qbezier[100](-14.14,-14.14)(-8.28,-20)(0,-20)}
\put(-40,0){\qbezier[100](0,-20)(8.28,-20)(14.14,-14.14)}
\put(-40,0){\qbezier[100](14.14,-14.14)(20,-8.28)(20,0)}

\put(-20,0){\makebox(0,0)[tl]{\vector(0,1){0}}}

\put(26,0){\qbezier[100](-34,0)(-34,-14.08)(-24.04,-24.04)}
\put(26,0){\qbezier[100](-24.04,-24.04)(-14.08,-34)(0,-34)}
\put(26,0){\qbezier[100](0,-34)(14.08,-34)(24.04,-24.04)}
\put(26,0){\qbezier[100](24.04,-24.04)(34,-14.08)(34,0)}

\put(-8,0){\vector(0,1){0}}

\put(0,0){\qbezier[150](-60,0)(-60,24.85)(-42.42,42.42)}
\put(0,0){\qbezier[150](-42.42,42.42)(-24.85,60)(0,60)}
\put(0,0){\qbezier[150](0,60)(24.85,60)(42.42,42.42)}
\put(0,0){\qbezier[150](42.42,42.42)(60,24.85)(60,0)}

\color{black}
\put(-54,30){\makebox(0,0)[br]{$ [m;n]_{\mathcal{C}}^{S=+1}$}}

\put(12,0){\qbezier[50](-12,0)(-12,-4.97)(-8.48,-8.48)}
\put(12,0){\qbezier[50](-8.48,-8.48)(-4.97,-12)(0,-12)}
\put(12,0){\qbezier[50](0,-12)(4.97,-12)(8.48,-8.48)}
\put(12,0){\qbezier[50](8.48,-8.48)(12,-4.97)(12,0)}
\put(24,0){\vector(0,1){0}}

\put(0,0){\vector(0,1){0}}
\put(12,-11){\makebox(0,0)[bc]{$\scriptstyle  \Lambda$}}
\put(1,-1){\makebox(0,0)[tl]{$\scriptstyle  q_{1}$}}
\put(25,-1){\makebox(0,0)[tl]{$\scriptstyle r_{1}$}}

\put(27,0){\qbezier[100](-27,0)(-27,11.18)(-19.09,19.09)}
\put(27,0){\qbezier[100](-19.09,19.09)(-11.18,27)(0,27)}
\put(27,0){\qbezier[100](0,27)(11.18,27)(19.09,19.09)}
\put(27,0){\qbezier[100](19.09,19.09)(27,11.18)(27,0)}

\put(54,0){\vector(0,-1){0}}

\put(54,-1){\makebox(0,0)[tl]{$\scriptstyle q_{2}$}}

\color{mygreen}

\put(36,0){\qbezier[50](-12,0)(-12,4.97)(-8.48,8.48)}
\put(36,0){\qbezier[50](-8.48,8.48)(-4.97,12)(0,12)}
\put(36,0){\qbezier[50](0,12)(4.97,12)(8.48,8.48)}
\put(36,0){\qbezier[50](8.48,8.48)(12,4.97)(12,0)}
\put(48,0){\vector(0,-1){0}}


\put(14,0){\qbezier[50](-10,0)(-10,4.14)(-7.07,7.07)}
\put(14,0){\qbezier[50](-7.07,7.07)(-4.14,10)(0,10)}
\put(14,0){\qbezier[50](0,10)(4.14,10)(7.07,7.07)}
\put(14,0){\qbezier[50](7.07,7.07)(10,4.14)(10,0)}
\put(4,0){\vector(0,-1){0}}

\put(25,0){\qbezier[100](-29,0)(-29,-12)(-20.5,-20.5)}
\put(25,0){\qbezier[100](-20.5,-20.5)(-12,-29)(0,-29)}
\put(25,0){\qbezier[100](0,-29)(12,-29)(20.5,-20.5)}
\put(25,0){\qbezier[100](20.5,-20.5)(29,-12)(29,0)}

\put(-4,0){\vector(0,1){0}}

\end{picture}
\caption{The extremities of any sustaining arc must be separated either by $p_{1}$ or by $p_{3}$ (II).  \label{figure18}}
\end{figure}

Let us now consider a right sustaining arc $\Lambda=[q_{1}; r_{1}]_{\Gamma}$ whose extremities $q_{1}$ and $r_{1}$ belong to $[n_{1}; n]_{\Delta}$ and are such that $p_{3} \notin ]p_{1}; r_{1}[$. Let us follow $\Psi(\Gamma)$ from $r_{1}$ after $q_{1}$ and consider the half circle $[q_{1}; q_{2}]_{S=+1}$ of $\Psi(\Gamma)$. If $q_{2} \in ]m;p_{1}]$, then the skeleton $\Psi(\Gamma)$ can be both continued after $m_{1}$ and $n_{1}$ for they would both finish spining to $p_{1}$, an impossible configuration (see figure \ref{figure12}). Now of course the point $q_{2}$ cannot belong to $]p_{1}; p_{2}[$, for this interval doesn't contain any point of $X= \Big\{ p_{1}; p_{2}; p_{3} \Big\}$ (see the spinning skeleton, theorem \ref{spinningskeleton}). If $q_{2} \in ]p_{2}; r_{1}]_{\Delta}$ then $\Psi(\Gamma)$ must finish spinning to $p_{2}$ after $q_{2}$ and this imply that $q_{1} = p_{2}$, contrary to the hypothesis that $q_{1}$ and $r_{1}$ are separated by $p_{2}$. If $q_{2} \in ]r_{1}; p_{3}]_{\Delta}$ (see figure ref{figure12}), one obtains a similar conclusion following $\Psi(\Gamma)$ after $r_{1}$ from $q_{1}$. In the case $q_{2}= p_{3}$, then $\Lambda$ must finish spinning to $p_{2}$ after $r_{1}$ from $q_{1}$ and this not possible.

Suppose now that $q_{2} \in ]p_{3}; n]_{\Delta}$ (see figure \ref{figure18}). Then following $\Lambda$ from $q_{1}$ after $r_{1}$, the point $r_{2}$ must belong to $]q_{1}; p_{2}]$ and in this case is impossible for $r_{1} = p_{2}$ would hold, or the point $r_{2}$ must belong to $[p_{3}; n[$. But in this case the point $q_{3}$ following $q_{2}$ must belong to $]n_{1}; q_{1}]$ (see figure \ref{figure18}), and then it has to go on cirling around for it can never join the left-hand side (see figure \ref{figure17}). The same reasoning shows that $p_{2}$ must separate the extremities of any emerging arc. In particular, if $p_{2}$ is an extremity of a non trivial $\Psi(\Gamma)$, it must belong to a sustaining arc, on the side of the visible sustaining arc that contains two of the points $p_{1}$; $p_{2}$; $p_{3}$.\end{proof}

 \section{ \label{pictures} Pictures of the simple topological snail}

The propositions \ref{proposition1} and \ref{proposition2} allows one to give a general description of the simplest spinning skeleton as a particular topological snail (figure \ref{figure19}).  On each of the three vertical half-lines $H_{1}$; $H_{2}^{+}$ and $H_{3}$, with origins $p_{1}$; $p_{2}$ and $p_{3}$ (figure \ref{figure7} page \pageref{figure7} and figure \ref{figure20}), let us consider the respective numbers $n_{1}$; $n_{2}$ and $n_{3}$ of extremities of sub-arcs of $\Psi(\Gamma)$. According to proposition $\ref{proposition2}$, any point of $\Psi(\Gamma) \cup \Delta^{+}$ belongs a unique sub-arc of $\Psi(\Gamma)$ doubly branched of $H_{1} \cup H_{2} \cup H_{3}$, with an extremity in $H_{1} \cup H_{2}$ and the other in $H_{3}$. Of course this correspondance is bijective and one therefore has $n_{1} + n_{3} = n_{2}$, a number that is equal to the cardinal of $\Psi(\Gamma) \cup \Delta^{+}$. In a similar fashion, $n_{1}$ is the number of extremities of the left sustaining arcs, the trival arc $p_{1}$ being taken into account if necessary as a trivial arc reduced to a single extremity, $n_{3}$ is the number of extremities of the left sustaining arcs, and $n_{2}$ the number of extremities of the emerging arcs of $\Psi(\Gamma)$ (figure \ref{figure19}). One can also see that $n_{1}$ is the number of intersection between $\Psi(\Gamma)$ and the homotopy class of $\Delta_{1}$ and $n_{3}$ the number of intersection between $\Psi(\Gamma)$ and the homotopy class of $\Delta_{2}$ (see figure \ref{figure6} page \pageref{figure6} and figure \ref{figure22}). The number $n_{1}$ is also the number of extremities of branched sub-arcs of $\Psi(\Gamma)$ on $[p_{2}; p_{3}]_{\Delta}$, while $n_{2}$ is the number of extremities of  branched sub-arcs of $\Psi(\Gamma)$ on $[p_{1};p_{2}]_{\Delta}$. To see this, one must simply take into account that $p_{3}$ is both green and red, and it must be considered a connection point both of a red arc and of a green one.

\setlength{\unitlength}{0.5mm}
\begin{figure}
\begin{picture}(100, 90)(-155,-45)

\put(-90,0){\line(1,0){180}}
\put(-90,-2){\makebox(0,0)[tc]{$\Delta$}}

\color{mygreen}

\put(40,0){\qbezier[300](-10,0)(-10,-4.14)(-7.07,-7.07)}
\put(40,0){\qbezier[300](-7.07,-7.07),(-4.14,-10),(0,-10)}
\put(40,0){\qbezier[300](0,-10),(4.14,-10)(7.07,-7.07)}
\put(40,0){\qbezier[300](7.07,-7.07)(10,-4.14)(10,0)}

\put(40,0){\qbezier[300](-30,0)(-30,-12.42)(-21.21,-21.21)}
\put(40,0){\qbezier[300](-21.21,-21.21)(-12.42,-30)(0,-30)}
\put(40,0){\qbezier[300](0,-30)(12.42,-30)(21.21,-21.21)}
\put(40,0){\qbezier[300](21.21,-21.21)(30,-12.42)(30,0)}

\color{red}

\put(-30,0){\qbezier[300](-20,0)(-20,-8.28)(-14.14,-14.14)}
\put(-30,0){\qbezier[300](-14.14,-14.14)(-8.28,-20)(0,-20)}
\put(-30,0){\qbezier[300](0,-20)(8.28,-20)(14.14,-14.14)}
\put(-30,0){\qbezier[300](14.14,-14.14)(20,-8.28)(20,0)}

\color{blue}

\put(10,0){\qbezier[300](-20,0)(-20,8.28)(-14.14,14.14)}
\put(10,0){\qbezier[300](-14.14,14.14)(-8.28,20)(0,20)}
\put(10,0){\qbezier[300](0,20)(8.28,20)(14.14,14.14)}
\put(10,0){\qbezier[300](14.14,14.14)(20,7,64)(20,0)}

\put(10,0){\qbezier[300](-40,0)(-40,16.56)(-28.28,28.28)}
\put(10,0){\qbezier[300](-28.28,28.28)(-16.56,40)(0,40)}
\put(10,0){\qbezier[300](0,40)(16.56,40)(28.28,28.28)}
\put(10,0){\qbezier[300](28.28,28.28)(40,16.56)(40,0)}

\put(10,0){\qbezier[300](-60,0)(-60,24.84)(-42.42,42.42)}
\put(10,0){\qbezier[300](-42.42,42.42)(-24.84,60)(0,60)}
\put(10,0){\qbezier[300](0,60)(24.84,60)(42.42,42.42)}
\put(10,0){\qbezier[300](42.42,42.42)(60,24.84)(60,0)}

\color{black}

\put(40,0){\circle*{1.5}}
\put(40,-2){\makebox(0,0)[tc]{$ p_{3}$}}

\put(50,0){\vector(0,1){0}}
\put(50,0){\vector(0,-1){0}}

\put(70,0){\vector(0,1){0}}
\put(70,0){\vector(0,-1){0}}

\put(30,0){\vector(0,1){0}}
\put(30,0){\vector(0,-1){0}}

\put(10,0){\circle*{1.5}}
\put(10,2){\makebox(0,0)[bc]{$ p_{2}$}}

\put(-10,0){\vector(0,1){0}}
\put(-10,0){\vector(0,-1){0}}

\put(-50,0){\vector(0,1){0}}
\put(-50,0){\vector(0,-1){0}}

\put(-30,0){\circle*{1.5}}
\put(-30,-2){\makebox(0,0)[tc]{$ p_{1}$}}

\put(-30,0){\vector(0,-1){0}}
\put(10,0){\vector(0,1){0}}

\end{picture}
\caption{The topological snail $\sn(3;4)$ as a puzzle of half-circles branched on $\Delta$.\label{figure19}}
\setlength{\unitlength}{0.7mm}
\end{figure}

\setlength{\unitlength}{0.5mm}
\begin{figure}
\begin{picture}(100, 120)(-155,-35)

\color{mygreen}
\put(40,-35){\line(0,1){35}}

\put(-30,-35){\line(0,1){35}}

\color{red}
\put(10,0){\line(0,1){65}}

\color{black}
\put(-31,-37){\makebox(0,0)[br]{$\scriptstyle H_{1}$}}

\put(41,-37){\makebox(0,0)[bl]{$\scriptstyle H_{3}$}}

\put(9,67){\makebox(0,0)[tr]{$\scriptstyle H_{2}$}}

\put(40,0){\circle*{1.5}}
\put(40,2){\makebox(0,0)[bc]{$p_{3}$}}

\put(40,-10){\vector(1,0){0}}
\put(40,-10){\vector(-1,0){0}}

\put(40,-30){\vector(1,0){0}}
\put(40,-30){\vector(-1,0){0}}

\put(10,0){\circle*{1.5}}
\put(9,0){\makebox(0,0)[rc]{$ p_{2}$}}

\put(10,0){\vector(0,1){0}}

\put(10,20){\vector(1,0){0}}
\put(10,20){\vector(-1,0){0}}

\put(10,40){\vector(1,0){0}}
\put(10,40){\vector(-1,0){0}}

\put(10,60){\vector(1,0){0}}
\put(10,60){\vector(-1,0){0}}

\put(-30,0){\circle*{1.5}}
\put(-31,0){\makebox(0,0)[rc]{$p_{1}$}}

\put(-30,0){\vector(0,-1){0}}
\put(-30,-20){\vector(1,0){0}}
\put(-30,-20){\vector(-1,0){0}}

\color{blue}

\put(40,0){\qbezier[300](-10,0)(-10,-4.14)(-7.07,-7.07)}
\put(40,0){\qbezier[300](-7.07,-7.07),(-4.14,-10),(0,-10)}
\put(40,0){\qbezier[300](0,-10),(4.14,-10)(7.07,-7.07)}
\put(40,0){\qbezier[300](7.07,-7.07)(10,-4.14)(10,0)}

\put(40,0){\qbezier[300](-30,0)(-30,-12.42)(-21.21,-21.21)}
\put(40,0){\qbezier[300](-21.21,-21.21)(-12.42,-30)(0,-30)}
\put(40,0){\qbezier[300](0,-30)(12.42,-30)(21.21,-21.21)}
\put(40,0){\qbezier[300](21.21,-21.21)(30,-12.42)(30,0)}

\put(-30,0){\qbezier[300](-20,0)(-20,-8.28)(-14.14,-14.14)}
\put(-30,0){\qbezier[300](-14.14,-14.14)(-8.28,-20)(0,-20)}
\put(-30,0){\qbezier[300](0,-20)(8.28,-20)(14.14,-14.14)}
\put(-30,0){\qbezier[300](14.14,-14.14)(20,-8.28)(20,0)}

\put(10,0){\qbezier[300](-20,0)(-20,8.28)(-14.14,14.14)}
\put(10,0){\qbezier[300](-14.14,14.14)(-8.28,20)(0,20)}
\put(10,0){\qbezier[300](0,20)(8.28,20)(14.14,14.14)}
\put(10,0){\qbezier[300](14.14,14.14)(20,7,64)(20,0)}

\put(10,0){\qbezier[300](-40,0)(-40,16.56)(-28.28,28.28)}
\put(10,0){\qbezier[300](-28.28,28.28)(-16.56,40)(0,40)}
\put(10,0){\qbezier[300](0,40)(16.56,40)(28.28,28.28)}
\put(10,0){\qbezier[300](28.28,28.28)(40,16.56)(40,0)}

\put(10,0){\qbezier[300](-60,0)(-60,24.84)(-42.42,42.42)}
\put(10,0){\qbezier[300](-42.42,42.42)(-24.84,60)(0,60)}
\put(10,0){\qbezier[300](0,60)(24.84,60)(42.42,42.42)}
\put(10,0){\qbezier[300](42.42,42.42)(60,24.84)(60,0)}

\end{picture}
\caption{The topological snail $\sn(3;4)$ as branched arcs on $H_{1} \cup H_{2} \cup H_{3}$.\label{figure20}}
\setlength{\unitlength}{0.7mm}
\end{figure}

\setlength{\unitlength}{0.5mm}
\begin{figure}
\begin{picture}(100,100)(-155,-35)




\color{red}

\put(40,0){\line(-1,0){30}}

\color{mygreen}
\put(10,0){\line(-1,0){40}}

\color{black}

\put(10,0){\line(0,1){65}}
\put(9,63){\makebox(0,0)[br]{$\scriptstyle H_{2}$}}

\put(10,20){\vector(1,0){0}}
\put(10,20){\vector(-1,0){0}}

\put(10,40){\vector(1,0){0}}
\put(10,40){\vector(-1,0){0}}

\put(10,60){\vector(1,0){0}}
\put(10,60){\vector(-1,0){0}}

\put(40,0){\circle*{1.5}}
\put(40,2){\makebox(0,0)[bc]{$p_{3}$}}

\put(30,0){\vector(0,1){0}}

\put(30,0){\vector(0,-1){0}}

\put(10,0){\vector(0,1){0}}

\put(10,0){\circle*{1.5}}
\put(9,-1){\makebox(0,0)[tr]{$p_{2}$}}

\put(-30,0){\circle*{1}}
\put(-30,-2){\makebox(0,0)[tc]{$p_{1}$}}

\put(-10,0){\vector(0,-1){0}}
\put(-10,0){\vector(0,1){0}}

\put(-30,0){\vector(0,-1){0}}
\color{blue}

\put(40,0){\qbezier[300](-10,0)(-10,-4.14)(-7.07,-7.07)}
\put(40,0){\qbezier[300](-7.07,-7.07),(-4.14,-10),(0,-10)}
\put(40,0){\qbezier[300](0,-10),(4.14,-10)(7.07,-7.07)}
\put(40,0){\qbezier[300](7.07,-7.07)(10,-4.14)(10,0)}

\put(40,0){\qbezier[300](-30,0)(-30,-12.42)(-21.21,-21.21)}
\put(40,0){\qbezier[300](-21.21,-21.21)(-12.42,-30)(0,-30)}
\put(40,0){\qbezier[300](0,-30)(12.42,-30)(21.21,-21.21)}
\put(40,0){\qbezier[300](21.21,-21.21)(30,-12.42)(30,0)}

\put(-30,0){\qbezier[300](-20,0)(-20,-8.28)(-14.14,-14.14)}
\put(-30,0){\qbezier[300](-14.14,-14.14)(-8.28,-20)(0,-20)}
\put(-30,0){\qbezier[300](0,-20)(8.28,-20)(14.14,-14.14)}
\put(-30,0){\qbezier[300](14.14,-14.14)(20,-8.28)(20,0)}


\put(10,0){\qbezier[300](-20,0)(-20,8.28)(-14.14,14.14)}
\put(10,0){\qbezier[300](-14.14,14.14)(-8.28,20)(0,20)}
\put(10,0){\qbezier[300](0,20)(8.28,20)(14.14,14.14)}
\put(10,0){\qbezier[300](14.14,14.14)(20,7,64)(20,0)}

\put(10,0){\qbezier[300](-40,0)(-40,16.56)(-28.28,28.28)}
\put(10,0){\qbezier[300](-28.28,28.28)(-16.56,40)(0,40)}
\put(10,0){\qbezier[300](0,40)(16.56,40)(28.28,28.28)}
\put(10,0){\qbezier[300](28.28,28.28)(40,16.56)(40,0)}

\put(10,0){\qbezier[300](-60,0)(-60,24.84)(-42.42,42.42)}
\put(10,0){\qbezier[300](-42.42,42.42)(-24.84,60)(0,60)}
\put(10,0){\qbezier[300](0,60)(24.84,60)(42.42,42.42)}
\put(10,0){\qbezier[300](42.42,42.42)(60,24.84)(60,0)}

\end{picture}
\caption{The topological snail $\sn(3;4)$ as branched arcs on $[p_{1};p_{2}]_{\Delta} \cup [p_{2};p_{3}]_{\Delta} \cup H_{3}$.\label{figure21}}
\setlength{\unitlength}{0.7mm}
\end{figure}

\setlength{\unitlength}{0.5mm}
\begin{figure}
\begin{picture}(100,110)(-155,-35)




\put(-90,0){\line(1,0){180}}
\put(-90,-2){\makebox(0,0)[tc]{$\Delta$}}

\color{red}


\color{black}

\put(40,0){\circle*{1.5}}
\put(40,-2){\makebox(0,0)[tc]{$p_{3}$}}

\put(10,0){\circle*{1.5}}
\put(10,2){\makebox(0,0)[bc]{$p_{2}$}}

\put(-30,0){\circle*{1.5}}
\put(-30,-2){\makebox(0,0)[tc]{$p_{1}$}}

\put(0,0){\qbezier[400](40,20)(50,55)(55,70)}
\put(0,0){\qbezier[400](40,20)(30,-20)(10,-20)}
\put(0,0){\qbezier[400](10,-20)(-10,-20)(-25,20)}
\put(0,0){\qbezier[400](-25,20)(-40,60)(-45,70)}

\color{blue}

\put(40,0){\qbezier[300](-10,0)(-10,-4.14)(-7.07,-7.07)}
\put(40,0){\qbezier[300](-7.07,-7.07),(-4.14,-10),(0,-10)}
\put(40,0){\qbezier[300](0,-10),(4.14,-10)(7.07,-7.07)}
\put(40,0){\qbezier[300](7.07,-7.07)(10,-4.14)(10,0)}

\put(40,0){\qbezier[300](-30,0)(-30,-12.42)(-21.21,-21.21)}
\put(40,0){\qbezier[300](-21.21,-21.21)(-12.42,-30)(0,-30)}
\put(40,0){\qbezier[300](0,-30)(12.42,-30)(21.21,-21.21)}
\put(40,0){\qbezier[300](21.21,-21.21)(30,-12.42)(30,0)}

\put(-30,0){\qbezier[300](-20,0)(-20,-8.28)(-14.14,-14.14)}
\put(-30,0){\qbezier[300](-14.14,-14.14)(-8.28,-20)(0,-20)}
\put(-30,0){\qbezier[300](0,-20)(8.28,-20)(14.14,-14.14)}
\put(-30,0){\qbezier[300](14.14,-14.14)(20,-8.28)(20,0)}

\put(10,0){\qbezier[300](-20,0)(-20,8.28)(-14.14,14.14)}
\put(10,0){\qbezier[300](-14.14,14.14)(-8.28,20)(0,20)}
\put(10,0){\qbezier[300](0,20)(8.28,20)(14.14,14.14)}
\put(10,0){\qbezier[300](14.14,14.14)(20,7,64)(20,0)}

\put(10,0){\qbezier[300](-40,0)(-40,16.56)(-28.28,28.28)}
\put(10,0){\qbezier[300](-28.28,28.28)(-16.56,40)(0,40)}
\put(10,0){\qbezier[300](0,40)(16.56,40)(28.28,28.28)}
\put(10,0){\qbezier[300](28.28,28.28)(40,16.56)(40,0)}

\put(10,0){\qbezier[300](-60,0)(-60,24.84)(-42.42,42.42)}
\put(10,0){\qbezier[300](-42.42,42.42)(-24.84,60)(0,60)}
\put(10,0){\qbezier[300](0,60)(24.84,60)(42.42,42.42)}
\put(10,0){\qbezier[300](42.42,42.42)(60,24.84)(60,0)}

\color{mygreen}
\put(47.8,46.4){\circle*{2}}
\put(41.3,24.8){\circle*{2}}
\put(31,-4.4){\circle*{2}}
\put(16.7,-18.9){\circle*{2}}

\color{red}
\put(-11.1,-7){\circle*{2}}
\put(-24.85,19.85){\circle*{2}}
\put(-33.2,41.5){\circle*{2}}

\end{picture}
\caption{A first intersection property of the topological snail $\sn(3;4)$.\label{figure22}}

\end{figure}

\setlength{\unitlength}{0.7mm}

Of course, the topological snail can be conceived for any two non zero integers $n$ and $p$. Simply consider $n+p$ distinct points on the line $\Delta^{+}$, with a first serie $x_{1} < x_{2} < \cdots < x_{n}$ on the left of $x'_{1} < x'_{2} < \cdots < x'_{p}$. Connect the points $x_{1}, x_{2}, \ldots, x_{n}$ by  half-circles in $\H^{-}$ of diameter $[x_{i}; x_{n-i}]_{\Delta}$, and the points $x'_{1}, x'_{2}, \ldots, x'_{p}$ by half-circles in $\H^{-}$ of diameter $[x'_{i}; x'_{p-i}]_{\Delta}$. Then connect the hole serie $x_{1}, \ldots , x_{n}, x'_{1}, \ldots, x_{p}$ according to the same principle by half-circles in $\H^{+}$. This  union of a finite number of half circles, organized in a simple curve and a serie of closed curves gives a picture of the positive topological snail $\sn(n;p)$. The positive topological snail is emerging in the top half plane $\H^{+}$ and sustained in $\H^{-}$. 

Let us give a more precise realization of $\sn(n;p)$. Denote $E(x)$ the integer part of any positive real number $x$. 

First fix the affix of $p_{1}$ to be $\displaystyle z_{1}= -\frac{n}{2}$. Then for an integer $i$ such that $0 \leqslant i \leqslant E\left(\frac{n -1}{2}\right)$ define the negative left sustaining half-circle in $\H^{-}$, by~:
$$\mathcal{C}_{i} = \Big\{ z \in \H^{-}~~|~~\vert z - z_{1} \vert = E\left(\frac{n}{2}\right) - \frac{n}{2}+ \frac{1}{2} + i \Big\}  $$  

Next take $z_{3} =\displaystyle  \frac{p}{2}$ to be the affix of $p_{3}$. Then for an integer $j$ such that $0 \leqslant j \leqslant E\left(\frac{p -1}{2}\right)$, define the negative right sustaining half-circle, inside $\H^{-}$ by~:
$$\mathcal{K}_{j} = \Big\{ z \in \H^{-}~~|~~\vert z - z_{3} \vert = E\left(\frac{p}{2}\right) - \frac{p}{2}+ \frac{1}{2} + j \Big\}  $$  

Finally take $z_{2} = \displaystyle \frac{-n+p}{2}$, the affix of $p_{2}$. Then for an integer  $k$ such that $ 0 \leqslant k \leqslant E\left(\frac{n+p -1}{2}\right) $, define the positive half-circle, inside $\H^{+}$, by~:
$$\mathcal{Q}_{k} = \Big\{ z \in \H^{+}~~|~~\vert z - z_{2} \vert = E\left(\frac{n+p}{2}\right) - \frac{n+p}{2}+ \frac{1}{2} + k \Big\}  $$  

One obtains a puzzle of circles as in figure \ref{figure19}, where I draw a picture of the topological snail $\sn(3;4)$ with parameters $n=3$ and $p = 4$. In the following $\sn(n;p)$ will denote both the concrete topological snail as a geometric subset of $\R^{2}$ with the set of marked points $X= \lbrace p_{1}; p_{2}; p_{3} \rbrace$ and the isotopy class of this concrete curve relatively to $X$ moving on the line $\Delta^{+}$. If $H$ and $K$ are subsets of $\R^{2}$ and if tthere is an isotopy $(f_{t})_{t\in [0;1]}$ with $f_{0} = \id$; $f_{1} = f$, $K=f(X)$, $H= f\left(\sn(n;p)\right)$ and $f_{t}(X) \subset \Delta^{+}$ for all $t\in [0;1]$, then I will say that $H$ has the type of $\sn(n;p)$ relatively to $X$ and $\Delta$, and simply write $H \equiv \sn(n;p)$ is there is no possible confusion relative to $X$.

Among the three points $p_{1}$; $p_{2}$; $p_{3}$, there are two extremities of $\sn(n;p)$ that I will call the base points, and a third point that I will call the free point in $\sn(n;p)$. The free point is the center of a half-disk with a half-unit radius oriented from the entrance to the interior of the half-disk that I call liberty's chamber. There is a very natural path from $\infty$ into liberty's chamber that I call the central path and note $\cp^{+}(n;p)$. I admit here its precise analytical definition, very close from that of $\sn(n;p)$ (take $O$, $p-1$ integers on the left, $n-1$ on the right, and connect suitably by arcs of circle and a half-line to obtain $\cp^{+}(n;p)$). The origin $O$ is another important point that I call the entrance of the snail $\sn(n;p)$ (see figure \ref{figure23}, with the central access to liberty's chamber in green color). Similarily, I will define the shell of $\sn(n;p)$ to be the union of the visible emerging arc and left and right sustaining arcs.

\section{The simple topological snail \label{simple}}

I will now consider two positive integer $n$, $p$, and the topological snail $\sn(n;p)$ defined at the previous section \ref{pictures} : see figure \ref{figure23}. The topological snail appears as a labyrinth with an entrance at the origine 0, and it seems very natural to explore it from the entrance, following the successive half-circles of the central access to liberty's chamber $\cp^{+}(n;p)$. It is useful for this purpose to undestand the sequence of the succesive extremities of the diameters of those circles. If for $x\in \R$ let us define the respective symetries around the points $p_{1}$, $p_{2}$ and $p_{3}$ by $\varsigma_{1}(x) = -n-x $, $\varsigma_{2}(x) = -n+p-x $, and $\varsigma_{3}(x) = p-x $, then the path comming from the origin $O$ successively visit the points $\varsigma_{2}(0)$, $\varsigma_{1} \circ \varsigma_{2} (0)$, $\varsigma_{2} \circ \varsigma_{1} \circ \varsigma_{2} (0)$,  $\varsigma_{3} \circ \varsigma_{2} \circ \varsigma_{1} \circ \varsigma_{2} (0)$, \ldots. If $n<p$, then the first move  $O$ is on the right and I will say that the weak side of $\sn(n;p)$ is on the left, and its strong side on the right. I will generally considers snails in this canonical position. If one considers the visitor's path to liberty's chamber,

Naturally the translations $\tau_{1}(x) = \varsigma_{1} \circ \varsigma_{2}(x) = x-p$ and $\tau_{2} = \varsigma_{2} \circ \varsigma_{2}(x) = x+n$ respectively act on $[0;p]$ and $[-n;0]$. One obtains a classical interval exchange map, and in this case it is particularily simple~: both transformations $\tau_{1}$ and $\tau_{2}$ come from a single rotation of an angle $n \equiv -p~~~[n+p]$  on the circle $\R/(n+p)\Z$. This transformation is periodic and its period is the least common multiple $n \vee p$ of $n$ and $p$. The greatest common divisor $n \wedge p$ is also the size of the largest open interval with exatcly $n \vee p$ disjoints images. To find such a generating interval, on can use the euclidean algorithm. It has a very nice interpretation in the topological snail and a great importance for the study of turbulent homeomorphisms. This is why I will give some technical precisions about this algorithm and its formulation. The notations defined in what follows are usual and appear in many different contexts, a particularily important one being the theory of continued fractions. From my perspective, I will call the sequence of the $q_{i}'$ the sequence of rests, the sequence of $u_{i}'$ the sequence of orders, and the sequence of $a_{i}'$ the sequence of the coefficients. Note that if you take that the orders for the rests, the sequence of the coefficients appears in reverse order (see figure \ref{figure25}).

Suppose that $q_{0}=p$ and $q_{1}=n$ are two strictly positive and different integers. Then there is a unique maximal integer $a_{0}\geqslant 0$ and a unique strictly positive integer $q_{2}$ such that 
$$q_{0} = a_{0} q_{1} + q_{2}~~~\mathrm{with}~~~q_{1} \geqslant q_{2} >0 $$

If $q_{0}<q_{1}$, then necessarily $a_{0} = 0$, $q_{2} = q_{0}$ and $q_{1}>q_{2}$. If $q_{0} = k  q_{1}$ is a multiple of $q_{1}$ with multiplicator $k>1$, then one has  $a_{0} = k-1$ and $q_{2} = q_{1}$. Iterating the process successively defines sequences of strictly positive integers $a_{0}; a_{1}; \ldots ; a_{K}$ and $q_{0}; q_{1}; \ldots ; q_{K}$ with for all $i$ :
\begin{eqnarray} q_{i} = a_{i} q_{i+1} + q_{i+2}~~~\mathrm{with}~~~q_{i+1} \geqslant q_{i+2} >0 \end{eqnarray}
This process must stop for some $i=K\geqslant 0$, when precisely $q_{i+1} = q_{i+2}$ for the first time. In this case $q_{K+1} = q_{K+2}$ is the greatest common divisor of any two of the numbers $q_{0}; q_{1}; \ldots ; q_{K}$. Fix $d = q_{0}\wedge q_{1}$. Then write for $0 \leqslant i \leqslant K$
\begin{eqnarray} \left[\begin{array}{c} q_{i} \\ q_{i+1} \end{array}\right] = \left[ \begin{array}{cc} a_{i} & 1 \\ 1 & 0 \end{array} \right]  \left[\begin{array}{c} q_{i+1} \\ q_{i+2} \end{array}\right] \end{eqnarray} 
and for all $0\leqslant i \leqslant K$~:
\begin{eqnarray}  \left[\begin{array}{c} q_{i} \\ q_{i+1} \end{array}\right] = \left[ \begin{array}{cc} a_{i} & 1 \\ 1 & 0 \end{array} \right]  \left[ \begin{array}{cc} a_{i+1} & 1 \\ 1 & 0 \end{array} \right]  \cdots \left[ \begin{array}{cc} a_{K} & 1 \\ 1 & 0 \end{array} \right]  \left[\begin{array}{c} d \\ d \end{array}\right] \end{eqnarray} 

On has in particular 

 \begin{eqnarray} \left[\begin{array}{c} q_{0} \\ q_{1} \end{array}\right] = \left[ \begin{array}{cc} a_{0} & 1 \\ 1 & 0 \end{array} \right]  \left[ \begin{array}{cc} a_{1} & 1 \\ 1 & 0 \end{array} \right]  \cdots \left[ \begin{array}{cc} a_{K} & 1 \\ 1 & 0 \end{array} \right]  \left[\begin{array}{c} d \\ d \end{array}\right]  \label{restes} \end{eqnarray}

Note that in this formula the length $N$ and the sequence of coefficients $(a_{0}; a_{1}; \cdots ; a_{K})$ are uniquely determined by $q_{0} = p$ and $q_{1} = n$, with the condition that all the coefficient $a_{i}$ are positive and strictly positive for $i\geqslant 1$. The fact that $a_{0}=0$ is simply equivalent to a permutation of $n$ and $p$ when $n>p$. One in particular has 
$$q_{0}<q_{1} \Longleftrightarrow a_{0} > 0\,. $$

One can also remark that 
\begin{eqnarray} \left[\begin{array}{c} d \\ d \end{array}\right] = \left[ \begin{array}{cc} 0 & 1 \\ 1 & 0 \end{array} \right] \left[\begin{array}{c} d \\ d \end{array}\right]  \label{restes} \end{eqnarray}
This implies that if $K = 2N$ is even, one can still set $a_{K+1}= a_{2N+1} = 0$ to obtain a unique completely symetric writing of $(n;p)$ under the equation 
 \begin{eqnarray} \left[\begin{array}{c} p \\ n \end{array}\right] =  \left[\begin{array}{c} q_{0} \\ q_{1} \end{array}\right] = \left[ \begin{array}{cc} a_{0} & 1 \\ 1 & 0 \end{array} \right]  \left[ \begin{array}{cc} a_{1} & 1 \\ 1 & 0 \end{array} \right]  \cdots \left[ \begin{array}{cc} a_{2N+1} & 1 \\ 1 & 0 \end{array} \right]  \left[\begin{array}{c} d \\ d \end{array}\right]  \label{restes} \end{eqnarray}
with all the $a_{i}'s$ strictly positive except possibly for $a_{0}=0$ and $a_{2N+1}=0$.

For any number $l\in \Z$, let us define the translation $\tau_{l}$  in $\Z/ (n+p)\Z$ by
$$ \tau_{l}(x) \equiv x+l~~~[n+k] $$ 
Let us fix $f= \tau_{k} = \tau_{-p}$, with $q_{0}= n$ and $q_{1}=p$; set $u_{1} = u_{0}=1$ and for $0 \leqslant i \leqslant 2N+1$, define
$$u_{i+2} = a_{i} u_{i+1} + u_{i} $$   
then the iterate $f^{u_{K}}$ is equal to $\tau_{(-1)^{N}u_{K}}$ (see figure \ref{figure25} and figure \ref{figure26}).
But one can also write 

 \begin{eqnarray}\left[\begin{array}{cc} u_{i+2} & u_{i+1} \end{array}\right] =  \left[\begin{array}{cc} u_{i+1} & u_{i} \end{array}\right] \left[ \begin{array}{cc} a_{i} & 1 \\ 1 & 0 \end{array} \right] \end{eqnarray}

and

 \begin{eqnarray}\left[\begin{array}{cc} u_{2N+3} & u_{2N+2} \end{array}\right] & = & \left[\begin{array}{cc} u_{0} & u_{1} \end{array}\right] \left[ \begin{array}{cc} a_{0} & 1 \\ 1 & 0 \end{array} \right]  \left[ \begin{array}{cc} a_{1} & 1 \\ 1 & 0 \end{array} \right]  \cdots \left[ \begin{array}{cc} a_{2N+1} & 1 \\ 1 & 0 \end{array} \right] \\
& = &\left[\begin{array}{cc} ~1~ & 1~ \end{array}\right] \left[ \begin{array}{cc} a_{0} & 1 \\ 1 & 0 \end{array} \right]  \left[ \begin{array}{cc} a_{1} & 1 \\ 1 & 0 \end{array} \right]  \cdots \left[ \begin{array}{cc} a_{2N+1} & 1 \\ 1 & 0 \end{array} \right]  \label{ordres}\end{eqnarray}

Now consider the matrix 
 \begin{eqnarray} \left[ \begin{array}{cc} r_{2} & g_{2} \\ r_{1} & g_{1} \end{array} \right] = \left[ \begin{array}{cc} a_{0} & 1 \\ 1 & 0 \end{array} \right]  \left[ \begin{array}{cc} a_{1} & 1 \\ 1 & 0 \end{array} \right]  \cdots  \left[ \begin{array}{cc} a_{2N+1} & 1 \\ 1 & 0 \end{array} \right]\end{eqnarray}


\begin{figure}
\begin{picture}(100,90)(-105,-45)




\color{mygreen}

\put(0,-30){\line(0,1){30}}

\put(0,-15){\vector(0,1){0}}
\put(10,0){\qbezier[300](-10,0)(-10,4.14)(-7.07,7.07)}
\put(10,0){\qbezier[300](-7.07,7.07),(-4.14,10),(0,10)}
\put(10,0){\qbezier[300](0,10),(4.14,10)(7.07,7.07)}
\put(10,0){\qbezier[300](7.07,7.07)(10,4.14)(10,0)}

\put(40,0){\qbezier[300](-20,0)(-20,-8.28)(-14.14,-14.14)}
\put(40,0){\qbezier[300](-14.14,-14.14)(-8.28,-20)(0,-20)}
\put(40,0){\qbezier[300](0,-20)(8.28,-20)(14.14,-14.14)}
\put(40,0){\qbezier[300](14.14,-14.14)(20,-8.28)(20,0)}

\put(10,0){\qbezier[300](-50,0)(-50,20.7)(-35.35,35.35)}
\put(10,0){\qbezier[300](-35.35,35.35),(-20.7,50),(0,50)}
\put(10,0){\qbezier[300](0,50),(20.7,50)(35.35,35.35)}
\put(10,0){\qbezier[300](35.35,35.35)(50,20.7)(50,0)}

\put(-30,0){\qbezier[300](-10,0)(-10,-4.14)(-7.07,-7.07)}
\put(-30,0){\qbezier[300](-7.07,-7.07),(-4.14,-10),(0,-10)}
\put(-30,0){\qbezier[300](0,-10),(4.14,-10)(7.07,-7.07)}
\put(-30,0){\qbezier[300](7.07,-7.07)(10,-4.14)(10,0)}

\put(10,0){\qbezier[300](-30,0)(-30,12.42)(-21.21,21.21)}
\put(10,0){\qbezier[300](-21.21,21.21)(-12.42,30)(0,30)}
\put(10,0){\qbezier[300](0,30)(12.42,30)(21.21,21.21)}
\put(10,0){\qbezier[300](21.21,21.21)(30,12.42)(30,0)}

\put(40,0){\vector(0,-1){0}}

\put(0,-30){\line(0,1){30}}

\put(0,-15){\vector(0,1){0}}

\color{blue}

\put(40,0){\qbezier[300](-10,0)(-10,-4.14)(-7.07,-7.07)}
\put(40,0){\qbezier[300](-7.07,-7.07),(-4.14,-10),(0,-10)}
\put(40,0){\qbezier[300](0,-10),(4.14,-10)(7.07,-7.07)}
\put(40,0){\qbezier[300](7.07,-7.07)(10,-4.14)(10,0)}

\put(40,0){\qbezier[300](-30,0)(-30,-12.42)(-21.21,-21.21)}
\put(40,0){\qbezier[300](-21.21,-21.21)(-12.42,-30)(0,-30)}
\put(40,0){\qbezier[300](0,-30)(12.42,-30)(21.21,-21.21)}
\put(40,0){\qbezier[300](21.21,-21.21)(30,-12.42)(30,0)}

\put(-30,0){\qbezier[300](-20,0)(-20,-8.28)(-14.14,-14.14)}
\put(-30,0){\qbezier[300](-14.14,-14.14)(-8.28,-20)(0,-20)}
\put(-30,0){\qbezier[300](0,-20)(8.28,-20)(14.14,-14.14)}
\put(-30,0){\qbezier[300](14.14,-14.14)(20,-8.28)(20,0)}

\put(10,0){\qbezier[300](-20,0)(-20,8.28)(-14.14,14.14)}
\put(10,0){\qbezier[300](-14.14,14.14)(-8.28,20)(0,20)}
\put(10,0){\qbezier[300](0,20)(8.28,20)(14.14,14.14)}
\put(10,0){\qbezier[300](14.14,14.14)(20,7,64)(20,0)}

\put(10,0){\qbezier[300](-40,0)(-40,16.56)(-28.28,28.28)}
\put(10,0){\qbezier[300](-28.28,28.28)(-16.56,40)(0,40)}
\put(10,0){\qbezier[300](0,40)(16.56,40)(28.28,28.28)}
\put(10,0){\qbezier[300](28.28,28.28)(40,16.56)(40,0)}

\put(10,0){\qbezier[300](-60,0)(-60,24.84)(-42.42,42.42)}
\put(10,0){\qbezier[300](-42.42,42.42)(-24.84,60)(0,60)}
\put(10,0){\qbezier[300](0,60)(24.84,60)(42.42,42.42)}
\put(10,0){\qbezier[300](42.42,42.42)(60,24.84)(60,0)}

\color{black}

\put(40,0){\circle*{1.5}}

\put(40,-2){\makebox(0,0)[tc]{$\frac{p}{2}$}}

\put(10,0){\circle*{1.5}}
\put(10,2){\makebox(0,0)[bc]{$\frac{-n+p}{2}$}}

\put(-30,0){\circle*{1.5}}
\put(-30,-2){\makebox(0,0)[tc]{$\frac{-n}{2}$}}

\put(0,0){\circle*{1}}
\put(-1,1){\makebox(0,0)[br]{$0$}}

\put(-90,0){\line(1,0){180}}
\put(-90,-2){\makebox(0,0)[tc]{$\Delta$}}

\put(-40,45){\makebox(0,0)[tc]{$\Gamma$}}

\end{picture}
\caption{The topological snail $\sn(n;p)$ with $\cp^{+}(n;p)$ in green color for $n=3$ and $p=4$.\label{figure23}}

\end{figure}


\begin{figure}
\begin{picture}(100,100)(-105,-55)

\put(-90,0){\line(1,0){180}}
\put(90,0){\vector(1,0){0}}

\put(-90,-2){\makebox(0,0)[tc]{$\Delta$}}
\put(0,2){\makebox(0,0)[bc]{$\ 0$}}

\color{black}

\put(60,20){\vector(-1,0){80}}

\put(20,22){\makebox(0,0)[bc]{$-p$}}

\put(60,0){\circle*{1}}
\put(-50,0){\circle*{1}}

\put(10,-20){\qbezier[25](0,0)(0,10)(0,20)}
\put(-50,-20){\qbezier[25](0,0)(0,10)(0,20)}

\put(60,0){\qbezier[25](0,0)(0,10)(0,20)}
\put(-20,0){\qbezier[25](0,0)(0,10)(0,20)}

\put(-50,-20){\vector(1,0){60}}

\put(-20,-18){\makebox(0,0)[bc]{$\ +n$}}

\put(80,-50){\qbezier[50](0,0)(0,25)(0,50)}
\put(0,-50){\qbezier[50](0,0)(0,25)(0,50)}
\put(-60,-50){\qbezier[50](0,0)(0,25)(0,50)}

\put(79,-50){\line(-1,0){78}}
\put(80,-50){\vector(1,0){0}}
\put(40,-48){\makebox(0,0)[bc]{$\ p$}}
\put(1,-50){\vector(-1,0){0}}
\put(-1,-50){\vector(1,0){0}}
\put(-59,-50){\vector(-1,0){0}}
\put(-30,-48){\makebox(0,0)[bc]{$\ n$}}
\put(-59,-50){\line(1,0){58}}

\color{mygreen}

\put(10,0){\qbezier[300](-50,0)(-50,20.7)(-35.35,35.35)}
\put(10,0){\qbezier[300](-35.35,35.35),(-20.7,50),(0,50)}
\put(10,0){\qbezier[300](0,50),(20.7,50)(35.35,35.35)}
\put(10,0){\qbezier[300](35.35,35.35)(50,20.7)(50,0)}

\put(-30,0){\qbezier[300](-10,0)(-10,-4.14)(-7.07,-7.07)}
\put(-30,0){\qbezier[300](-7.07,-7.07),(-4.14,-10),(0,-10)}
\put(-30,0){\qbezier[300](0,-10),(4.14,-10)(7.07,-7.07)}
\put(-30,0){\qbezier[300](7.07,-7.07)(10,-4.14)(10,0)}

\put(-20,0){\vector(0,1){0}}

\color{blue}

\put(40,0){\qbezier[100](-30,0)(-30,-12.42)(-21.21,-21.21)}
\put(40,0){\qbezier[100](-21.21,-21.21)(-12.42,-30)(0,-30)}
\put(40,0){\qbezier[100](0,-30)(12.42,-30)(21.21,-21.21)}
\put(40,0){\qbezier[100](21.21,-21.21)(30,-12.42)(30,0)}

\put(10,0){\vector(0,1){0}}

\put(10,0){\qbezier[300](-60,0)(-60,24.84)(-42.42,42.42)}
\put(10,0){\qbezier[300](-42.42,42.42)(-24.84,60)(0,60)}
\put(10,0){\qbezier[300](0,60)(24.84,60)(42.42,42.42)}
\put(10,0){\qbezier[300](42.42,42.42)(60,24.84)(60,0)}

\color{black}

\put(40,0){\qbezier[300](-40,0)(-40,-16.56)(-28.28,-28.28)}
\put(40,0){\qbezier[300](-28.28,-28.28)(-16.56,-40)(0,-40)}
\put(40,0){\qbezier[300](0,-40)(16.56,-40)(28.28,-28.28)}
\put(40,0){\qbezier[300](28.28,-28.28)(40,-16.5)(40,0)}

\put(-30,0){\qbezier[300](-30,0)(-30,-12.42)(-21.21,-21.21)}
\put(-30,0){\qbezier[300](-21.21,-21.21)(-12.42,-30)(0,-30)}
\put(-30,0){\qbezier[300](0,-30)(12.42,-30)(21.21,-21.21)}
\put(-30,0){\qbezier[300](21.21,-21.21)(30,-12.42)(30,0)}


\put(10,0){\qbezier[300](-70,0)(-70,28.98)(-49.49,49.49)}
\put(10,0){\qbezier[300](-49.49,49.49)(-28.98,70)(0,70)}
\put(10,0){\qbezier[300](0,70)(28.98,70)(49.49,49.49)}
\put(10,0){\qbezier[300](49.49,49.49)(70,28.98)(70,0)}

\end{picture}
\caption{Suspension of an interval exchange map as the dynamical structure of the snail.\label{figure24}}

\end{figure}

\begin{figure}
\begin{picture}(100,80)(-60,-40)

\put(0,-42){\line(0,1){108}}
\put(-42,-42){\line(0,1){108}}
\put(156,-42){\line(0,1){108}}

\put(0,60){\vector(-1,0){42}}
\put(-21,61){\makebox(0,0)[bc]{$\scriptstyle 21$}}
\put(-21,59){\makebox(0,0)[tc]{$\scriptstyle f$}}

\put(0,60){\vector(1,0){156}}
\put(78,61){\makebox(0,0)[bc]{$\scriptstyle 78$}}
\put(78,59){\makebox(0,0)[tc]{$\scriptstyle f$}}

\put(-42,54){\line(1,0){198}}
\put(-42,66){\line(1,0){198}}

\put(156,48){\vector(-1,0){42}}
\put(135,49){\makebox(0,0)[bc]{$\scriptstyle 21$}}
\put(135,47){\makebox(0,0)[tc]{$\scriptstyle f$}}

\put(114,48){\vector(-1,0){42}}
\put(93,49){\makebox(0,0)[bc]{$\scriptstyle 21$}}
\put(93,47){\makebox(0,0)[tc]{$\scriptstyle f$}}

\put(51,49){\makebox(0,0)[bc]{$\scriptstyle 21$}}
\put(51,47){\makebox(0,0)[tc]{$\scriptstyle f$}}
\put(72,48){\vector(-1,0){42}}

\put(0,48){\vector(1,0){30}}
\put(15,47){\makebox(0,0)[tc]{$\scriptstyle f^{4}$}}
\put(15,49){\makebox(0,0)[bc]{$\scriptstyle 15$}}

\put(0,48){\vector(-1,0){42}}
\put(-21,49){\makebox(0,0)[bc]{$\scriptstyle 21$}}
\put(-21,47){\makebox(0,0)[tc]{$\scriptstyle f$}}

\put(-42,42){\line(1,0){198}}

\put(30,42){\line(0,1){12}}
\put(72,42){\line(0,1){12}}
\put(114,42){\line(0,1){12}}


\put(0,36){\vector(1,0){30}}
\put(16,35){\makebox(0,0)[tc]{$\scriptstyle f^{4}$}}
\put(16,37){\makebox(0,0)[bc]{$\scriptstyle 15$}}

\put(0,36){\vector(-1,0){42}}
\put(-21,37){\makebox(0,0)[bc]{$\scriptstyle 21$}}
\put(-21,35){\makebox(0,0)[tc]{$\scriptstyle f$}}

\put(35,36){\makebox(0,0)[l]{$78 = 3 \times 21 + 15 $}}

\put(85,36){\makebox(0,0)[l]{$4 = 3 \times 1 + 1 $}}

\put(135,36){\makebox(0,0)[l]{$a_{0}=3$}}

\put(30,30){\line(0,1){12}}

\put(-42,30){\line(1,0){72}}



\put(35,24){\makebox(0,0)[l]{$21 = 1 \times 15 + 6 $}}

\put(85,24){\makebox(0,0)[l]{$5 = 1 \times 4 + 1 $}}

\put(135,24){\makebox(0,0)[l]{$a_{1}=1$}}

\put(-42,24){\vector(1,0){30}}
\put(-25,23){\makebox(0,0)[tc]{$\scriptstyle f^{4}$}}
\put(-25,25){\makebox(0,0)[bc]{$\scriptstyle 15$}}

\put(0,24){\vector(-1,0){12}}
\put(-6,23){\makebox(0,0)[tc]{$\scriptstyle f^{5}$}}
\put(-6,25){\makebox(0,0)[bc]{$\scriptstyle 6$}}

\put(-12,18){\line(0,1){12}}

\put(0,24){\vector(1,0){30}}
\put(15,23){\makebox(0,0)[tc]{$\scriptstyle f^{4}$}}
\put(15,25){\makebox(0,0)[bc]{$\scriptstyle 15$}}
\put(30,18){\line(0,1){12}}

\put(-42,18){\line(1,0){72}}


\put(35,12){\makebox(0,0)[l]{$15 = 2 \times 6 + 3 $}}

\put(85,12){\makebox(0,0)[l]{$14 = 2 \times 5 + 4 $}}

\put(135,12){\makebox(0,0)[l]{$a_{2}=2$}}


\put(0,12){\vector(1,0){30}}
\put(15,11){\makebox(0,0)[tc]{$\scriptstyle f^{4}$}}
\put(15,13){\makebox(0,0)[bc]{$\scriptstyle 15$}}

\put(0,12){\vector(-1,0){12}}
\put(-6,11){\makebox(0,0)[tc]{$\scriptstyle f^{5}$}}
\put(-6,13){\makebox(0,0)[bc]{$\scriptstyle 6$}}

\put(-12,6){\line(1,0){42}}

\put(-12,6){\line(0,1){12}}
\put(30,6){\line(0,1){12}}


\put(35,0){\makebox(0,0)[l]{$6 = 1 \times 3 + 3 $}}

\put(85,0){\makebox(0,0)[l]{$19 = 1 \times 14 + 5 $}}

\put(135,0){\makebox(0,0)[l]{$a_{3}=1$}}

\put(0,0){\vector(-1,0){12}}
\put(-6,-1){\makebox(0,0)[tc]{$\scriptstyle f^{5}$}}
\put(-6,1){\makebox(0,0)[bc]{$\scriptstyle 6$}}
\put(-12,-6){\line(0,1){12}}

\put(0,0){\vector(1,0){6}}
\put(3,-1){\makebox(0,0)[tc]{$\scriptstyle f^{14}$}}
\put(3,1){\makebox(0,0)[bc]{$\scriptstyle 3$}}
\put(6,-6){\line(0,1){12}}

\put(30,0){\vector(-1,0){12}}
\put(24,-1){\makebox(0,0)[tc]{$\scriptstyle f^{5}$}}
\put(24,1){\makebox(0,0)[bc]{$\scriptstyle 6$}}
\put(18,-6){\line(0,1){12}}

\put(-12,-6){\line(1,0){42}}

\put(30,-6){\line(0,1){12}}

\put(18,0){\vector(-1,0){12}}
\put(12,-1){\makebox(0,0)[tc]{$\scriptstyle f^{5}$}}
\put(12,1){\makebox(0,0)[bc]{$\scriptstyle 6$}}



\put(0,-12){\vector(1,0){6}}
\put(3,-13){\makebox(0,0)[tc]{$\scriptstyle f^{14}$}}
\put(3,-11){\makebox(0,0)[bc]{$\scriptstyle 3$}}
\put(6,-18){\line(0,1){12}}

\put(0,-12){\vector(-1,0){12}}
\put(-6,-13){\makebox(0,0)[tc]{$\scriptstyle f^{5}$}}
\put(-6,-11){\makebox(0,0)[bc]{$\scriptstyle 6$}}
\put(-12,-18){\line(0,1){12}}

\put(-12,-18){\line(1,0){18}}


\put(0,-24){\vector(1,0){6}}
\put(3,-25){\makebox(0,0)[tc]{$\scriptstyle f^{14}$}}
\put(3,-23){\makebox(0,0)[bc]{$\scriptstyle 3$}}
\put(6,-30){\line(0,1){12}}

\put(-12,-24){\vector(1,0){6}}
\put(-9,-25){\makebox(0,0)[tc]{$\scriptstyle f^{14}$}}
\put(-9,-23){\makebox(0,0)[bc]{$\scriptstyle 3$}}
\put(-12,-30){\line(0,1){12}}

\put(0,-24){\vector(-1,0){6}}
\put(-3,-25){\makebox(0,0)[tc]{$\scriptstyle f^{19}$}}
\put(-3,-23){\makebox(0,0)[bc]{$\scriptstyle 3$}}
\put(-6,-30){\line(0,1){12}}

\put(-12,-30){\line(1,0){18}}



\put(0,-36){\vector(-1,0){6}}
\put(-3,-37){\makebox(0,0)[tc]{$\scriptstyle f^{14}$}}
\put(-3,-35){\makebox(0,0)[bc]{$\scriptstyle 3$}}
\put(-6,-42){\line(0,1){12}}

\put(0,-36){\vector(1,0){6}}
\put(3,-37){\makebox(0,0)[tc]{$\scriptstyle f^{19}$}}
\put(3,-35){\makebox(0,0)[bc]{$\scriptstyle 3$}}
\put(6,-42){\line(0,1){12}}

\put(-42,-42){\line(1,0){198}}

\color{blue}

\end{picture}
\caption{Euclidean algorithm modulo $q_{1}+ q_{0}=99$ for $q_{0}=21$ and $q_{1} = 78$ .\label{figure25}}

\end{figure}

From the equation (\ref{restes}) one has $(g_{1} + r_{1})  d = q_{0}$ and 
$(g_{2}+ r_{2})  d = q_{1}$, and from the equation (\ref{ordres}) it comes $g_{1}+ g_{2} = u_{2N+2}$ and $r_{1}+ r_{2} = u_{2N+3}$. So finally 
\begin{eqnarray} (u_{2N+3} + u_{2N+2}) d = q_{0} + q_{1} \end{eqnarray}

One can also remark that since $d= q_{0} \wedge q_{1} = q_{0} \wedge (q_{0}+q_{1})$, the least common multiple $q_{0} \vee (q_{0}+q_{1})$ of $q_{0}$ and $q_{0}+q_{1}$ is $\displaystyle \frac{q_{0}  (q_{0}+q_{1})}{d} = q_{0}  (u_{2N+3} + u_{2N+2})$.

 Hence the translation $\tau_{d}$ and $f= \tau_{q_{0}}= \tau_{-q_{1}}$ are both periodic of period $\frac{q_{0}+q_{1}}{d} = u_{2N+3} + u_{2N+2}$. Any interval of length $d$ is translated exactly to itself after $u_{2N+3} + u_{2N+2}$ iterations, coming back after having recovered the hole of $\R/(q_{0}+q_{1})\R$.

 In the topological snail $\sn(n;p)$ in weak position ($n<p$), it means that if I consider the interval $I$ located on the left of the entrance point $0$ and with size $n\wedge p$, then it will come back adjacent to itself after $u_{2N+2}$ iteration of the rotation $f$, on the right if and only if $a_{2N+1}$ is even (see figures \ref{figure23}, \ref{figure24} and \ref{figure25}), and on the left otherwise. Note that considering the topological snail as the suspension of a rotation (figure \ref{figure24}) leads to many intersesting considerations. The curves obtained by gluing together all the emerging and sustaining half circles split into three categories (see figure \ref{figure27}) : the regular curves, homeomorphic to a circle, and obtained by gluing together true arcs of circles, the connected component of the visible emerging half circle, that contains one singular circle located on the center of one of the three half-discs, and the connected component of the other two centers. In particular, the snail $\sn(n;p)$ is connected if and only if $n$ and $p$ are relatively prime. In this case $\sn(n;p)$ is a simple curve connecting the centers of the two odd diameters of the half-circles that compose $\sn(n;p)$ when $n$ and $p$ are relativey prime (see figure \ref{figure24} and \ref{figure27}). 
 
 But there is another important point of view that will prove itself very fructuous from a dynamical point of view, and gives a second justification to the particular formulation of the euclidean algorithm used here. The topological snail $\sn(n;p)$ is made of three half-disks : the left sustaining half-disk of diameter $n$, the right sustaining half-disk of diameter $p$, and the emerging half-disk of diameter $n+p$. And there are natural operations associated with this consideration : one can rotate, fold or unfold the snail according to simple geometric operations. Let us focus for the moment on the folding-unfolding operations : the top-unfolding operation of the type $A$ applies to a topological snail in top-position and consists in replacing the left sustaining disk of diameter $n$ by the emerging disk of diameter $n+p$. The analogous operation of type $B$ consists in replacing the right sustaining disk by the emerging one. Let $A$ and $B$ be the applications on numbers associated to the preceeding operation : $A(n;p) = (n+p; p)$ and $B(n;p) = (n;n+p)$. Of course, any couple $(n,p)$ is obtained from $(1;1)$ by a unique sequence of unfolding operations of the type $A$ and $B$, repeated as many times as possible and necessary. This leads to a new way of writing the operations succesively, that has a very nice interpretation for the study of homeomorphism of the plane, and is clear in itself (see figure \ref{figure28}). The first theorem is the following~:

 \begin{theo} \label{caracteristic}
 For any couple $(n;p)$ of relatively prime natural numbers, there is a unique natural number $N>0$ and two unique caracteristic sequences of positive natural numbers $(\alpha_{1},\alpha_{2}, \ldots , \alpha_{N})$ and  $(\beta_{1},\beta_{2}, \ldots , \beta_{N})$ such that : 

 \begin{enumerate}
 \item $\alpha_{i} >0$ except possibly for $i=1$ where $\alpha_{1}=0$ is possible.
 \item $\beta_{i} >0$ except possibly for $i=N$ where $\beta_{N}=0$ is possible.
 \item $\displaystyle (n;p) = B^{\, \beta_{N}} \circ A^{\alpha_{N}} \circ  \cdots  \circ B^{\, \beta_{1}} \circ A^{\alpha_{1}}(1;1)$ 
 \end{enumerate}
 
  \end{theo}
  
\begin{remarks} 

\item If I say that $n$ and $p$ are relatively prime if and only if their only positive common divisor is one, then $(1;1)$ is admissible for the theorem. And in this case $N=1$ and $\alpha_{1} = \beta_{1}=0$ is convenient. Of course, this case is the basis for a general and very self-evident direct proof of the theorem.

\item Recall that $A$ and $B$ also denote the matrices
$$A=\left[\begin{array}{cc} 1 & 1 \\ 0 & 1 \end{array}\right]~~~~B=\left[\begin{array}{cc} 1 & 0 \\ 1 & 1 \end{array}\right]$$
and define

$$\mt(n;p) = B^{\, \beta_{N}} \circ A^{\alpha_{N}} \circ  \cdots  \circ B^{\, \beta_{1}} \circ A^{\alpha_{1}} $$
 
 Then 
 
\begin{eqnarray} \mt(n;p) & = & \left[\begin{array}{cc} 1 & 0 \\ \beta_{N} & 1 \end{array}\right]\left[\begin{array}{cc} 1 & \alpha_{N} \\ 0 & 1 \end{array}\right]\cdots \left[\begin{array}{cc} 1 & 0 \\ \beta_{1} & 1 \end{array}\right]\left[\begin{array}{cc} 1 & \alpha_{1} \\ 0 & 1 \end{array}\right] \\
& = &  \left[\begin{array}{cc} 0 & 1 \\ 1 & 0 \end{array}\right] \left[ \begin{array}{cc} a_{0}\, & 1 \\ 1 & 0 \end{array} \right]  \left[ \begin{array}{cc} a_{1}\, & 1\, \\ 1 & 0 \end{array} \right]  \cdots  \left[ \begin{array}{cc} a_{2N+1} & 1 \\ 1 & 0 \end{array} \right] \left[\begin{array}{cc} 0 & 1 \\ 1 & 0 \end{array}\right] \\
& = & \left[ \begin{array}{cc} g_{1} & r_{1} \\ g_{2} & r_{2} \end{array} \right]  \label{turbulence}
\end{eqnarray}

To understand this, just use the equality~:

\begin{eqnarray}  \left[\begin{array}{cc} 0 & 1 \\ 1 & 0 \end{array}\right]\left[\begin{array}{cc} 1 & 0 \\ \beta & 1 \end{array}\right]\left[\begin{array}{cc} 1 & \alpha \\ 0 & 1 \end{array}\right]\left[\begin{array}{cc} 0 & 1 \\ 1 & 0 \end{array}\right] & =& \left[\begin{array}{cc} \beta & 1 \\ 1 & 0 \end{array}\right] \left[\begin{array}{cc} \alpha & 1 \\ 1 & 0 \end{array}\right]
\end{eqnarray}

This relation is a particularily simple kind of zip-relation, see section \ref{canonicalsection} page \pageref{canonicalsection}.

 \item Note that of course one has $\beta_{N-k} = a_{2k}$, and  $\alpha_{N-k} = a_{2k+1}$ for $1\leqslant k \leqslant N$. But this order is much more convenient here for it coïncides with the composition both of homeomorphisms and unfoldings of topological snails (see figure \ref{figure29}). Note that the number $N$ can only be found at the end of the algorithm.
 
 \item Let $M=\left[\begin{array}{cc} a & b \\ c & d \end{array}\right]\in \mathrm{SL}_{2}(\Z)$ with positive integer coefficients ($ad-bc=1$). Then $ad=1+bc$ shows that $a>0$ and $d>0$. Further more 
 \begin{eqnarray} \left[\begin{array}{cc} 1 & -1 \\ 0 & 1 \end{array}\right]\left[\begin{array}{cc} a & b \\ c & d \end{array}\right] & =& \left[\begin{array}{cc} a-c & b-d \\ c & d \end{array}\right]\end{eqnarray}
 
 If $n=a+b$ and $p=c+d$ then $(a-c)d= 1 + c(b-d)$ and 
 \begin{eqnarray*} n>p & \iff & a-c > d-b \\
 & \iff & (a-c)d > (d-b)d \\
  & \iff & 1+ c(b-d) > (d-b)d \\
   & \iff & 1 > (d-b)(d+c) \\
   &\iff & \left\lbrace \begin{array}{rcl} b & \leqslant & d \\
   a-c & >&  0\end{array}\right. \\
      &\iff & \left\lbrace \begin{array}{rcl} b -d & \geqslant & 0 \\
   a-c & >&  0\end{array}\right.
\end{eqnarray*}

This shows that one has therefore $M = \mt(n;p)$.
 
\end{remarks}

\begin{figure}
\begin{picture}(100,120)(-60,0)

\put(0,30){\line(0,1){40}}
\put(-42,30){\line(0,1){40}}
\put(156,30){\line(0,1){40}}

\put(0,60){\vector(-1,0){42}}
\put(-21,62){\makebox(0,0)[bc]{$ q_{n+1}$}}
\put(-21,58){\makebox(0,0)[tc]{$ f^{u_{n+1}}$}}

\put(0,60){\vector(1,0){156}}
\put(78,62){\makebox(0,0)[bc]{$q_{n}$}}
\put(78,58){\makebox(0,0)[tc]{$ f^{u_{n}}$}}

\put(-42,50){\line(1,0){198}}
\put(-42,70){\line(1,0){198}}

\put(156,40){\vector(-1,0){42}}
\put(135,42){\makebox(0,0)[bc]{$q_{n+1}$}}
\put(135,38){\makebox(0,0)[tc]{$f^{u_{n+1}}$}}

\put(114,40){\vector(-1,0){42}}
\put(93,42){\makebox(0,0)[bc]{$q_{n+1}$}}
\put(93,38){\makebox(0,0)[tc]{$f^{u_{n+1}}$}}

\put(51,42){\makebox(0,0)[bc]{$q_{n+1}$}}
\put(51,38){\makebox(0,0)[tc]{$f^{u_{n+1}}$}}
\put(72,40){\vector(-1,0){42}}

\put(0,40){\vector(1,0){30}}
\put(15,38){\makebox(0,0)[tc]{$f^{u_{n+2}}$}}
\put(15,42){\makebox(0,0)[bc]{$q_{n+2}$}}

\put(0,40){\vector(-1,0){42}}
\put(-21,42){\makebox(0,0)[bc]{$q_{n+1}$}}
\put(-21,38){\makebox(0,0)[tc]{$f^{u_{n+1}}$}}

\put(-20,10){\makebox(0,0)[bl]{$\displaystyle q_{n}= a_{n}q_{n+1} + q_{n+2} $}}

\put(60,10){\makebox(0,0)[bl]{$\displaystyle u_{n+2}= a_{n}u_{n+1} + u_{n} $}}

\put(30,30){\line(0,1){20}}
\put(72,30){\line(0,1){20}}
\put(114,30){\line(0,1){20}}

\put(-42,30){\line(1,0){198}}

\color{blue}

\end{picture}
\caption{Euclidean algorithm modulo $q_{0}+q_{1}$ for finding a generating interval .\label{figure26}}

\end{figure}

\begin{figure}
\begin{picture}(100,100)(-105,-45)

\put(-90,0){\line(1,0){180}}
\put(90,0){\vector(1,0){0}}
\put(-90,-2){\makebox(0,0)[tc]{$\Delta$}}
\put(-1,1){\makebox(0,0)[br]{$ 0$}}

\color{red}


\color{black}

\put(40,0){\circle*{1}}

\put(40,-2){\makebox(0,0)[tc]{$\frac{p}{2}$}}

\put(10,0){\circle*{1}}
\put(10,2){\makebox(0,0)[bc]{$\frac{-n+p}{2}$}}

\put(-30,0){\circle*{1}}
\put(-30,-2){\makebox(0,0)[tc]{$\frac{-n}{2}$}}

\put(0,0){\circle*{1}}

\color{mygreen}

\put(40,0){\qbezier[300](-40,0)(-40,-16.56)(-28.28,-28.28)}
\put(40,0){\qbezier[300](-28.28,-28.28)(-16.56,-40)(0,-40)}
\put(40,0){\qbezier[300](0,-40)(16.56,-40)(28.28,-28.28)}
\put(40,0){\qbezier[300](28.28,-28.28)(40,-16.5)(40,0)}

\put(-30,0){\qbezier[300](-30,0)(-30,-12.42)(-21.21,-21.21)}
\put(-30,0){\qbezier[300](-21.21,-21.21)(-12.42,-30)(0,-30)}
\put(-30,0){\qbezier[300](0,-30)(12.42,-30)(21.21,-21.21)}
\put(-30,0){\qbezier[300](21.21,-21.21)(30,-12.42)(30,0)}


\put(10,0){\qbezier[300](-70,0)(-70,28.98)(-49.49,49.49)}
\put(10,0){\qbezier[300](-49.49,49.49)(-28.98,70)(0,70)}
\put(10,0){\qbezier[300](0,70)(28.98,70)(49.49,49.49)}
\put(10,0){\qbezier[300](49.49,49.49)(70,28.98)(70,0)}
\put(10,0){\qbezier[300](-10,0)(-10,4.14)(-7.07,7.07)}
\put(10,0){\qbezier[300](-7.07,7.07),(-4.14,10),(0,10)}
\put(10,0){\qbezier[300](0,10),(4.14,10)(7.07,7.07)}
\put(10,0){\qbezier[300](7.07,7.07)(10,4.14)(10,0)}

\put(40,0){\qbezier[300](-20,0)(-20,-8.28)(-14.14,-14.14)}
\put(40,0){\qbezier[300](-14.14,-14.14)(-8.28,-20)(0,-20)}
\put(40,0){\qbezier[300](0,-20)(8.28,-20)(14.14,-14.14)}
\put(40,0){\qbezier[300](14.14,-14.14)(20,-8.28)(20,0)}

\put(10,0){\qbezier[300](-50,0)(-50,20.7)(-35.35,35.35)}
\put(10,0){\qbezier[300](-35.35,35.35),(-20.7,50),(0,50)}
\put(10,0){\qbezier[300](0,50),(20.7,50)(35.35,35.35)}
\put(10,0){\qbezier[300](35.35,35.35)(50,20.7)(50,0)}

\put(-30,0){\qbezier[300](-10,0)(-10,-4.14)(-7.07,-7.07)}
\put(-30,0){\qbezier[300](-7.07,-7.07),(-4.14,-10),(0,-10)}
\put(-30,0){\qbezier[300](0,-10),(4.14,-10)(7.07,-7.07)}
\put(-30,0){\qbezier[300](7.07,-7.07)(10,-4.14)(10,0)}

\put(10,0){\qbezier[300](-30,0)(-30,12.42)(-21.21,21.21)}
\put(10,0){\qbezier[300](-21.21,21.21)(-12.42,30)(0,30)}
\put(10,0){\qbezier[300](0,30)(12.42,30)(21.21,21.21)}
\put(10,0){\qbezier[300](21.21,21.21)(30,12.42)(30,0)}

\color{blue}

\put(40,0){\qbezier[300](-10,0)(-10,-4.14)(-7.07,-7.07)}
\put(40,0){\qbezier[300](-7.07,-7.07),(-4.14,-10),(0,-10)}
\put(40,0){\qbezier[300](0,-10),(4.14,-10)(7.07,-7.07)}
\put(40,0){\qbezier[300](7.07,-7.07)(10,-4.14)(10,0)}

\put(40,0){\qbezier[300](-30,0)(-30,-12.42)(-21.21,-21.21)}
\put(40,0){\qbezier[300](-21.21,-21.21)(-12.42,-30)(0,-30)}
\put(40,0){\qbezier[300](0,-30)(12.42,-30)(21.21,-21.21)}
\put(40,0){\qbezier[300](21.21,-21.21)(30,-12.42)(30,0)}

\put(-30,0){\qbezier[300](-20,0)(-20,-8.28)(-14.14,-14.14)}
\put(-30,0){\qbezier[300](-14.14,-14.14)(-8.28,-20)(0,-20)}
\put(-30,0){\qbezier[300](0,-20)(8.28,-20)(14.14,-14.14)}
\put(-30,0){\qbezier[300](14.14,-14.14)(20,-8.28)(20,0)}

\put(10,0){\qbezier[300](-20,0)(-20,8.28)(-14.14,14.14)}
\put(10,0){\qbezier[300](-14.14,14.14)(-8.28,20)(0,20)}
\put(10,0){\qbezier[300](0,20)(8.28,20)(14.14,14.14)}
\put(10,0){\qbezier[300](14.14,14.14)(20,7,64)(20,0)}

\put(10,0){\qbezier[300](-40,0)(-40,16.56)(-28.28,28.28)}
\put(10,0){\qbezier[300](-28.28,28.28)(-16.56,40)(0,40)}
\put(10,0){\qbezier[300](0,40)(16.56,40)(28.28,28.28)}
\put(10,0){\qbezier[300](28.28,28.28)(40,16.56)(40,0)}

\put(10,0){\qbezier[300](-60,0)(-60,24.84)(-42.42,42.42)}
\put(10,0){\qbezier[300](-42.42,42.42)(-24.84,60)(0,60)}
\put(10,0){\qbezier[300](0,60)(24.84,60)(42.42,42.42)}
\put(10,0){\qbezier[300](42.42,42.42)(60,24.84)(60,0)}

\end{picture}
\caption{Singular curves associated to the snail.\label{figure27}}

\end{figure}

\begin{figure}
\begin{picture}(90,70)(-50,-20)

\put(-5,60){\makebox(0,0)[br]{$ 7$}}
\put(5,60){\makebox(0,0)[bl]{$ 26$}}


\put(14,57){\qbezier[80](0,3)(4,0)(0,-3)}
\put(14,60){\vector(-1,1){0}}

\put(-10,57){\makebox(0,0)[br]{$\scriptstyle $}}
\put(18,57){\makebox(0,0)[cl]{$ B$}}


\put(-5,50){\makebox(0,0)[br]{$ 7$}}
\put(5,50){\makebox(0,0)[bl]{$ 19$}}

\put(14,47){\qbezier[80](0,3)(4,0)(0,-3)}
\put(14,50){\vector(-1,1){0}}

\put(-10,47){\makebox(0,0)[br]{$\scriptstyle $}}
\put(18,47){\makebox(0,0)[cl]{$ B$}}


\put(-5,40){\makebox(0,0)[br]{$ 7$}}
\put(5,40){\makebox(0,0)[bl]{$ 12$}}

\put(14,37){\qbezier[80](0,3)(4,0)(0,-3)}
\put(14,40){\vector(-1,1){0}}

\put(-10,37){\makebox(0,0)[br]{$\scriptstyle $}}
\put(18,37){\makebox(0,0)[cl]{$ B$}}


\put(-5,30){\makebox(0,0)[br]{$ 7$}}
\put(5,30){\makebox(0,0)[bl]{$ 5$}}

\put(-9,27){\qbezier[80](0,3)(-4,0)(0,-3)}
\put(-9,30){\vector(1,1){0}}

\put(-13,27){\makebox(0,0)[cr]{$A $}}
\put(18,27){\makebox(0,0)[cl]{$ $}}


\put(-5,20){\makebox(0,0)[br]{$ 2$}}
\put(5,20){\makebox(0,0)[bl]{$ 5$}}


\put(-5,10){\makebox(0,0)[br]{$ 2$}}
\put(5,10){\makebox(0,0)[bl]{$ 3$}}

\put(10,17){\qbezier[80](0,3)(4,0)(0,-3)}
\put(10,20){\vector(-1,1){0}}

\put(-10,17){\makebox(0,0)[br]{$\scriptstyle $}}
\put(14,17){\makebox(0,0)[cl]{$ B$}}


\put(-5,0){\makebox(0,0)[br]{$ 2$}}
\put(5,0){\makebox(0,0)[bl]{$ 1$}}

\put(10,7){\qbezier[80](0,3)(4,0)(0,-3)}
\put(10,10){\vector(-1,1){0}}

\put(-10,7){\makebox(0,0)[br]{$\scriptstyle $}}
\put(14,7){\makebox(0,0)[cl]{$ B$}}

\put(-10,-3){\qbezier[80](0,3)(-4,0)(0,-3)}
\put(-10,0){\vector(1,1){0}}

\put(-13,-3){\makebox(0,0)[cr]{$A $}}
\put(18,-3){\makebox(0,0)[cl]{$ $}}


\put(-5,-10){\makebox(0,0)[br]{$1$}}
\put(5,-10){\makebox(0,0)[bl]{$ 1$}}

\put(120,37){\makebox(0,0)[c]{$(7;26) = B^{3} \, A \, B^{2} \, A (1;1)$}}

\put(120,22){\makebox(0,0)[c]{$ \Longleftrightarrow $}}

\put(120,7){\makebox(0,0)[c]{$\sn^{+}(7;26) = f_{B}^{3} \circ f_{A} \circ f_{B}^{2} \circ f_{A} \left(\sn^{+}(1;1)\right)$}}

\put(29,12){\makebox(0,0)[cl]{$\beta_{1}=2$}}

\put(23,47){\qbezier[80](0,12)(5,0)(0,-12)}

\put(21,12){\qbezier[80](0,8)(3,0)(0,-8)}
\put(-9,30){\vector(1,1){0}}

\put(-18,26){\qbezier[80](0,4)(-3,0)(0,-4)}

\put(-22,26){\makebox(0,0)[rc]{$\alpha_{2}=1$}}

\put(-22,-4){\makebox(0,0)[rc]{$\alpha_{1}=1$}}

\put(-18,-4){\qbezier[80](0,4)(-3,0)(0,-4)}

\put(29,47){\makebox(0,0)[cl]{$\beta_{2}=3$}}

\color{blue}

\end{picture}
\caption{Unique decomposition of $\sn^{+}(7;26)$.\label{figure28}}

\end{figure}

One has the following result~:

 \begin{theo} \label{decomposition}
 Let $n$ and $p$ be two integers relatively prime, and  let $(\alpha_{1},\alpha_{2}, \ldots , \alpha_{N})$ and $(\beta_{1},\beta_{2}, \ldots , \beta_{N})$ be its associated caracteristic sequences (see theorem \ref{caracteristic}). Let $f_{A}$ and $f_{B}$ be two homeomorphisms of respective types $A$ and $B$ relative to $X = \lbrace p_{1}, p_{2} ; p_{3} \rbrace$ on $\Delta^{+}$, where $p_{1}$; $p_{2}$, $p_{3}$ are the centers of the topological snail $\sn(n;p)$. Then
 
  \begin{eqnarray} \sn(n;p) & = &f_{B}^{\, \beta_{N}} \circ f_{A}^{\alpha_{N}} \circ  \cdots  \circ f_{B}^{\, \beta_{1}} \circ f_{A}^{\alpha_{1}}\left( \sn(1;1) \right) \end{eqnarray}
 
Furthermore, $n$ is the minimum of the number of intersection points of $\sn(n;p)$ and the isotopy class of the perpendicular bissector $\Delta_{1}$ of $[p_{1}; p_{2}]_{\Delta}$, and  $p$ is the minimum of the number of intersection points of $\sn(n;p)$ and the isotopy class of the perpendicular bissector $\Delta_{2}$ of $[p_{2}; p_{3}]_{\Delta}$ (see figure \ref{figure22}).
\end{theo}

\begin{proof} 
Let us first prove that for any homeomorphism of the type $A$ relatively to $p_{1}$; $p_{2}$, $p_{3}$ on $\Delta^{+}$, one has
 \begin{eqnarray*} f_{A}(\sn(n;p)) & \equiv & \sn(n+p; p)
 \end{eqnarray*} 
 for any strictly positive integers $n$ and $p$. 
 
Let us consider the topological snail $\sn(n;p)$ defined in section \ref{pictures} and simply translate the coordinate system to take $p_{1}$ as new origin. 

In an associated complex coordinate system, let us define $ k\left(\frac{n+p}{2}; -\frac{\varepsilon}{2}\right)$; $z_{K} = \frac{n+p}{2}- \frac{\varepsilon}{2}\, i$ for $\displaystyle \varepsilon < \frac{1}{4}$, and let us consider three positive $\mathrm{C}^{\infty}$-fonctions, $\psi_{1}$, $\psi_{2}$ and $\psi_{3}$ :

\begin{enumerate}
\item The fonction $\displaystyle \displaystyle\psi_{1}~:~\R~\rightarrow [0;1]$ such that~:
\begin{itemize}
\item[$\bullet$] $\psi_{1}(x) = 1$ if $\displaystyle x\notin \left[\frac{n}{2}+\frac{1}{4} -\frac{\varepsilon}{2}; \frac{n}{2}+p-\frac{1}{4}+\frac{\varepsilon}{2}\right]$, 
\item[$\bullet$] $\psi_{1}(x) = 0$ if $ \displaystyle x\in  \left[\frac{n}{2}+\frac{1}{4}; \frac{n}{2}+p-\frac{1}{4} \right]$,
\item[$\bullet$] $\psi_{1}$ is strictly decreasing on $\displaystyle \left[\frac{n}{2}+\frac{1}{4} -\frac{\varepsilon}{2};\frac{n}{2}+\frac{1}{4} \right]$,
\item[$\bullet$] $\psi_{1}$ is strictly increasing on $ \displaystyle \left[\frac{n}{2}+p-\frac{1}{4} -\frac{\varepsilon}{2};\frac{n}{2}+p-\frac{1}{4} \right]$.
\end{itemize}

\item The fonction $\displaystyle \displaystyle\psi_{2}~:~\R~\rightarrow [0;1]$ such that~:
\begin{itemize}
\item[$\bullet$] $\psi_{2}(x) = 1$ if $\displaystyle x\notin \left[-\frac{p}{2}-\frac{1}{4} -\frac{\varepsilon}{2}; 0\right]$, 
\item[$\bullet$] $\psi_{2}(x) = 0$ if $\displaystyle x\in \left[-\frac{p}{2}-\frac{1}{4} ; -\frac{\varepsilon}{2}\right]$, 
\item[$\bullet$] $\psi_{2}$ is strictly decreasing on $\displaystyle\left[-\frac{p}{2}-\frac{1}{4} -\frac{\varepsilon}{2}; -\frac{p}{2}-\frac{1}{4}\right]$.
\item[$\bullet$] $\psi_{2}$ is strictly increasing on $\displaystyle \left[-\frac{\varepsilon}{2}; 0\right]$.
\end{itemize}

\item The fonction $\displaystyle \psi_{3}~:~[0,+\infty[~\rightarrow [0;1]$ such that~:
\begin{itemize}
\item[$\bullet$] $\psi_{3}(r) = 0$ if $\displaystyle r \in \left[0;\frac{\varepsilon}{2}\right]$, 
\item[$\bullet$] $\psi_{3}(r) = 1$ if  $\displaystyle r \in \left[\varepsilon, +\infty \right[$,
\item[$\bullet$] $\psi_{3}$ is strictly increasing on $\displaystyle \left[\frac{\varepsilon}{2}; \varepsilon\right]$.
\end{itemize}

\end{enumerate}

\input{figure29.tex}


\begin{figure}
\begin{picture}(100,150)(-115,-75)

\put(-90,0){\line(1,0){200}}
\put(110,0){\vector(1,0){0}}
\put(-60,-30){\line(1,1){120}}
\put(-90,-2){\makebox(0,0)[tc]{$\Delta$}}
\put(-1,1){\makebox(0,0)[br]{$\scriptstyle \frac{n}{2}$}}
\put(0,0){\circle*{1}}
\put(40,-2.5){\circle*{1}}
\put(40,-11){\makebox(0,0)[tc]{$\scriptstyle K(\frac{n+p}{2}; -\frac{\varepsilon}{2})$}}
\put(-1.72,28.28){\circle*{1}}

\put(10,0){\circle*{1}}
\put(11,1){\makebox(0,0)[bl]{$\scriptstyle \frac{n+1}{2}$}}
\put(80,0){\circle*{1}}
\put(81,1){\makebox(0,0)[bl]{$\scriptstyle \frac{n}{2}+p$}}

\put(-30,0){\circle*{1}}
\put(-30,-2){\makebox(0,0)[tc]{$\scriptstyle  0$}}
\put(-10,0){\circle*{1}}
\put(-10,-2){\makebox(0,0)[tl]{$\scriptstyle  1$}}



\color{mygreen}

\put(-1.72,28.28){\qbezier[300](-50,0)(-50,20.7)(-35.35,35.35)}
\put(-1.72,28.28){\qbezier[300](-35.35,35.35),(-20.7,50),(0,50)}
\put(-1.72,28.28){\qbezier[300](0,50),(20.7,50)(35.35,35.35)}
\put(-1.72,28.28){\qbezier[300](-35.35,-35.35)(-50,-20.7)(-50,0)}

\put(40,0){\qbezier[100](-30,0)(-30,-12.42)(-21.21,-21.21)}
\put(40,0){\qbezier[100](-21.21,-21.21)(-12.42,-30)(0,-30)}
\put(40,0){\qbezier[100](0,-30)(12.42,-30)(21.21,-21.21)}
\put(40,0){\qbezier[100](21.21,-21.21)(30,-12.42)(30,0)}

\color{mygreen}

\put(40,0){\qbezier[300](-10,0)(-10,-4.14)(-7.07,-7.07)}
\put(40,0){\qbezier[300](-7.07,-7.07),(-4.14,-10),(0,-10)}
\put(40,0){\qbezier[300](0,-10),(4.14,-10)(7.07,-7.07)}
\put(40,0){\qbezier[300](7.07,-7.07)(10,-4.14)(10,0)}

\put(40,0){\qbezier[300](-30,0)(-30,-12.42)(-21.21,-21.21)}
\put(40,0){\qbezier[300](-21.21,-21.21)(-12.42,-30)(0,-30)}
\put(40,0){\qbezier[300](0,-30)(12.42,-30)(21.21,-21.21)}
\put(40,0){\qbezier[300](21.21,-21.21)(30,-12.42)(30,0)}

\put(-30,0){\qbezier[300](20,0)(20,8.28)(14.14,14.14)}
\put(-30,0){\qbezier[300](-14.14,-14.14)(-8.28,-20)(0,-20)}
\put(-30,0){\qbezier[300](0,-20)(8.28,-20)(14.14,-14.14)}
\put(-30,0){\qbezier[300](14.14,-14.14)(20,-8.28)(20,0)}

\put(-1.72,28.28){\qbezier[300](-20,0)(-20,8.28)(-14.14,14.14)}
\put(-1.72,28.28){\qbezier[300](-14.14,14.14)(-8.28,20)(0,20)}
\put(-1.72,28.28){\qbezier[300](0,20)(8.28,20)(14.14,14.14)}
\put(-1.72,28.28){\qbezier[300](-14.14,-14.14)(-20,-8.28)(-20,0)}

\put(-1.72,28.28){\qbezier[300](-40,0)(-40,16.56)(-28.28,28.28)}
\put(-1.72,28.28){\qbezier[300](-28.28,28.28)(-16.56,40)(0,40)}
\put(-1.72,28.28){\qbezier[300](0,40)(16.56,40)(28.28,28.28)}
\put(-1.72,28.28){\qbezier[300](-28.28,-28.28)(-40,-16.56)(-40,0)}

\put(-1.72,28.28){\qbezier[300](-60,0)(-60,24.84)(-42.42,42.42)}
\put(-1.72,28.28){\qbezier[300](-42.42,42.42)(-24.84,60)(0,60)}
\put(-1.72,28.28){\qbezier[300](0,60)(24.84,60)(42.42,42.42)}
\put(-1.72,28.28){\qbezier[300](-42.42,-42.42)(-60,-24.84)(-60,0)}

\put(-1.72,28.28){\qbezier[300](-10,0)(-10,4.14)(-7.07,7.07)}
\put(-1.72,28.28){\qbezier[300](-7.07,7.07),(-4.14,10),(0,10)}
\put(-1.72,28.28){\qbezier[300](0,10),(4.14,10)(7.07,7.07)}
\put(-1.72,28.28){\qbezier[300](-7.07,-7.07)(-10,-4.14)(-10,0)}


\put(40,0){\qbezier[300](-20,0)(-20,-8.28)(-14.14,-14.14)}
\put(40,0){\qbezier[300](-14.14,-14.14)(-8.28,-20)(0,-20)}
\put(40,0){\qbezier[300](0,-20)(8.28,-20)(14.14,-14.14)}
\put(40,0){\qbezier[300](14.14,-14.14)(20,-8.28)(20,0)}

\put(-1.72,28.28){\qbezier[300](-50,0)(-50,20.7)(-35.35,35.35)}
\put(-1.72,28.28){\qbezier[300](-35.35,35.35),(-20.7,50),(0,50)}
\put(-1.72,28.28){\qbezier[300](0,50),(20.7,50)(35.35,35.35)}
\put(-1.72,28.28){\qbezier[300](-35.35,-35.35)(-50,-20.7)(-50,0)}

\put(-30,0){\qbezier[300](10,0)(10,4.14)(7.07,7.07)}
\put(-30,0){\qbezier[300](-7.07,-7.07),(-4.14,-10),(0,-10)}
\put(-30,0){\qbezier[300](0,-10),(4.14,-10)(7.07,-7.07)}
\put(-30,0){\qbezier[300](7.07,-7.07)(10,-4.14)(10,0)}

\put(-1.72,28.28){\qbezier[300](-30,0)(-30,12.42)(-21.21,21.21)}
\put(-1.72,28.28){\qbezier[300](-21.21,21.21)(-12.42,30)(0,30)}
\put(-1.72,28.28){\qbezier[300](0,30)(12.42,30)(21.21,21.21)}
\put(-1.72,28.28){\qbezier[300](-21.21,-21.21)(-30,-12.42)(-30,0)}

\put(-30,0){\qbezier[300](28.28,28.28)(40,16.5)(40,0)}
\put(-30,0){\qbezier[300](70,0)(70,28.98)(49.49,49.49)}
\put(-30,0){\qbezier[300](56.56,56.56)(80,33)(80,0)}

\put(-30,0){\qbezier[300](50,0),(50,20.7)(35.35,35.35)}

\put(-30,0){\qbezier[300](60,0)(60,24.84)(42.42,42.42)}
\put(-30,0){\qbezier[300](70.7,70.7)(100,41.41)(100,0)}
\put(-30,0){\qbezier[300](63.63,63.63)(90,37.25)(90,0)}

\color{black}

\put(-30,0){\qbezier[300](77.77,77.77)(110,45.54)(110,0)}

\put(40,0){\qbezier[300](-40,0)(-40,-16.56)(-28.28,-28.28)}
\put(40,0){\qbezier[300](-28.28,-28.28)(-16.56,-40)(0,-40)}
\put(40,0){\qbezier[300](0,-40)(16.56,-40)(28.28,-28.28)}
\put(40,0){\qbezier[300](28.28,-28.28)(40,-16.5)(40,0)}

\put(-30,0){\qbezier[300](30,0)(30,12.42)(21.21,21.21)}
\put(-30,0){\qbezier[300](-21.21,-21.21)(-12.42,-30)(0,-30)}
\put(-30,0){\qbezier[300](0,-30)(12.42,-30)(21.21,-21.21)}
\put(-30,0){\qbezier[300](21.21,-21.21)(30,-12.42)(30,0)}


\put(-1.72,28.28){\qbezier[300](-70,0)(-70,28.98)(-49.49,49.49)}
\put(-1.72,28.28){\qbezier[300](-49.49,49.49)(-28.98,70)(0,70)}
\put(-1.72,28.28){\qbezier[300](0,70)(28.98,70)(49.49,49.49)}
\put(-1.72,28.28){\qbezier[300](-49.49,-49.49)(-70,-28.98)(-70,0)}

\color{blue}

\put(40,-2.5){\circle{10}}
\put(40,-2.5){\circle{5}}

\put(5,-2.5){\line(1,0){70}}
\put(5,-42.5){\line(1,0){70}}
\put(2.5,0){\line(1,0){75}}
\put(2.5,-45){\line(1,0){75}}

\put(5,-2.5){\line(0,-1){40}}
\put(2.5,0){\line(0,-1){45}}

\put(75,-2.5){\line(0,-1){40}}
\put(77.5,0){\line(0,-1){45}}

\put(5,-42.5){\circle*{1}}
\put(6,-41.5){\makebox(0,0)[bl]{$\scriptstyle  \alpha'$}}

\put(75,-42.5){\circle*{1}}
\put(74,-41.5){\makebox(0,0)[br]{$\scriptstyle  \beta'$}}
\put(75,-2.5){\circle*{1}}
\put(74,-3.5){\makebox(0,0)[tr]{$\scriptstyle  \gamma'$}}
\put(5,-2.5){\circle*{1}}
\put(6,-3.5){\makebox(0,0)[tl]{$\scriptstyle  \delta'$}}

\put(2.5,-45){\circle*{1}}
\put(2.5,-46){\makebox(0,0)[tc]{$\scriptstyle  \alpha$}}
\put(77.5,-45){\circle*{1}}
\put(77.5,-46){\makebox(0,0)[tc]{$\scriptstyle  \beta$}}
\put(77.5,0){\circle*{1}}
\put(77.5,1){\makebox(0,0)[bc]{$\scriptstyle  \gamma$}}
\put(2.5,0){\circle*{1}}
\put(2.5,1){\makebox(0,0)[bc]{$\scriptstyle  \delta$}}

\end{picture}
\caption{Image of $\varphi_{\frac{\pi}{4}}\left(\sn(3;4)\right)$ and $\varphi_{\frac{\pi}{4}}\left(\cp(3;4)\right)$ \label{figure30}}

\end{figure}

Let (see figure \ref{figure30})~:
$$\begin{array}{lcllcl}
\alpha & = & \left(\frac{n}{2} + \frac{1}{4} - \frac{\varepsilon}{2}\right) +i \left( -\frac{p}{2}-\frac{1}{4} -\frac{\varepsilon}{2}\right) 
&\alpha' & = & \left(\frac{n}{2} + \frac{1}{4}\right) +i \left( -\frac{p}{2}-\frac{1}{4}\right)

\\
\beta & = & \left(\frac{n}{2} +p- \frac{1}{4} + \frac{\varepsilon}{2}\right) +i \left( -\frac{p}{2}-\frac{1}{4} -\frac{\varepsilon}{2}\right) 
&\beta' & = & \left(\frac{n}{2} +p- \frac{1}{4} \right) +i \left( -\frac{p}{2}-\frac{1}{4} \right) 

\\
\gamma & = & \left(\frac{n}{2} +p- \frac{1}{4} + \frac{\varepsilon}{2}\right)
& \gamma' & = & \left(\frac{n}{2} +p- \frac{1}{4} \right)- i \left(\frac{\varepsilon}{2}\right)  \\
\delta & = & \left(\frac{n}{2} + \frac{1}{4} - \frac{\varepsilon}{2} \right)
& \delta' & = & \left(\frac{n}{2} + \frac{1}{4}\right)- i \left(\frac{\varepsilon}{2}\right)
\end{array}$$

Define $\Omega_{1}$ to be the union of the rectangle $\alpha\beta\gamma\delta$ and the disk of center $K$ and radius $\varepsilon$ (see figure \ref{figure30}), and $\Omega_{0}$ to be the union of the rectangle $\alpha'\beta'\gamma'\delta'$ and the disk of center $K$ and radius $\frac{\varepsilon}{2}$. Then $\psi(z)=0 \iff z\in \Omega_{0}$ and $\psi(z) = 1 \iff z\notin \Omega_{1}$. Let us now define the function $\psi: \C \rightarrow [0;+\infty[$, for $z= \vert z \vert e^{i\, \theta}$ by 
$$\psi(x+iy) = \psi_{1}(x) \times \psi_{2}(y) \times\psi_{3}(\vert z - z_{K} \vert)\, , $$
and the vector field $\ov{X}(z) = i z \psi(z)$. It is a $\mathrm{C}^{\infty}$ vector field so by Cauchy-Lipschitz' theorem there exists a unique flow $\varphi_{t}(z)$, solution of the differential equation $\frac{d}{dt} \varphi_{t}(z) = \ov{X}(\varphi_{t})$. 

For $z_{0} \notin \Omega_{1}$, while $e^{it}z_{0}\notin \Omega_{1}$, one has $\varphi_{t}(z_{0})= e^{it}z_{0}$. Of course, if $z_{0} \in \Omega_{0}$ or $z_{0} = 0$, then for all $t\in \R$, $\varphi_{t}(z_{0})= z_{0}$.

Let $D_{1}$ be the left sustaining half-disk, of center $p_{1}(0;0)$ and radius $\frac{n}{2}$ in $\H^{-}$, $D_{2}$ the emerging half-disk, of center $p_{2}(\frac{p}{2};0)$ and radius $\frac{n+p}{2}$ in $\H^{+}$ and $D_{3}$ the right sustaing half-disk of center $\frac{n+p}{2}$ and radius $\frac{p}{2}$ in $\H^{-}$. For $t\in [0;\pi]$, $ \varphi_{t}$ coïncides on $D_{1} \cup D_{2}$ with the rotation of angle $t$, and with the identity map on $\Omega_{0}\cap D_{3}$. Note that $\varphi_{t}$ is a rotary $\mathrm{C}^{\infty}$ homeomorphism, and that each subsarc of $\sn(n;p)$ that is located in $D_{3}$ is fixed, except for its section doubly branched between the half-lines defined by $y=0$ and $y=-\varepsilon$. Those sections are extended to curves doubly branched on $y=-\frac{\varepsilon}{2}$ and $\theta=t$, respecting the foliation by concentric circles of center $p_{1}$, and therefore homeomorphic to the junction of an arc of circle doubly branched on $\theta=0$ and $\theta=t$ and a vertical segment doubly branched between the half-lines defined by $y=0$ and $y=-\varepsilon$. (see figure \ref{figure30}). 

Let $Q_{1}(-\frac{p}{2};0)$, $Q_{2}=p_{1}(0;0)$ and $Q_{3}  = p_{3}(\frac{n+p}{2};0)$ and $h_{t}(x;y)$ defined for $t\in [0;1]$ by 
$$h_{t}(x;y) = \left(x+ t\frac{p}{2};y\right)~~\mathrm{if}~~x\leqslant 0 ~~\mathrm{and}~~h_{t}(x;y) = \left(\frac{n+(1-t)p}{n+p}x+ t\frac{p}{2};y\right)~~\mathrm{if}~~x\geqslant 0  $$
One has an isotopy defined for $t\in [0;1]$ by $h_{t} = \id + t (h- \id)$ with $h_{0} = \id$ and $h(Q_{1}) = p_{1}$, $h(Q_{2})=p_{2}$ and $h(Q_{3}) = p_{3}$.
Therefore if one defines $f_{t}(x;y) = \varphi_{2t}(x;y)$ for $t\in [0;\frac{1}{2}]$ and $f_{t} = h_{2t-1}$ for $t\in [\frac{1}{2};1]$ then one obtains an isotopy from identity to a type $A$ homeomorphism relatively to $\Delta$ and $X$, transforming $\sn(n;p)$ in $\sn(n+p;p)$.   

Of course, a similar symetric proof shows that for any homeomorphism $f_{B}$ of the type $B$ relatively to $p_{1}$; $p_{2}$, $p_{3}$ on $\Delta^{+}$, then $f_{B}(\sn(n;p)) = \sn(n;n+p)$ for any positive integers $n$ and $p$.
And the first part of the theorem follows immediatly by recurrence.

The second part of the theorem can also be shown by recurrence. Let $\Delta_{1}$ be the isotopy class relatively to $X= \lbrace p_{1}; p_{2}; p_{3} \rbrace$ of the perpendicular bissector $\Delta_{1}$ of $[p_{1}; p_{2}]_{\Delta}$ and $\Delta_{2}$ the isotopy class relatively to $X= \lbrace p_{1}; p_{2}; p_{3} \rbrace$ of the perpendicular bissector $\Delta_{2}$ of $[p_{2}; p_{3}]_{\Delta}$. Since $\sn(1;1)$ connects $p_{1}$ to $p_{3}$, and since each of $\Delta_{1}$ and $\Delta_{2}$ separates those points, Jordan's theorem implie that $n=1$ is also the minimum of the number of intersection points between $\sn(1;1)$ and the isotopy class of $\Delta_{1}$, and $p=1$ the minimum of the number of intersection points between $\sn(1;1)$ and the isotopy class of $\Delta_{2}$.

Now suppose that this property holds for some $(n;p)$ and consider $(n+p;p)$. Since for an homeomorphism of the type $A$ on has $f^{-1}(\Delta_{1}) = \Delta_{1} \cup \Delta_{2}$ (see figure \ref{figure6}), then the minimum of the number of intersection points between $f^{-1}(\Delta_{1}) = \Delta_{1} \cup \Delta_{2}$ and $\sn(n;p)$ is $n+p$, and since $f_{A}$ is bijective it is the minimum of the number of intersection points between   $\Delta_{1}) = \Delta_{1}$ with $f_{A}(\sn(n;p))$. Similarily the minimum of the number of intersection points between $f^{-1}(\Delta_{2}) = \Delta_{2}$ and $\sn(n;p)$ is $p$ and it is the minimum of the number of intersection points between   $\Delta_{2}) $ with $f_{A}(\sn(n;p))$. Of course, this property has its symetric analogous for homeomorphisms of the type $B$ and implies the second part of the theorem by a classical recurrence from $n=1$ and $p=1$.

\end{proof}

\section{The double topological snail \label{sectiondouble}}

The topological type relatively to a set $X$ on a line $\Delta^{+}$ of plane simple curve branched on two points of this set and emerging in the top half-plane $\H^{+}$ is a topological snail $\sn(n;p)$. Furthemore, if $(n;p)$ has caracteristic sequences of positive natural numbers $(\alpha_{1},\alpha_{2}, \ldots , \alpha_{N})$ and  $(\beta_{1},\beta_{2}, \ldots , \beta_{N})$, and if one set (see equation \ref{turbulence} page \pageref{turbulence})

\begin{eqnarray} \mt(n;p) & = & B^{\, \beta_{N}} \circ A^{\alpha_{N}} \circ  \cdots  \circ B^{\, \beta_{1}} \circ A^{\alpha_{1}} \\
& = & \left[\begin{array}{cc} 1 & 0 \\ \beta_{N} & 1 \end{array}\right]\left[\begin{array}{cc} 1 & \alpha_{N} \\ 0 & 1 \end{array}\right]\cdots \left[\begin{array}{cc} 1 & 0 \\ \beta_{1} & 1 \end{array}\right]\left[\begin{array}{cc} 1 & \alpha_{1} \\ 0 & 1 \end{array}\right] \\
& = & \left[ \begin{array}{cc} g_{1} & r_{1} \\ g_{2} & r_{2} \end{array} \right]\, ,  \label{turbulence}
\end{eqnarray}
and 
$$f = f_{B}^{\, \beta_{K}} \circ f_{A}^{\alpha_{K}} \circ  \cdots  \circ f_{B}^{\, \beta_{1}} \circ f_{A}^{\alpha_{1}}$$
then one has both 
$$\sn(n;p) \equiv f(\sn(1;1))$$
 and 
$$ \left[ \begin{array}{c}  n \\ p \end{array} \right]=  \left[ \begin{array}{cc} g_{1} & r_{1} \\ g_{2} & r_{2} \end{array} \right]\left[ \begin{array}{c}  1 \\ 1 \end{array} \right]= \mt(f) \left[ \begin{array}{c}  1 \\ 1 \end{array} \right]$$

Let us define the set $\mathcal{T}^{+}$ of positively turbulent homeomorphisms relative to $X$ on $\Delta^{+}$ to be the free monoïd generated by $f_{A}$ and $f_{B}$ in the group of the homeomorphims of the plane, and associate to any $f\in \mathcal{T}^{+}$ its matrix of turbulence 
$$\Phi(f) =  \mt(n;p)\, ,$$ 
for $\sn(n;p) = f(\sn(1;1))$ Then $\Phi$ is an isomorphism between $\mathcal{T}^{+}$ and the set of matrices of $\mathrm{SL_{2}}(\Z)$ with positive coefficients. And it can be extended in a very natural fashion to an isomophism between the hole mapping class group of $\R^{2} \setminus X$ and $\mathrm{PSL_{2}}(\Z)$.

To see this, let us deepen our knowledge of the topological snail by slightly extending its signification and adding some colors on it. Let us firts define $\sn(1;0)$ to be the topological type of the segment $[P_{1}; P_{2}]_{\Delta}$ relative to $X$ and  $\sn(0;1)$ the topological type of the segment $[P_{1}; P_{2}]_{\Delta}$ relative to $X$. Let us associate to $\sn(1;0) \equiv \sn(1;0)$ the vector $\left[ \begin{array}{c}  1 \\ 0 \end{array} \right]$ and the green color, and to $\sn(0;1) \equiv \sn(0;1)$ the vector $\left[ \begin{array}{c}  0 \\ 1 \end{array} \right]$ and the red color. On has the following very beautiful result~:

\begin{theo}
Let $f\in \mathcal{T}^{+}$ be a positively turbulent homeomorphism and 
$$\Phi(f) =\mt(f)= \left[ \begin{array}{cc} g_{1} & r_{1} \\ g_{2} & r_{2} \end{array} \right]$$ 
its turbulence matrix. Then $f(\sn(1;0)) = \sn(g_{1}; g_{2})$ and $f(\sn(0;1))=\sn(r_{1}; r_{2})$.
\end{theo} 
\begin{proof}
 One has obvisouly (see section \ref{simple} and figure \ref{figure8}) 
 $$f_{A}(\sn(1;0)) \equiv \sn(1;0)~~\mathrm{and}~~f_{A}(\sn(0;1)) \equiv \sn(1;1)$$
 and 
 $$f_{B}(\sn(1;0)) \equiv \sn(1;0)~~\mathrm{and}~~ f_{B}(\sn(1;0)) \equiv \sn(1;1)$$

This means that the rules $f_{A}(\sn(n;p)) = \sn(n+p, p)$ and $f_{B}(\sn(n;p)) = \sn(n, n+p)$ can be extended to the cases $(n;p)=(1;0)$ and $(n;p)=(0;1)$ to imply this generalised theorem and this deeper result.
\end{proof}

\setlength{\unitlength}{0.5mm}
\begin{figure}
\begin{picture}(100,120)(-150,-65)

\color{black}
\put(-100,0){\line(1,0){60}}
\color{mygreen}
\put(-40,0){\line(1,0){65}}
\color{red}
\put(25,0){\line(1,0){40}}
\color{black}
\put(65,0){\line(1,0){85}}

\put(65,0){\circle*{1.5}}
\put(65,-1){\makebox(0,0)[tc]{$\scriptstyle p'_{3}$}}

\put(25,0){\circle*{1.5}}
\put(25,1){\makebox(0,0)[bc]{$\scriptstyle p'_{2}$}}

\put(-40,0){\circle*{1.5}}
\put(-40,-1){\makebox(0,0)[tc]{$\scriptstyle p'_{1}$}}

\put(-120,10){\makebox(0,0)[br]{$p'_{1}=f(p_{2})$}}

\put(-120,0){\makebox(0,0)[br]{$p'_{2}= f(p_{1})$}}

\put(-120,-10){\makebox(0,0)[br]{$p'_{3}=f(p_{3})$}}

\put(-65,-32){\makebox(0,0)[tr]{$g_{1} = 5$}}
\put(-65,-40){\makebox(0,0)[tr]{$r_{1} = 3$}}
\put(120,-32){\makebox(0,0)[tl]{$ g_{2} = 8$}}
\put(120,-40){\makebox(0,0)[tl]{$r_{2} = 5$}}

\put(-70,70){\makebox(0,0)[tr]{$\Phi(f) =M_{T}(f) = \left[\begin{array}{cc} 5 & 3 \\ 8 & 5 \end{array}\right]$}}

\put(120,70){\makebox(0,0)[tl]{$\begin{array}{cl}  \left[\begin{array}{cc} g_{1} & r_{1} \\ g_{2} & r_{2} \end{array}\right] & \hspace{-0.5cm}\begin{array}{l} \vspace{0.05cm}\mathrm{L} \\  \vspace{-0.05cm}\mathrm{R} \end{array} \\
\vspace{-0.1cm}\mathrm{G~~~~~R~} & \end{array}$}}

\color{mygreen}
\put(-40,0){\qbezier[300](-15,0)(-15,-6.21)(-10.59,-10.59)}
\put(-40,0){\qbezier[300](-10.59,-10.59),(-6.21,-15),(0,-15)}
\put(-40,0){\qbezier[300](0,-15),(6.21,-15)(10.59,-10.59)}
\put(-40,0){\qbezier[300](10.59,-10.59)(15,-6.21)(15,0)}


\color{red}
\put(-40,0){\qbezier[300](-25,0)(-25,-10.35)(-17.65,-17.65)}
\put(-40,0){\qbezier[300](-17.65,-17.65),(-10.35,-25),(0,-25)}
\put(-40,0){\qbezier[300](0,-25),(10.35,-25)(17.65,-17.65)}
\put(-40,0){\qbezier[300](17.65,-17.65)(25,-10.35)(25,0)}

\color{mygreen}
\put(-40,0){\qbezier[100](-35,0)(-35,-14.49)(-24.75,-24.75)}
\put(-40,0){\qbezier[100](-24.75,-24.75)(-14.49,-35)(0,-35)}
\put(-40,0){\qbezier[100](0,-35)(14.49,-35)(24.75,-24.75)}
\put(-40,0){\qbezier[100](24.75,-24.75)(35,-14.49)(35,0)}


\color{mygreen}

\put(65,0){\qbezier[300](-10,0)(-10,-4.14)(-7.07,-7.07)}
\put(65,0){\qbezier[300](-7.07,-7.07),(-4.14,-10),(0,-10)}
\put(65,0){\qbezier[300](0,-10),(4.14,-10)(7.07,-7.07)}
\put(65,0){\qbezier[300](7.07,-7.07)(10,-4.14)(10,0)}

\color{mygreen}
\put(65,0){\qbezier[300](-20,0)(-20,-8.28)(-14.14,-14.14)}
\put(65,0){\qbezier[300](-14.14,-14.14)(-8.28,-20)(0,-20)}
\put(65,0){\qbezier[300](0,-20)(8.28,-20)(14.14,-14.14)}
\put(65,0){\qbezier[300](14.14,-14.14)(20,-8.28)(20,0)}

\color{red}

\put(65,0){\qbezier[300](-30,0)(-30,-12.42)(-21.21,-21.21)}
\put(65,0){\qbezier[300](-21.21,-21.21)(-12.42,-30)(0,-30)}
\put(65,0){\qbezier[300](0,-30)(12.42,-30)(21.21,-21.21)}
\put(65,0){\qbezier[300](21.21,-21.21)(30,-12.42)(30,0)}

\color{mygreen}

\put(65,0){\qbezier[300](-40,0)(-40,-16.56)(-28.28,-28.28)}
\put(65,0){\qbezier[300](-28.28,-28.28)(-16.56,-40)(0,-40)}
\put(65,0){\qbezier[300](0,-40)(16.56,-40)(28.28,-28.28)}
\put(65,0){\qbezier[300](28.28,-28.28)(40,-16.5)(40,0)}

\color{red}

\put(65,0){\qbezier[300](-50,0)(-50,-20.7)(-35.35,-35.35)}
\put(65,0){\qbezier[300](-35.35,-35.35),(-20.7,-50),(0,-50)}
\put(65,0){\qbezier[300](0,-50),(20.7,-50)(35.35,-35.35)}
\put(65,0){\qbezier[300](35.35,-35.35)(50,-20.7)(50,0)}

\color{mygreen}
\put(65,0){\qbezier[300](-60,0)(-60,-24.84)(-42.42,-42.42)}
\put(65,0){\qbezier[300](-42.42,-42.42)(-24.84,-60)(0,-60)}
\put(65,0){\qbezier[300](0,-60)(24.84,-60)(42.42,-42.42)}
\put(65,0){\qbezier[300](42.42,-42.42)(60,-24.84)(60,0)}

\color{red}

\put(25,0){\qbezier[300](-10,0)(-10,4.14)(-7.07,7.07)}
\put(25,0){\qbezier[300](-7.07,7.07),(-4.14,10),(0,10)}
\put(25,0){\qbezier[300](0,10),(4.14,10)(7.07,7.07)}
\put(25,0){\qbezier[300](7.07,7.07)(10,4.14)(10,0)}

\color{mygreen}
\put(25,0){\qbezier[300](-20,0)(-20,8.28)(-14.14,14.14)}
\put(25,0){\qbezier[300](-14.14,14.14)(-8.28,20)(0,20)}
\put(25,0){\qbezier[300](0,20)(8.28,20)(14.14,14.14)}
\put(25,0){\qbezier[300](14.14,14.14)(20,8.28)(20,0)}

\color{mygreen}
\put(25,0){\qbezier[300](-30,0)(-30,12.42)(-21.21,21.21)}
\put(25,0){\qbezier[300](-21.21,21.21)(-12.42,30)(0,30)}
\put(25,0){\qbezier[300](0,30)(12.42,30)(21.21,21.21)}
\put(25,0){\qbezier[300](21.21,21.21)(30,12.42)(30,0)}

\color{red}
\put(25,0){\qbezier[300](-40,0)(-40,16.56)(-28.28,28.28)}
\put(25,0){\qbezier[300](-28.28,28.28)(-16.56,40)(0,40)}
\put(25,0){\qbezier[300](0,40)(16.56,40)(28.28,28.28)}
\put(25,0){\qbezier[300](28.28,28.28)(40,16.56)(40,0)}

\color{mygreen}
\put(25,0){\qbezier[300](-50,0)(-50,20.7)(-35.35,35.35)}
\put(25,0){\qbezier[300](-35.35,35.35),(-20.7,50),(0,50)}
\put(25,0){\qbezier[300](0,50),(20.7,50)(35.35,35.35)}
\put(25,0){\qbezier[300](35.35,35.35)(50,20.7)(50,0)}


\color{mygreen}
\put(25,0){\qbezier[300](-65,0)(-60,24.84)(-42.42,42.42)}
\put(25,0){\qbezier[300](-42.42,42.42)(-24.84,60)(0,60)}
\put(25,0){\qbezier[300](0,60)(24.84,60)(42.42,42.42)}
\put(25,0){\qbezier[300](42.42,42.42)(60,24.84)(60,0)}

\color{red}
\put(25,0){\qbezier[300](-65,0)(-70,28.98)(-49.49,49.49)}
\put(25,0){\qbezier[300](-49.49,49.49)(-28.98,70)(0,70)}
\put(25,0){\qbezier[300](0,70)(28.98,70)(49.49,49.49)}
\put(25,0){\qbezier[300](49.49,49.49)(70,28.98)(70,0)}

\color{mygreen}

\put(25,0){\qbezier[300](-80,0)(-80,31.12)(-56.56,56.56)}
\put(25,0){\qbezier[300](-56.56,56.56)(-31.12,80)(0,80)}
\put(25,0){\qbezier[300](0,80)(31.12,80)(56.56,56.56)}
\put(25,0){\qbezier[300](56.56,56.56)(80,31.12)(80,0)}

\color{red}

\put(25,0){\qbezier[300](-90,0)(-90,37.26)(-63.63,63.63)}
\put(25,0){\qbezier[300](-63.63,63.63)(-37.26,90)(0,90)}
\put(25,0){\qbezier[300](0,90)(37.26,90)(63.63,63.63)}
\put(25,0){\qbezier[300](63.63,63.63)(90,37.26)(90,0)}

\color{mygreen}
\put(25,0){\qbezier[300](-100,0)(-100,41.41)(-70.7,70.7)}
\put(25,0){\qbezier[300](-70.7,70.7),(-41.4,100),(0,100)}
\put(25,0){\qbezier[300](0,100),(41.4,100)(70.7,70.7)}
\put(25,0){\qbezier[300](70.7,70.7)(100,41.41)(100,0)}

\color{black}

\put(65,-40){\vector(1,0){0}}
\put(25,80){\vector(-1,0){0}}
\put(-40,-15){\vector(1,0){0}}
\put(25,50){\vector(1,0){0}}
\put(65,-10){\vector(-1,0){0}}
\put(25,30){\vector(-1,0){0}}

\put(-40,-35){\vector(-1,0){0}}
\put(25,80){\vector(-1,0){0}}
\put(25,100){\vector(1,0){0}}
\put(25,50){\vector(1,0){0}}
\put(65,-60){\vector(-1,0){0}}
\put(25,20){\vector(1,0){0}}

\put(65,-20){\vector(1,0){0}}
\put(25,60){\vector(-1,0){0}}
\put(25,70){\vector(1,0){0}}
\put(65,-30){\vector(-1,0){0}}
\put(25,10){\vector(-1,0){0}}
\put(65,-50){\vector(1,0){0}}

\put(25,90){\vector(-1,0){0}}
\put(-40,-25){\vector(1,0){0}}
\put(25,40){\vector(1,0){0}}

\put(25,90){\vector(-1,0){0}}

\end{picture}
\caption{Natural interpretation of $\sn^{+}(8,13)$ as $\sn^{+}(5; 8) \sqcup \sn^{+}(3;5)$ .\label{figure31}}

\end{figure}
\setlength{\unitlength}{0.7mm}

\begin{remarks}

\item There is a natural way to split the topological snail $\sn(n;p)$ into its green and red parts, using a very general and important topological consideration. The snail $\sn(1;1)$ can be represented as a curve $\gamma$ very close to  $\sn(1;0) \cup \sn(0;1)$ although not containing the point $P_{2}$, and one can add a little vertical arc from $P_{2}$ to a given point on $\sn(n;p)$, and separate $\sn(1;1)$ into colors green and red, inducing a positive orientation on this little arc from $P_{2}$, colouring the left part of $\gamma$ in green and the right in red. During the successive applications of the type $A$ and type $B$ homeomorphisms, one can always keep this little arc just in front of the free point on the snail, as a mediator segment in the liberty chamber. The splitting of $\sn(n;p)$ into its green part $\sn(g_{1}; g_{2})$ and its red part $\sn(r_{1};r_{2})$ is therefore obtained directly from a splitting of the frontal arc in the chamber, naturally oriented from the free point inside the liberty chamber, taking the green arc on the left and the red on the right.The numbers $g_{1}$, $g_{2}$, $r_{1}$, $r_{2}$ can be found directly from the colouring of the snail as in figures \ref{figure19} to \ref{figure21}. It can also be taken directy at the center of the only half-disc of $\sn(n;p)$ that has an even radius, oriented from the center to the interior of the half-disk. One obtain's a direct result collapsing the hole liberty chamber to a point (see figure \ref{figure31}), just remembering the induced orientation on the snail.

\item This is also a direct mean to determine the images by $f$ of $P_{1}$ and $P_{3}$ : the only extremity on the green side is the image of $P_{1}$, and the only extremity on the red side the image of $P_{2}$. At the end of the transformation, one can contract the segment to $P_{3}$, obtaining the splitting  $f(\sn(1;0) \cup \sn(0;1)) = f(\sn(1;0)) \sqcup f(\sn(0;1))$. This leads to a new interpretation of $\sn(n;p)$, as the $f\left(\sn(1;0) \sqcup\sn(0;1)\right)$, for an homeomorphism $f$ such that $f(\sn(1;1))\equiv \sn(n;p)$.

\item The splitting of the topological snail produces a direct interpretation of the turbulence matrix $\mt(f)$ for $f\in \mathcal{T}^{+}$. In particular, one obtains graphically Bezout couples for $(n;p)$~: one always has $g_{1} p - g_{2}n = r_{2} n - r_{1}p = 1$.  For example, in figure \ref{figure19}~: $$5 \times 8 - 3 \times 13= 5\times 13 - 8 \times 8 = 1\,.$$ 

\item Of course, it is natural to consider the colored topological snail as the image by a positively turbulent homeomorphism $f\in \mathcal{T}^{+}$ of the segment $[P_{1}; P_{3}]_{\Delta}$, with $[P_{1}; P_{2}]_{\Delta}$ colored in green and $[P_{2}; P_{3}]_{\Delta}$ colored in red. Note that the entrance of the snail in green means that the last top-displacement of the middle point composing $f$ is an homeomorphism of the type $B$, and the color green for the shell means that the first top-displacement of the middle point composing $f$ is an homeomorphism of the type $B$. This is the case in figure \ref{figure31}, where $f\equiv f_{B} \circ f_{A} \circ f_{B} \circ f_{A} \circ f_{B}$.  

\item There is a natural orientation on the line $\Delta^{+}$ that coincides with the orientation from the green to the red, the classical orientation from left to right. When $f\in \mathcal{T}^{+}$ is a positively turbulent homeomorphism, then $f$ induces an orientation on the emerging visible arc of $\sn(n;p)$ that is also directed from the left to the right : this an important fact that can be deduced from the proof of theorem \ref{decomposition} (see figure \ref{figure31}).

\item The homeomorphisms of the types $A$ and $B$ can be understood in another very important manner using the colored arcs $\sn(1;0)$ and $\sn(0;1)$~: 

\begin{itemize}
\item Put a thin rectangular box all around $\sn(0;1)$, and imagine an isotopy that takes the middle point $P_{2}$ to $P_{3}$, on its left and reversing the arc $\sn(0;1)$ on itself, keeping the image of $\sn(0;1)$ inside the thin box. This is of course an homeomorphism of the type $B$, and its action on $\sn(1;0) \sqcup \sn(0;1)$ can be thought of as a very concrete operation on curves : cut the curve $\sn(1;0) \sqcup \sn(0;1)$ to separate $\sn(1;0)$ and $\sn(0;1)$, take the former common extremity $P_{2}$ of $\sn(1;0)$ to the other former free extremity of $\sn(0;1)$, without changing $\sn(0;1)$ and passing on its left for the positive orientation of $\Delta$. Then reverse the orientation on $\sn(0;1)$ to obtain a coherent one with that of $\sn(1;0)$. 

\item Equivalently cut the curve $\sn(1;0) \sqcup \sn(0;1)$ to separate $\sn(1;0)$ and $\sn(0;1)$, duplicate $\sn(0;1)$ into close two copies of itself with a parallel orientation, connect the green $\sn(0;1)$ to the left red curve, changing its color to green and keeping its orientation, and connect the new red extremity obtained to the corresponding end of the second copy of $\sn(0;1)$, keeping its red color and reversing its orientation. 

\item On can also describe this operation in the following manner : cut $\sn(1;0) \sqcup \sn(0;1)$ to $\sn(1;0)$ and $\sn(0;1)$, put on the left of $\sn(0;1)$ a green copy of itself and reverse its own orientation orientation. Connect the obtained three curves according to their orientations.

\end{itemize}
\setlength{\unitlength}{0.5mm}
\begin{figure}

\begin{picture}(100,400)(-140,-330)



\put(-90,20){\line(1,0){60}}
\put(40,20){\line(1,0){50}}
\put(-90,18){\makebox(0,0)[tc]{$\Delta$}}

\color{mygreen}
\put(-30,20){\line(1,0){40}}
\color{red}
\put(10,20){\line(1,0){30}}

\color{black}
\put(40,20){\circle*{1}}
\put(10,20){\circle*{1}}
\put(-30,20){\circle*{1}}

\color{blue}

\color{red}

\color{red}
\put(40,20){\qbezier[300](-30,0)(-30,-12.42)(-21.21,-21.21)}
\put(40,20){\qbezier[300](-21.21,-21.21)(-12.42,-30)(0,-30)}
\put(40,20){\qbezier[300](0,-30)(12.42,-30)(21.21,-21.21)}
\put(40,20){\qbezier[300](21.21,-21.21)(30,-12.42)(30,0)}

\color{red}
\put(-30,20){\qbezier[300](-20,0)(-20,-8.28)(-14.14,-14.14)}
\put(-30,20){\qbezier[300](-14.14,-14.14)(-8.28,-20)(0,-20)}
\put(-30,20){\qbezier[300](0,-20)(8.28,-20)(14.14,-14.14)}
\put(-30,20){\qbezier[300](14.14,-14.14)(20,-8.28)(20,0)}

\color{red}
\put(10,20){\qbezier[300](-20,0)(-20,8.28)(-14.14,14.14)}
\put(10,20){\qbezier[300](-14.14,14.14)(-8.28,20)(0,20)}
\put(10,20){\qbezier[300](0,20)(8.28,20)(14.14,14.14)}
\put(10,20){\qbezier[300](14.14,14.14)(20,7,64)(30,0)}

\color{mygreen}
\put(10,20){\qbezier[300](-40,0)(-40,16.56)(-28.28,28.28)}
\put(10,20){\qbezier[300](-28.28,28.28)(-16.56,40)(0,40)}
\put(10,20){\qbezier[300](0,40)(16.56,40)(28.28,28.28)}
\put(10,20){\qbezier[300](28.28,28.28)(40,16.56)(30,0)}

\color{red}
\put(10,20){\qbezier[300](-60,0)(-60,24.84)(-42.42,42.42)}
\put(10,20){\qbezier[300](-42.42,42.42)(-24.84,60)(0,60)}
\put(10,20){\qbezier[300](0,60)(24.84,60)(42.42,42.42)}
\put(10,20){\qbezier[300](42.42,42.42)(60,24.84)(60,0)}

\color{black}

\put(10,80){\vector(1,0){0}}
\put(10,60){\vector(1,0){0}}
\put(10,40){\vector(-1,0){0}}

\put(40,-10){\vector(-1,0){0}}
\put(-30,0){\vector(-1,0){0}}


\put(-90,-120){\line(1,0){60}}
\put(40,-120){\line(1,0){50}}
\put(-90,-122){\makebox(0,0)[tc]{$\Delta$}}

\color{mygreen}
\put(-30,-120){\line(1,0){40}}
\color{red}
\put(10,-120){\line(1,0){30}}

\put(40,-120){\circle*{1}}
\put(10,-120){\circle*{1}}
\put(-30,-120){\circle*{1}}

\color{blue}

\color{red}
\put(10,-120){\qbezier[300](-10,0)(-10,4.14)(-7.07,7.07)}
\put(10,-120){\qbezier[300](-7.07,7.07),(-4.14,10),(0,10)}
\put(10,-120){\qbezier[300](0,10),(4.14,10)(7.07,7.07)}
\put(10,-120){\qbezier[300](7.07,7.07)(10,4.14)(15,0)}

\color{red}
\put(-30,-120){\qbezier[300](-30,0)(-30,-12.42)(-21.21,-21.21)}
\put(-30,-120){\qbezier[300](-21.21,-21.21)(-12.42,-30)(0,-30)}
\put(-30,-120){\qbezier[300](0,-30)(12.42,-30)(21.21,-21.21)}
\put(-30,-120){\qbezier[300](21.21,-21.21)(30,-12.42)(30,0)}

\color{red}
\put(10,-120){\qbezier[300](-70,0)(-70,28.98)(-49.49,49.49)}
\put(10,-120){\qbezier[300](-49.49,49.49)(-28.98,70)(0,70)}
\put(10,-120){\qbezier[300](0,70)(28.98,70)(49.49,49.49)}
\put(10,-120){\qbezier[300](49.49,49.49)(70,28.98)(70,0)}

\color{red}

\put(40,-120){\qbezier[300](-35,-5)(-40,-16.56)(-28.28,-28.28)}
\put(40,-120){\qbezier[300](-28.28,-28.28)(-16.56,-40)(0,-40)}
\put(40,-120){\qbezier[300](0,-40)(16.56,-40)(28.28,-28.28)}
\put(40,-120){\qbezier[300](28.28,-28.28)(40,-16.56)(40,0)}

\color{red}
\put(40,-120){\qbezier[300](-30,0)(-30,-12.42)(-21.21,-21.21)}
\put(40,-120){\qbezier[300](-21.21,-21.21)(-12.42,-30)(0,-30)}
\put(40,-120){\qbezier[300](0,-30)(12.42,-30)(21.21,-21.21)}
\put(40,-120){\qbezier[300](21.21,-21.21)(30,-12.42)(30,0)}

\color{red}
\put(-30,-120){\qbezier[300](-20,0)(-20,-8.28)(-14.14,-14.14)}
\put(-30,-120){\qbezier[300](-14.14,-14.14)(-8.28,-20)(0,-20)}
\put(-30,-120){\qbezier[300](0,-20)(8.28,-20)(14.14,-14.14)}
\put(-30,-120){\qbezier[300](14.14,-14.14)(20,-8.28)(20,0)}

\color{red}
\put(10,-120){\qbezier[300](-20,0)(-20,8.28)(-14.14,14.14)}
\put(10,-120){\qbezier[300](-14.14,14.14)(-8.28,20)(0,20)}
\put(10,-120){\qbezier[300](0,20)(8.28,20)(14.14,14.14)}
\put(10,-120){\qbezier[300](14.14,14.14)(20,7,64)(30,0)}

\color{mygreen}
\put(10,-120){\qbezier[300](-40,0)(-40,16.56)(-28.28,28.28)}
\put(10,-120){\qbezier[300](-28.28,28.28)(-16.56,40)(0,40)}
\put(10,-120){\qbezier[300](0,40)(16.56,40)(28.28,28.28)}
\put(10,-120){\qbezier[300](28.28,28.28)(40,16.56)(40,0)}

\put(40,-120){\qbezier[300](0,-10),(4.14,-10)(7.07,-7.07)}
\put(40,-120){\qbezier[300](7.07,-7.07)(10,-4.14)(10,0)}


\color{red}
\put(10,-120){\qbezier[300](-60,0)(-60,24.84)(-42.42,42.42)}
\put(10,-120){\qbezier[300](-42.42,42.42)(-24.84,60)(0,60)}
\put(10,-120){\qbezier[300](0,60)(24.84,60)(42.42,42.42)}
\put(10,-120){\qbezier[300](42.42,42.42)(60,24.84)(60,0)}

\color{black}
\put(40,-130){\vector(-1,0){0}}
\put(10,-110){\vector(-1,0){0}}
\put(-30,-150){\vector(-1,0){0}}
\put(10,-50){\vector(1,0){0}}
\put(40,-160){\vector(-1,0){0}}

\put(10,-60){\vector(1,0){0}}
\put(10,-80){\vector(1,0){0}}
\put(10,-100){\vector(-1,0){0}}

\put(5,-125){\vector(1,3){0}}

\put(40,-150){\vector(-1,0){0}}
\put(-30,-140){\vector(-1,0){0}}

\color{black}


\color{black}
\put(-90,-260){\line(1,0){60}}
\put(40,-260){\line(1,0){50}}
\put(-90,-262){\makebox(0,0)[tc]{$\Delta$}}

\color{mygreen}
\put(-30,-260){\line(1,0){40}}
\color{red}
\put(10,-260){\line(1,0){30}}

\put(40,-260){\circle*{1}}
\put(10,-260){\circle*{1}}
\put(-30,-260){\circle*{1}}


\color{mygreen}
\put(10,-260){\qbezier[300](-10,0)(-10,4.14)(-7.07,7.07)}
\put(10,-260){\qbezier[300](-7.07,7.07),(-4.14,10),(0,10)}
\put(10,-260){\qbezier[300](0,10),(4.14,10)(7.07,7.07)}
\put(10,-260){\qbezier[300](7.07,7.07)(10,4.14)(15,0)}

\color{mygreen}
\put(-30,-260){\qbezier[300](-30,0)(-30,-12.42)(-21.21,-21.21)}
\put(-30,-260){\qbezier[300](-21.21,-21.21)(-12.42,-30)(0,-30)}
\put(-30,-260){\qbezier[300](0,-30)(12.42,-30)(21.21,-21.21)}
\put(-30,-260){\qbezier[300](21.21,-21.21)(30,-12.42)(30,0)}

\color{mygreen}
\put(10,-260){\qbezier[300](-70,0)(-70,28.98)(-49.49,49.49)}
\put(10,-260){\qbezier[300](-49.49,49.49)(-28.98,70)(0,70)}
\put(10,-260){\qbezier[300](0,70)(28.98,70)(49.49,49.49)}
\put(10,-260){\qbezier[300](49.49,49.49)(70,28.98)(70,0)}

\color{mygreen}

\put(40,-260){\qbezier[300](-30,0)(-40,-16.56)(-28.28,-28.28)}
\put(40,-260){\qbezier[300](-28.28,-28.28)(-16.56,-40)(0,-40)}
\put(40,-260){\qbezier[300](0,-40)(16.56,-40)(28.28,-28.28)}
\put(40,-260){\qbezier[300](28.28,-28.28)(40,-16.56)(40,0)}

\color{red}
\put(40,-260){\qbezier[300](-30,0)(-30,-12.42)(-21.21,-21.21)}
\put(40,-260){\qbezier[300](-21.21,-21.21)(-12.42,-30)(0,-30)}
\put(40,-260){\qbezier[300](0,-30)(12.42,-30)(21.21,-21.21)}
\put(40,-260){\qbezier[300](21.21,-21.21)(30,-12.42)(30,0)}

\color{red}
\put(-30,-260){\qbezier[300](-20,0)(-20,-8.28)(-14.14,-14.14)}
\put(-30,-260){\qbezier[300](-14.14,-14.14)(-8.28,-20)(0,-20)}
\put(-30,-260){\qbezier[300](0,-20)(8.28,-20)(14.14,-14.14)}
\put(-30,-260){\qbezier[300](14.14,-14.14)(20,-8.28)(20,0)}

\color{red}
\put(10,-260){\qbezier[300](-20,0)(-20,8.28)(-14.14,14.14)}
\put(10,-260){\qbezier[300](-14.14,14.14)(-8.28,20)(0,20)}
\put(10,-260){\qbezier[300](0,20)(8.28,20)(14.14,14.14)}
\put(10,-260){\qbezier[300](14.14,14.14)(20,7,64)(30,0)}

\color{mygreen}
\put(10,-260){\qbezier[300](-40,0)(-40,16.56)(-28.28,28.28)}
\put(10,-260){\qbezier[300](-28.28,28.28)(-16.56,40)(0,40)}
\put(10,-260){\qbezier[300](0,40)(16.56,40)(28.28,28.28)}
\put(10,-260){\qbezier[300](28.28,28.28)(40,16.56)(40,0)}

\put(40,-260){\qbezier[300](0,-10),(4.14,-10)(7.07,-7.07)}
\put(40,-260){\qbezier[300](7.07,-7.07)(10,-4.14)(10,0)}

\put(40,-260){\qbezier[300](0,-10),(-4.14,-10)(-7.07,-7.07)}
\put(40,-260){\qbezier[300](-7.07,-7.07)(-10,-4.14)(-15,0)}

\color{black}
\put(40,-130){\vector(-1,0){0}}
\put(10,-110){\vector(-1,0){0}}
\put(-30,-150){\vector(-1,0){0}}
\put(10,-50){\vector(1,0){0}}
\put(40,-160){\vector(-1,0){0}}

\color{red}
\put(10,-260){\qbezier[300](-60,0)(-60,24.84)(-42.42,42.42)}
\put(10,-260){\qbezier[300](-42.42,42.42)(-24.84,60)(0,60)}
\put(10,-260){\qbezier[300](0,60)(24.84,60)(42.42,42.42)}
\put(10,-260){\qbezier[300](42.42,42.42)(60,24.84)(60,0)}

\color{black}


\color{black}

\put(10,-190){\vector(1,0){0}}
\put(10,-220){\vector(1,0){0}}
\put(10,-200){\vector(-1,0){0}}
\put(10,-240){\vector(1,0){0}}

\put(-30,-280){\vector(1,0){0}}
\put(-30,-290){\vector(-1,0){0}}

\put(40,-130){\vector(-1,0){0}}
\put(40,-270){\vector(-1,0){0}}

\put(10,-250){\vector(-1,0){0}}
\put(40,-290){\vector(1,0){0}}
\put(40,-300){\vector(-1,0){0}}

\color{black}

\color{black}

\color{blue}

\end{picture}
\caption{Green inner growth of $\sn(3;4)$ to $\sn(5;7)$.\label{figure32}}

\end{figure}

\setlength{\unitlength}{0.7mm}

This operation can be extended to the topological snail using a positively turbulent homeomorphism $f\in \mathcal{T}^{+}$ : the reciprocal image by $f$ of a tubular neighborhood of the red $\sn(r_{1}, r_{2})$ inside $\sn(n;p) = f(\sn(1;0) \sqcup \sn(0;1))$ will contain such a thin rectangular box, and the preceeding operation will be possible to do directly in the snail. I call this description the green inner growth of $f$ or $\sn(n;p)$ (see figure \ref{figure32}). It is called green inner growth for two good reasons. Firstly, it happens inside the snail and it appears as an inner growth of the green snail $\sn(g_{1}, g_{2})$ to $\sn(g_{1}+r_{1}, g_{2}+r_{2})$, keeping the red snail $\sn(r_{1}, r_{2})$ almost unchanged (see figure \ref{figure32}). And secondly, it corresponds to the composition on the right, inside $f$, of the generating homeomorphism $f \in \mathcal{T}^{+}$ associated to $\sn(n;p)$: the green inner growth corresponds to the equation~:
\begin{eqnarray*}
\Phi(f \circ f_{B}) & = & \Phi(f) \times \Phi(f_{B}) \\
& = & \left[ \begin{array}{cc} g_{1} & r_{1} \\ g_{2} & r_{2} \end{array} \right] \times \left[ \begin{array}{cc} 1 & 0 \\ 1 & 1 \end{array} \right] \\
& = & \left[ \begin{array}{cc} g_{1} +r_{1}& r_{1} \\ g_{2}+r_{2} & r_{2} \end{array} \right] 
\end{eqnarray*}

Of course, the red inner growth corresponds to a composition inside $f$ by $f_{A}$. Note that a new visible emerging arc appears at each inner growth, with the color of this inner growth, but that the color of the entrance of the snail remains the same. It is the exact opposite, from this point of view, of the unfolding, and this is not suprising since the inner growth of $f$ is the outer growth of $f^{-1}$ (see next section). A new proof of the theorem \ref{decomposition} appears, since the color of the visible emerging arcs gives the information to reverse the inner growth and obtain an algorithm to reduce $f$ from the inside. Note that the inner growth is a fundamental tool to understand the topological structure of the set of Nielsen classes of fixed points of $f$.

\end{remarks}

\section{The inner universal tree}

The preceeding reflexion allows one to understand the distribution of the colors, green or red, and the orientation upward or downward, of the points on $\sn(n;p) \cap \Delta^{+}$, by constructing a beautiful universal two-coloured inner tree. The snail associated to $f \in \mathcal{T}^{+}$ contains a sequence of vertical arrows, green or red, in top or bottom direction, that I will call its horizontal arrow distribution. It can be obtained directly from a very simple principle. The green operation traduces the application of a type B homeomorphism inside a homeomorphism $f\in \mathcal{T}^{+}$ : if the distribution of arrows associated to $f$ is given, it consists in keeping unchanged the green arrows and replace each red one by a green one in the same position with a additional red one on its right, in opposite direction. The red operation is exactly symetric, associated to  the application of a type A homeomorphism inside a homeomorphism $f$. One therefore only has to remember $\GT|\RB$ both as the initial data and the figure to place according to the red or the right arrow and its direction. One can also consider this data as a word in the letter $G^{+}$, $G^{-}$, $R^{+}$, $R_{-}$. The initial data will be $G^{+}|R^{-}$, and the associated substitutions~:

$$\begin{array}{l} B~: \left\lbrace \begin{array}{l} R^{-} \mapsto R^{+}G^{-} \\R^{+} \mapsto G^{+}R^{-}\end{array}\right. \\ A~: \left\lbrace \begin{array}{l} G^{-} \mapsto G^{+}R^{-} \\ G^{+} \mapsto R^{+}G^{-} \end{array}\right. \end{array}$$

For example, $\sn(3;10)$ is of the type $B^{3}A^{2}$, and its horizontal distribution will be obtained applying successively the corresponding operations in the reading order : 
$$\begin{array}{rrcl} 
 &\GT & | & \RB \\ 
 B~:&\GT  & | & \boxed{\RT \GB} \\
B^{2}~:&\GT  & | & \boxed{\GT \RB} \GB \\
B^{3}~:&\GT  & | & \GT \boxed{\RT \GB} \GB \\
B^{3}A~:&\boxed{\RT \GB}  & | & \boxed{\RT \GB} \RT \boxed{\GT \RB} \boxed{\GT \RB} \\
B^{3} A^{2}~:& \RT \boxed{\GT \RB}  & | & \RT \boxed{\GT \RB} \RT \boxed{\RT \GB}\RB \boxed{\RT \GB} \RB \\
& & & \\
B^{3} A^{2}~:& \RT~\GT~\RB  & | & \RT~\GT~\RB~\RT~\RT~\GB~\RB~\RT~\GB~\RB 
\end{array}$$

In order to check this process, one just has to draw the coloured $\sn(3;10)$ and put the positive orientation on every point of $\sn(3;10) \cap \Delta^{+}$, according to figure \ref{figure33}.

\setlength{\unitlength}{0.6mm}
\begin{figure}
\begin{picture}(100,120)(-90,-85)

\color{black}
\put(-50,0){\line(1,0){35}}
\put(-50,-2){\makebox(0,0)[tr]{$\Delta$}}
\color{mygreen}
\put(-15,0){\line(1,0){50}}
\color{red}
\put(35,0){\line(1,0){15}}
\color{black}
\put(45,0){\line(1,0){65}}
\put(50,0){\circle*{1}}
\put(35,0){\circle*{1}}

\color{red}
\put(-25,0){\vector(0,1){8}}
\put(-25,0){\circle*{1}}
\put(-5,0){\vector(0,-1){8}}
\put(-5,0){\circle*{1}}
\put(5,0){\vector(0,1){8}}
\put(5,0){\circle*{1}}
\put(25,0){\vector(0,-1){8}}
\put(25,0){\circle*{1}}
\put(35,0){\vector(0,1){8}}

\put(45,0){\vector(0,1){8}}
\put(45,0){\circle*{1}}
\put(65,0){\vector(0,-1){8}}
\put(65,0){\circle*{1}}
\put(75,0){\vector(0,1){8}}
\put(75,0){\circle*{1}}
\put(95,0){\vector(0,-1){8}}
\put(95,0){\circle*{1}}
\color{mygreen}
\put(55,0){\vector(0,-1){8}}
\put(55,0){\circle*{1}}
\put(85,0){\vector(0,-1){8}}
\put(85,0){\circle*{1}}
\put(-15,0){\vector(0,1){8}}

\put(15,0){\vector(0,1){8}}
\put(15,0){\circle*{1}}

\color{black}
\put(0,-5){\line(0,1){10}}
\put(50,0){\vector(0,-1){5}}
\put(0,-5){\line(0,1){10}}


\color{red}
\put(-25,-60){\vector(0,1){8}}
\put(-25,-60){\circle*{1}}

\put(-5,-53){\vector(0,-1){8}}
\put(-5,-53){\circle*{1}}

\put(5,-60){\vector(0,1){8}}
\put(5,-60){\circle*{1}}

\put(25,-53){\vector(0,-1){8}}
\put(25,-53){\circle*{1}}

\put(35,-60){\vector(0,1){8}}
\put(35,-60){\circle*{1}}

\put(45,-60){\vector(0,1){8}}
\put(45,-60){\circle*{1}}

\put(65,-53){\vector(0,-1){8}}
\put(65,-53){\circle*{1}}

\put(75,-60){\vector(0,1){8}}
\put(75,-60){\circle*{1}}

\put(95,-53){\vector(0,-1){8}}
\put(95,-53){\circle*{1}}

\color{mygreen}

\put(55,-53){\vector(0,-1){8}}
\put(55,-53){\circle*{1}}

\put(85,-53){\vector(0,-1){8}}
\put(85,-53){\circle*{1}}

\put(-15,-60){\vector(0,1){8}}
\put(-15,-60){\circle*{1}}
\put(15,-60){\vector(0,1){8}}
\put(15,-60){\circle*{1}}

\color{black}
\put(0,-62){\line(0,1){11}}

\put(-15,0){\circle*{1}}

\put(-25,-50){\qbezier[50](0,0)(0,25)(0,50)}
\put(-15,-50){\qbezier[50](0,0)(0,25)(0,50)}
\put(-5,-50){\qbezier[50](0,0)(0,25)(0,50)}
\put(5,-50){\qbezier[50](0,0)(0,25)(0,50)}
\put(15,-50){\qbezier[50](0,0)(0,25)(0,50)}
\put(25,-50){\qbezier[50](0,0)(0,25)(0,50)}
\put(35,-50){\qbezier[50](0,0)(0,25)(0,50)}
\put(45,-50){\qbezier[50](0,0)(0,25)(0,50)}
\put(55,-50){\qbezier[50](0,0)(0,25)(0,50)}
\put(65,-50){\qbezier[50](0,0)(0,25)(0,50)}
\put(75,-50){\qbezier[50](0,0)(0,25)(0,50)}
\put(85,-50){\qbezier[50](0,0)(0,25)(0,50)}


\color{red}
\put(-15,0){\qbezier[300](-10,0)(-10,-4.14)(-7.07,-7.07)}
\put(-15,0){\qbezier[300](-7.07,-7.07),(-4.14,-10),(0,-10)}
\put(-15,0){\qbezier[300](0,-10),(4.14,-10)(7.07,-7.07)}
\put(-15,0){\qbezier[300](7.07,-7.07)(10,-4.14)(10,0)}


\color{red}
\put(50,0){\qbezier[300](-5,0)(-5,-2.07)(-3.53,-3.53)}
\put(50,0){\qbezier[300](-3.53,-3.53),(-2.07,-5),(0,-5)}

\color{mygreen}
\put(50,0){\qbezier[300](0,-5),(2.07,-5)(3.53,-3.53)}
\put(50,0){\qbezier[300](3.53,-3.53)(5,-2.07)(5,0)}

\color{red}
\put(50,0){\qbezier[300](-15,0)(-15,-6.21)(-10.59,-10.59)}
\put(50,0){\qbezier[300](-10.59,-10.59),(-6.21,-15),(0,-15)}
\put(50,0){\qbezier[300](0,-15),(6.21,-15)(10.59,-10.59)}
\put(50,0){\qbezier[300](10.59,-10.59)(15,-6.21)(15,0)}


\color{red}
\put(50,0){\qbezier[300](-25,0)(-25,-10.35)(-17.65,-17.65)}
\put(50,0){\qbezier[300](-17.65,-17.65),(-10.35,-25),(0,-25)}
\put(50,0){\qbezier[300](0,-25),(10.35,-25)(17.65,-17.65)}
\put(50,0){\qbezier[300](17.65,-17.65)(25,-10.35)(25,0)}

\color{mygreen}
\put(50,0){\qbezier[100](-35,0)(-35,-14.49)(-24.75,-24.75)}
\put(50,0){\qbezier[100](-24.75,-24.75)(-14.49,-35)(0,-35)}
\put(50,0){\qbezier[100](0,-35)(14.49,-35)(24.75,-24.75)}
\put(50,0){\qbezier[100](24.75,-24.75)(35,-14.49)(35,0)}

\color{red}
\put(50,0){\qbezier[300](-45,0)(-45,-18.63)(-31.81,-31.81)}
\put(50,0){\qbezier[300](-31.81,-31.81),(-18.63,-45),(0,-45)}
\put(50,0){\qbezier[300](0,-45),(18.63,-45)(31.81,-31.81)}
\put(50,0){\qbezier[300](31.81,-31.81)(45,-18.63)(45,0)}

\color{red}

\put(35,0){\qbezier[300](-10,0)(-10,4.14)(-7.07,7.07)}
\put(35,0){\qbezier[300](-7.07,7.07),(-4.14,10),(0,10)}
\put(35,0){\qbezier[300](0,10),(4.14,10)(7.07,7.07)}
\put(35,0){\qbezier[300](7.07,7.07)(10,4.14)(10,0)}

\color{mygreen}
\put(35,0){\qbezier[300](-20,0)(-20,8.28)(-14.14,14.14)}
\put(35,0){\qbezier[300](-14.14,14.14)(-8.28,20)(0,20)}
\put(35,0){\qbezier[300](0,20)(8.28,20)(14.14,14.14)}
\put(35,0){\qbezier[300](14.14,14.14)(20,8.28)(20,0)}

\color{red}
\put(35,0){\qbezier[300](-30,0)(-30,12.42)(-21.21,21.21)}
\put(35,0){\qbezier[300](-21.21,21.21)(-12.42,30)(0,30)}
\put(35,0){\qbezier[300](0,30)(12.42,30)(21.21,21.21)}
\put(35,0){\qbezier[300](21.21,21.21)(30,12.42)(30,0)}

\color{red}
\put(35,0){\qbezier[300](-40,0)(-40,16.56)(-28.28,28.28)}
\put(35,0){\qbezier[300](-28.28,28.28)(-16.56,40)(0,40)}
\put(35,0){\qbezier[300](0,40)(16.56,40)(28.28,28.28)}
\put(35,0){\qbezier[300](28.28,28.28)(40,16.56)(40,0)}

\color{mygreen}
\put(35,0){\qbezier[300](-50,0)(-50,20.7)(-35.35,35.35)}
\put(35,0){\qbezier[300](-35.35,35.35),(-20.7,50),(0,50)}
\put(35,0){\qbezier[300](0,50),(20.7,50)(35.35,35.35)}
\put(35,0){\qbezier[300](35.35,35.35)(50,20.7)(50,0)}


\color{red}
\put(35,0){\qbezier[300](-60,0)(-60,24.84)(-42.42,42.42)}
\put(35,0){\qbezier[300](-42.42,42.42)(-24.84,60)(0,60)}
\put(35,0){\qbezier[300](0,60)(24.84,60)(42.42,42.42)}
\put(35,0){\qbezier[300](42.42,42.42)(60,24.84)(60,0)}

\color{black}

\end{picture}
\caption{Distribution of coloured arrows on $\sn^{+}(3;10)$.\label{figure33}}

\end{figure}
\setlength{\unitlength}{0.7mm}

But there is another graphical mean for finding the distribution of coloured arrows on $\sn(n;p)$ by applying a specific algorithm traducing the eucidean algorithm in terms of coloured trees. The principe is very simple. The tree begins with two branches~: the green on left, oriented upward; and the  red on right, oriented donward. For each homeomorphism of the type $A$, one draws a red horizontal line, and for each of the type $B$ one draws a green horizontal line, according to the code defining the type of the homeomorphism, the reading order corresponding to the upward direction. For example on figure \ref{figure34}, I have drawn the inner tree $\tn(3;10)$. It corresponds to the type $B^{3}A^{2}$ since $\sn(3;10)= \Phi(f)$ is $f$ is of this type. So from the botttom to the top there are three green and two red horizontal lines. The branches of a given color are not affected by the horizontal lines of their own color. But branches of a given colour change direction and orientation passing thrue an horizontal line of a different colour, giving birth to additional branches with the initial orientation and direction but the opposite colour (see figure \ref{figure31}). A natural and universal recursive structure develops, closely linked with the structure of the topological snail, revealing what happens in the inside of the snail. Branching points have colors according to the color of the line in which thay arise, each end of the tree is connected to a unique branching point in the tree, and one obtains the following new theorem~:

\begin{theo}
Let $(n;p)$ be two positive integers, $f\in \mathcal{T}^{+}$ such that $\sn(n;p) = \Phi(f)$ and $\tn(n;p)$ the inner two-colored tree associated. Let $L_{g}$ be the number of green branching points on the left branch of $\tn(n;p)$, and respectively $L_{r}$ for the red branching points on its left branch and $R_{g}$, $R_{r}$ for the branching points of colors green and red on the right branch of $\tn(n;p)$. Then 
$$\mt(f) = \left[\begin{array}{cc} L_{g}(f) & L_{r}(f) \\ R_{g}(f) & R_{r}(f) \end{array}\right] $$
Furthermore, if $f$ and $g$ are positive positively turbulent homeomorphisms, the tree $\tn(f\circ g)$ associated to $f\circ g$ is obtained by substituting to each green end of $\tn(f)$ the left (green) branch of $\tn(g)$ or its symetric figure according to the direction of this end, and to each red end of $\tn(f)$ the right branch of $\tn(g)$ or its symetric figure according to the same principle, respecting the rule~:

$$\mt(f \circ g) =  \mt(f) \times \mt(g)$$

\end{theo} 

\begin{remarks}
   
\item If one considers $\sn(3;10)$, associated to $f = f_{B}^{3}f_{A}^{2}$, one can has 
$$\mt(f_{B}^{3}) = \left[\begin{array}{cc} 1 & 0 \\ 3 & 1 \end{array}\right]~~\mathrm{and}~~\mt(f_{A}^{2}) = \left[\begin{array}{cc} 1 & 2 \\ 0 & 1 \end{array}\right]$$
and
$$\mt(f) =  \left[\begin{array}{cc} 1 & 0 \\ 3 & 1 \end{array}\right] \left[\begin{array}{cc} 1 & 2 \\ 0 & 1 \end{array}\right]= \left[\begin{array}{cc} 1 & 2 \\ 3 & 7 \end{array}\right]$$
One can check and immediatly understand the signification of theorem on figure \ref{figure34}. Note that this figure appears as a geometric application of the euclidean algorithm. 

\item The fixed points of a positively turbulent homeorphism $f\in \mathcal{T}^{+}$ will be found using Steve Smale's method and associated to green branching points of the left (green) branch and red branching points of the right (red) branch. One will even be able to compute the topological type of those fixed points using the structure of this tree.

\item One can either construct the tree from the inside, in a natural fashion, or add symetric copies of hole red or green branches on the opposite side of the tree and connect them at the bottom. It corresponds to the composition on the left (the outside), the composition on the right being traduced be the natural growth of the tree by its ends.

\item The structure of the topological snail obeys to identical rules relative to the composition of homeomorphisms.

\setlength{\unitlength}{0.6mm}
\begin{figure}
\begin{picture}(100,300)(-90,-20)

\color{black}
\put(-30,290){\line(1,0){130}}

\color{red}
\put(-25,290){\vector(0,1){8}}
\put(-25,290){\circle*{1}}
\put(-5,290){\vector(0,-1){8}}
\put(-5,290){\circle*{1}}
\put(5,290){\vector(0,1){8}}
\put(5,290){\circle*{1}}
\put(25,290){\vector(0,-1){8}}
\put(25,290){\circle*{1}}
\put(35,290){\vector(0,1){8}}
\put(35,290){\circle*{1}}
\put(45,290){\vector(0,1){8}}
\put(45,290){\circle*{1}}
\put(65,290){\vector(0,-1){8}}
\put(65,290){\circle*{1}}
\put(75,290){\vector(0,1){8}}
\put(75,290){\circle*{1}}
\put(95,290){\vector(0,-1){8}}
\put(95,290){\circle*{1}}

\color{mygreen}
\put(55,290){\vector(0,-1){8}}
\put(55,290){\circle*{1}}
\put(85,290){\vector(0,-1){8}}
\put(85,290){\circle*{1}}
\put(-15,290){\vector(0,1){8}}

\put(15,290){\vector(0,1){8}}
\put(15,290){\circle*{1}}
\put(-15,290){\circle*{1}}

\color{black}

\put(-30,0){\line(1,0){130}}
\put(-30,0){\line(0,1){280}}
\put(100,0){\line(0,1){280}}
\put(-30,0){\line(1,0){130}}
\put(0,0){\line(0,1){280}}
\put(0,0){\circle*{2}}

\put(70,100){\line(0,1){180}}
\put(30,170){\line(0,1){110}}
\put(40,210){\line(0,1){70}}

\put(-20,240){\line(0,1){40}}
\put(10,240){\line(0,1){40}}
\put(60,240){\line(0,1){40}}
\put(90,240){\line(0,1){40}}

\put(-10,260){\line(0,1){20}}
\put(20,260){\line(0,1){20}}
\put(50,260){\line(0,1){20}}
\put(80,260){\line(0,1){20}}

\put(-30,280){\line(1,0){130}}

\color{mygreen}

\put(0,0){\qbezier[1600](-3,0)(-20,0)(-20,240)}

\put(70,100){\qbezier[1200](0,0)(20,0)(20,140)}
\put(30,170){\qbezier[900](0,0)(-20,0)(-20,70)}
\put(40,210){\qbezier[900](0,0)(20,0)(20,30)}
\put(-20,240){\qbezier[300](0,0)(10,0)(10,20)}
\put(10,240){\qbezier[300](0,0)(10,0)(10,20)}
\put(60,240){\qbezier[300](0,0)(-10,0)(-10,20)}

\put(90,240){\qbezier[300](0,0)(-10,0)(-10,20)}
\put(-10,260){\qbezier[300](0,0)(-5,0)(-5,20)}
\put(20,260){\qbezier[300](0,0)(-5,0)(-5,20)}
\put(50,260){\qbezier[300](0,0)(5,0)(5,20)}
\put(80,260){\qbezier[300](0,0)(5,0)(5,20)}

\color{red}
\put(10,240){\qbezier[300](0,0)(-5,0)(-5,40)}
\put(-20,240){\qbezier[300](0,0)(-5,0)(-5,40)}
\put(90,240){\qbezier[300](0,0)(5,0)(5,40)}
\put(60,240){\qbezier[300](0,0)(5,0)(5,40)}

\color{mygreen}
\put(-9.6,11){\vector(-1,3){0}}
\put(79.6,110.8){\vector(-1,-3){0}}
\put(-18.1,111){\vector(0,1){0}}
\put(-19.7,200){\vector(0,1){0}}
\put(89.75,200){\vector(0,-1){0}}
\put(12.3,200){\vector(-1,4){0}}

\put(59.2,230){\vector(-1,-4){0}}
\put(-10.1,254){\vector(0,-1){0}}
\put(19.9,254){\vector(0,-1){0}}
\put(50.1,256){\vector(0,1){0}}
\put(80.1,256){\vector(0,1){0}}

\put(80.1,256){\vector(0,1){0}}

\put(54.9,275){\vector(0,-1){0}}
\put(84.9,275){\vector(0,-1){0}}

\put(15,276){\vector(0,1){0}}
\put(-15,276){\vector(0,1){0}}

\color{red}
\put(45,276){\vector(0,1){0}}
\put(75,276){\vector(0,1){0}}

\put(-24.9,275){\vector(0,1){0}}
\put(-5,275){\vector(0,-1){0}}
\put(5,275){\vector(0,1){0}}
\put(25,275){\vector(0,-1){0}}
\put(65,275){\vector(0,-1){0}}
\put(95.1,275){\vector(0,-1){0}}

\put(94.3,254){\vector(0,-1){0}}
\put(64.3,254){\vector(0,-1){0}}
\put(5.8,255){\vector(0,1){0}}
\put(-24.2,255){\vector(0,1){0}}

\put(36.1,232){\vector(0,1){0}}

\put(39.9,200){\vector(0,-1){0}}

\put(30.2,160){\vector(0,1){0}}
\put(44.7,111){\vector(-1,1){0}}

\put(69.9,90){\vector(0,-1){0}}
\put(40,10.75){\vector(-4,-3){0}}

\put(3,0){\circle*{3}}

\put(0,0){\qbezier[1200](3,0)(70,0)(70,100)}
\put(70,100){\qbezier[600](0,0)(-40,0)(-40,70)}
\put(30,170){\qbezier[600](0,0)(10,0)(10,40)}
\put(40,210){\qbezier[900](0,0)(-5,0)(-5,70)}

\put(-10,260){\qbezier[300](0,0)(5,0)(5,20)}
\put(20,260){\qbezier[300](0,0)(5,0)(5,20)}
\put(50,260){\qbezier[300](0,0)(-5,0)(-5,20)}
\put(80,260){\qbezier[300](0,0)(-5,0)(-5,20)}

\color{mygreen}
\put(0,100){\line(1,0){100}}
\put(0,170){\line(1,0){70}}
\put(30,210){\line(1,0){40}}
\put(70,100){\circle*{3}}
\put(30,170){\circle*{3}}
\put(40,210){\circle*{3}}
\put(-3,0){\circle*{3}}

\color{red}

\put(-30,240){\line(1,0){60}}
\put(40,240){\line(1,0){60}}

\put(-20,260){\line(1,0){20}}
\put(10,260){\line(1,0){20}}
\put(40,260){\line(1,0){20}}
\put(70,260){\line(1,0){20}}

\put(-20,240){\circle*{3}}
\put(10,240){\circle*{3}}
\put(60,240){\circle*{3}}
\put(90,240){\circle*{3}}

\put(-10,260){\circle*{3}}
\put(20,260){\circle*{3}}
\put(50,260){\circle*{3}}
\put(80,260){\circle*{3}}


\end{picture}
\caption{Inner universal two-coloured tree $\tn(3;10)$ associated to $\sn(3;10)$ (see figure \ref{figure33}).\label{figure34}}

\end{figure}
\setlength{\unitlength}{0.7mm}

\end{remarks}

\section{The closed topological snail and the zip relations\label{canonicalsection}}

As I have shown in the preceeding section, the consideration of the topological snail $\sn(n;p)$ associated to a positively turbulent homeomorphims $f\in \mathcal{T}^{+}$ leads  to define a turbulence matrix $\mt(f)$ (theorem \ref{decomposition} page \pageref{decomposition}) and a morphism of free monoid between the set of isotopy classes of positively turbulent homeomorphism relative to a given line $\Delta^{+}$ and matrices in $\mathrm{SL}_{2}(\Z)$ with positive coefficients. But this method can also be extended to any $X$-presering homeomorphism of the plane $f$, to define an isomorphism between the hole mapping-class group of $\R^{2} \setminus X$ and  $\mathrm{PSL}_{2}(\Z)$. This is the object of the present section.

At first sight, the notion of a positively turbulent homeomorphism is not intrisic since it depens on the choice of the oriented line $\Delta^{+}$. For example, if one changes the orientation of $\Delta^{+}$, then the positively turbulent homeomorphisms change. But of course the new turbulent homeomorphisms have identical properties and nothing distinguishes them from the preceeding ones since the attribution of an orientation to $\Delta^{+}$ is purely conventionnal. The notation $\mathcal{T}^{+}$ is therefore convenient ~: let us call those homeomorphisms positively turbulent, as opposed to the negatively turbulent $\mathcal{T}^{-}$ that are turbulent relative to $X$ and $\Delta^{-}$. Let us also define the turbulent homeomorphism as those that are conjugated in $\hom$ to some positively turbulent homeomorphism.

Let us supppose that $X = \lbrace -1; 0 ; 1 \rbrace$ and $\Delta^{+}$ is the real line with positive orientation, and consider the symetry $\varsigma(x;y) = (-x;-y)$. It preserves the orientation or $\R^{2}$ and the line $\Delta^{+}$, but its restriction to $\Delta^{+}$ returns its orientation.  And since the inverse of a top displacement to the left is a bottom displacement to the left, it is also a top displacement to the right according to the reversed orientation. A negative positively turbulent homeomorphism is therefore simply the inverse of the positive  positively turbulent one, and its isotopy classe is $\varsigma$-conjugated to the isotopy class of its inverse. More precisely, if one denotes by $\varphi_{A}$ the isotopy class of a homeomorphism of the type $A$, and $f_{B}$  the isotopy class of a homeomorphism of the type $B$, one has 
$$\varsigma\, f_{A} \, \varsigma \equiv f_{B}^{-1}~~\iff \varsigma\, f_{A}=f_{B}^{-1} \varsigma  ~~~~\mathrm{and}~~\varsigma\, f_{B} \, \varsigma \equiv f_{A}^{-1} $$

\setlength{\unitlength}{0.7mm}
\begin{figure}
\begin{picture}(100,260)(-65,-150)


\color{mygreen}
\put(-30,100){\line(1,0){20}}
\put(110,100){\line(1,0){20}}

\color{red}
\put(-10,100){\line(1,0){20}}
\put(90,100){\line(1,0){20}}

\color{black}
\put(0,100){\vector(1,0){0}}
\put(-19,100){\vector(1,0){0}}

\put(10,100){\circle*{1.5}}
\put(-10,100){\circle*{1.5}}
\put(-30,100){\circle*{1.5}}

\put(-10,100){\qbezier[25](-20,-4)(-20,-8.28)(-14.14,-14.14)}
\put(-10,100){\qbezier[25](-14.14,-14.14)(-8.28,-20)(0,-20)}
\put(-10,100){\qbezier[25](0,-20)(8.28,-20)(14.14,-14.14)}
\put(-10,100){\qbezier[25](14.14,-14.14)(20,-8.28)(19.9,-4)}

\put(-10,100){\qbezier[25](20,4)(20,8.28)(14.14,14.14)}
\put(-10,100){\qbezier[25](14.14,14.14)(8.28,20)(0,20)}
\put(-10,100){\qbezier[25](0,20)(-8.28,20)(-14.14,14.14)}
\put(-10,100){\qbezier[25](-14.14,14.14)(-20,8.28)(-20,4)}

\put(-30,104){\vector(-1,-4){0}}
\put(10,96){\vector(1,4){0}}

\put(35,100){\vector(1,0){30}}
\put(50,102){\makebox(0,0)[bc]{$\varsigma$}}
\put(-10,122){\makebox(0,0)[bc]{$\varsigma$}}

\put(120,100){\vector(-1,0){0}}
\put(101,100){\vector(-1,0){0}}
\put(90,100){\circle*{1.5}}
\put(110,100){\circle*{1.5}}
\put(130,100){\circle*{1.5}}


\put(-10,122){\makebox(0,0)[bc]{$\varsigma$}}
\color{mygreen}
\put(-30,50){\line(1,0){20}}

\put(110,50){\qbezier[100](20,0)(20,-8.28)(14.14,-14.14)}
\put(110,50){\qbezier[100](14.14,-14.14)(8.28,-20)(0,-20)}
\put(110,50){\qbezier[100](0,-20)(-8.28,-20)(-14.14,-14.14)}
\put(110,50){\qbezier[100](-14.14,-14.14)(-20,-8.28)(-20,0)}

\color{red}
\put(-10,50){\line(1,0){20}}
\put(110,50){\line(1,0){20}}

\color{black}
\put(-0,50){\vector(1,0){0}}
\put(-19,50){\vector(1,0){0}}

\put(10,50){\circle*{1.5}}
\put(-10,50){\circle*{1.5}}
\put(-30,50){\circle*{1.5}}

\put(119,50){\vector(-1,0){0}}
\put(110,30){\vector(1,0){0}}

\put(90,50){\circle*{1.5}}
\put(110,50){\circle*{1.5}}
\put(130,50){\circle*{1.5}}

\put(0,50){\qbezier[10](-10,-2)(-10,-4.14)(-7.07,-7.07)}
\put(0,50){\qbezier[10](-7.07,-7.07),(-4.14,-10),(0,-10)}
\put(0,50){\qbezier[10](0,-10),(4.14,-10)(7.07,-7.07)}
\put(0,50){\qbezier[10](7.07,-7.07)(10,-4.14)(10,-2)}
\put(-10,52){\vector(-1,-4){0}}

\put(0,50){\qbezier[10](-10,2)(-10,4.14)(-7.07,7.07)}
\put(0,50){\qbezier[10](-7.07,7.07),(-4.14,10),(0,10)}
\put(0,50){\qbezier[10](0,10),(4.14,10)(7.07,7.07)}
\put(0,50){\qbezier[10](7.07,7.07)(10,4.14)(10,2)}
\put(10,48){\vector(1,4){0}}

\put(100,50){\qbezier[10](-10,-2)(-10,-4.14)(-7.07,-7.07)}
\put(100,50){\qbezier[10](-7.07,-7.07),(-4.14,-10),(0,-10)}
\put(100,50){\qbezier[10](0,-10),(4.14,-10)(7.07,-7.07)}
\put(100,50){\qbezier[10](7.07,-7.07)(10,-4.14)(10,-2)}
\put(90,52){\vector(-1,-4){0}}

\put(100,50){\qbezier[10](-10,2)(-10,4.14)(-7.07,7.07)}
\put(100,50){\qbezier[10](-7.07,7.07),(-4.14,10),(0,10)}
\put(100,50){\qbezier[10](0,10),(4.14,10)(7.07,7.07)}
\put(100,50){\qbezier[10](7.07,7.07)(10,4.14)(10,2)}
\put(110,48){\vector(1,4){0}}

\put(0,62){\makebox(0,0)[bc]{$f_{B}^{-1}$}}

\put(100,62){\makebox(0,0)[bc]{$f_{A}$}}

\put(110,20){\vector(-1,0){0}}

\put(50,52){\makebox(0,0)[bc]{$f_{B}^{-1}$}}

\put(35,50){\vector(1,0){30}}

\color{mygreen}
\put(-10,0){\line(1,0){20}}
\put(110,0){\line(1,0){20}}

\color{red}
\put(-30,0){\line(1,0){20}}

\put(110,0){\qbezier[100](20,0)(20,8.28)(14.14,14.14)}
\put(110,0){\qbezier[100](14.14,14.14)(8.28,20)(0,20)}
\put(110,0){\qbezier[100](0,20)(-8.28,20)(-14.14,14.14)}
\put(110,0){\qbezier[100](-14.14,14.14)(-20,8.28)(-20,0)}

\color{black}
\put(0,0){\vector(-1,0){0}}
\put(-19,0){\vector(-1,0){0}}

\put(120,-12){\makebox(0,0)[tc]{$f_{B}^{-1}$}}

\put(90,0){\circle*{1.5}}
\put(110,0){\circle*{1.5}}
\put(130,0){\circle*{1.5}}

\put(-30,0){\circle*{1.5}}
\put(-10,0){\circle*{1.5}}
\put(10,0){\circle*{1.5}}

\put(120,0){\vector(1,0){0}}

\put(120,0){\qbezier[10](-10,-2)(-10,-4.14)(-7.07,-7.07)}
\put(120,0){\qbezier[10](-7.07,-7.07),(-4.14,-10),(0,-10)}
\put(120,0){\qbezier[10](0,-10),(4.14,-10)(7.07,-7.07)}
\put(120,0){\qbezier[10](7.07,-7.07)(10,-4.14)(10,-2)}
\put(110,2){\vector(-1,-4){0}}

\put(120,0){\qbezier[10](-10,2)(-10,4.14)(-7.07,7.07)}
\put(120,0){\qbezier[10](-7.07,7.07),(-4.14,10),(0,10)}
\put(120,0){\qbezier[10](0,10),(4.14,10)(7.07,7.07)}
\put(120,0){\qbezier[10](7.07,7.07)(10,4.14)(10,2)}
\put(130,-2){\vector(1,4){0}}

\put(50,2){\makebox(0,0)[bc]{$f_{B}^{-1}$}}
\put(65,0){\vector(-1,0){30}}

\put(140,25){\qbezier[100](0,20)(20,0)(0,-20)}
\put(153,25){\makebox(0,0)[cl]{$f_{A}$}}
\put(140,5){\vector(-1,-1){0}}

\put(-40,25){\qbezier[100](0,20)(-20,0)(0,-20)}
\put(-53,25){\makebox(0,0)[cr]{$\varsigma$}}
\put(-40,5){\vector(1,-1){0}}

\color{red}
\put(-10,-50){\line(1,0){20}}
\put(110,-50){\qbezier[100](20,0)(20,-8.28)(14.14,-14.14)}
\put(110,-50){\qbezier[100](14.14,-14.14)(8.28,-20)(0,-20)}
\put(110,-50){\qbezier[100](0,-20)(-8.28,-20)(-14.14,-14.14)}
\put(110,-50){\qbezier[100](-14.14,-14.14)(-20,-8.28)(-20,0)}

\color{mygreen}
\put(-30,-50){\line(1,0){20}}
\put(90,-50){\line(1,0){20}}

\color{black}

\put(50,-48){\makebox(0,0)[bc]{$f_{A}^{-1}$}}
\put(35,-50){\vector(1,0){30}}

\put(0,-50){\vector(1,0){0}}
\put(-19,-50){\vector(1,0){0}}

\put(99,-50){\vector(-1,0){0}}
\put(10,-50){\circle*{1.5}}
\put(-10,-50){\circle*{1.5}}
\put(-30,-50){\circle*{1.5}}

\put(90,-50){\circle*{1.5}}
\put(110,-50){\circle*{1.5}}
\put(130,-50){\circle*{1.5}}

\put(-20,-50){\qbezier[10](-10,-2)(-10,-4.14)(-7.07,-7.07)}
\put(-20,-50){\qbezier[10](-7.07,-7.07),(-4.14,-10),(0,-10)}
\put(-20,-50){\qbezier[10](0,-10),(4.14,-10)(7.07,-7.07)}
\put(-20,-50){\qbezier[10](7.07,-7.07)(10,-4.14)(10,-2)}
\put(-30,-52){\vector(-1,4){0}}

\put(-20,-50){\qbezier[10](-10,2)(-10,4.14)(-7.07,7.07)}
\put(-20,-50){\qbezier[10](-7.07,7.07),(-4.14,10),(0,10)}
\put(-20,-50){\qbezier[10](0,10),(4.14,10)(7.07,7.07)}
\put(-20,-50){\qbezier[10](7.07,7.07)(10,4.14)(10,2)}
\put(-10,-48){\vector(1,-4){0}}

\put(120,-50){\qbezier[10](-10,-2)(-10,-4.14)(-7.07,-7.07)}
\put(120,-50){\qbezier[10](-7.07,-7.07),(-4.14,-10),(0,-10)}
\put(120,-50){\qbezier[10](0,-10),(4.14,-10)(7.07,-7.07)}
\put(120,-50){\qbezier[10](7.07,-7.07)(10,-4.14)(10,-2)}
\put(130,-48){\vector(1,-4){0}}

\put(120,-50){\qbezier[10](-10,2)(-10,4.14)(-7.07,7.07)}
\put(120,-50){\qbezier[10](-7.07,7.07),(-4.14,10),(0,10)}
\put(120,-50){\qbezier[10](0,10),(4.14,10)(7.07,7.07)}
\put(120,-50){\qbezier[10](7.07,7.07)(10,4.14)(10,2)}
\put(110,-52){\vector(-1,4){0}}

\put(120,-38){\makebox(0,0)[bc]{$f_{B}$}}
\put(-20,-38){\makebox(0,0)[bc]{$f_{A}^{-1}$}}

\put(140,-75){\qbezier[100](0,20)(20,0)(0,-20)}
\put(153,-75){\makebox(0,0)[cl]{$f_{B}$}}
\put(140,-95){\vector(-1,-1){0}}

\put(-40,-75){\qbezier[100](0,20)(-20,0)(0,-20)}
\put(-53,-75){\makebox(0,0)[cr]{$\varsigma$}}
\put(-40,-95){\vector(1,-1){0}}


\color{red}

\put(-30,-100){\line(1,0){20}}
\put(90,-100){\line(1,0){20}}

\color{mygreen}
\put(-10,-100){\line(1,0){20}}

\put(110,-100){\qbezier[100](20,0)(20,8.28)(14.14,14.14)}
\put(110,-100){\qbezier[100](14.14,14.14)(8.28,20)(0,20)}
\put(110,-100){\qbezier[100](0,20)(-8.28,20)(-14.14,14.14)}
\put(110,-100){\qbezier[100](-14.14,14.14)(-20,8.28)(-20,0)}

\color{black}

\put(100,-100){\vector(1,0){0}}

\put(90,-100){\circle*{1.5}}
\put(110,-100){\circle*{1.5}}
\put(130,-100){\circle*{1.5}}

\put(-30,-100){\circle*{1.5}}
\put(-10,-100){\circle*{1.5}}
\put(10,-100){\circle*{1.5}}

\put(50,-98){\makebox(0,0)[bc]{$f_{A}^{-1}$}}
\put(65,-100){\vector(-1,0){30}}

\put(100,-100){\qbezier[10](-10,-2)(-10,-4.14)(-7.07,-7.07)}
\put(100,-100){\qbezier[10](-7.07,-7.07),(-4.14,-10),(0,-10)}
\put(100,-100){\qbezier[10](0,-10),(4.14,-10)(7.07,-7.07)}
\put(100,-100){\qbezier[10](7.07,-7.07)(10,-4.14)(10,-2)}
\put(90,-98){\vector(-1,4){0}}

\put(100,-100){\qbezier[10](-10,2)(-10,4.14)(-7.07,7.07)}
\put(100,-100){\qbezier[10](-7.07,7.07),(-4.14,10),(0,10)}
\put(100,-100){\qbezier[10](0,10),(4.14,10)(7.07,7.07)}
\put(100,-100){\qbezier[10](7.07,7.07)(10,4.14)(10,2)}
\put(110,-102){\vector(1,-4){0}}

\put(100,-112){\makebox(0,0)[tc]{$f_{A}^{-1}$}}

\put(0,-100){\vector(-1,0){0}}
\put(-19,-100){\vector(-1,0){0}}

\put(110,-80){\vector(-1,0){0}}

\end{picture}
\caption{Topological illustration of the relation $\varsigma = f_{B}^{-1} \circ f_{A} \circ f_{B}^{-1} = f_{A}^{-1} \circ f_{B} \circ f_{A}^{-1}$. \label{figure35}}

\end{figure}
\setlength{\unitlength}{0.7mm}

Of course the homeomorphism $\varsigma$ is not a turbulent homeomorphism since for example $\varsigma^{2}=\id$ and no non trivial turbulent homeomorphism can be periodic. But one can precise its relation to turbulent homeomorphisms. One can for example look at figure \ref{figure35} or compute particular isotopies to satisfy and prove general expressions for $\varsigma$~:

\begin{eqnarray*} \varsigma & = & f_{A}^{-1} \circ f_{B} \circ f_{A}^{-1} \\  
				& = & f_{B}^{-1} \circ f_{A} \circ f_{B}^{-1} \\
				& = & f_{A} \circ f_{B}^{-1} \circ f_{A} \\
				& = & f_{B} \circ f_{A}^{-1} \circ f_{B} \end{eqnarray*}

An important consequence is the equality 

\begin{eqnarray*} f_{A}^{-1} \circ f_{B} \circ f_{A}^{-1} = f_{B}^{-1} \circ f_{A} \circ f_{B}^{-1} 
& \iff & (f_{A} \circ f_{B}^{-1})^{3}= \id \\ 
& \iff & (f_{B}^{-1} \circ f_{A})^{3}= \id \end{eqnarray*}

It means that $\rho = f_{A} \circ f_{B}^{-1}$ is another non turbulent homeomorphism, and element of order $3$ satisfying the relation 
$$\varsigma \circ \rho \circ \varsigma = \left(f_{A}^{-1} \circ f_{B} \circ f_{A}^{-1}) \circ  \left(f_{A} \circ f_{B}^{-1}\right) \circ  \left(f_{B} \circ f_{A}^{-1} \circ f_{B} \right) = (f_{A} \circ f_{B}^{-1})^{2} = \rho^{2}  \right. $$

To extend on the hole set of isotopy classes of homeomorphims relative to $X$ the morphism 
$$\begin{array}{lccl} \Phi: & \mathcal{T}^{+} & \longrightarrow &\mathrm{SL}_{2}(\Z) \\ & f & \longmapsto & \mt(f) \end{array}$$  one has naturally to set 
\begin{eqnarray*} \mt(f_{A}^{-1}) = \left(\mt (f_{A}) \right)^{-1} = \left[\begin{array}{cc} 1 & 1 \\ 0 & 1 \end{array} \right]^{-1} & =& \left[\begin{array}{cc} 1 & -1 \\ 0 & 1 \end{array} \right] \\
\mt(f_{B}^{-1}) = \left(\mt (f_{B}) \right)^{-1} = \left[\begin{array}{cc} 1 & 0 \\ 1 & 1 \end{array} \right]^{-1} & =& \left[\begin{array}{cc} 1 & 0 \\ -1 & 1 \end{array} \right] \end{eqnarray*}

But this imply that 

\begin{eqnarray*} \mt(\varsigma)  & = &  A^{-1} \times B \times A^{-1} = \left[\begin{array}{cc} 1 & -1 \\ 0 & 1 \end{array} \right]\left[\begin{array}{cc} 1 & 0 \\ 1 & 1 \end{array} \right]\left[\begin{array}{cc} 1 & -1 \\ 0 & 1 \end{array} \right] = \left[\begin{array}{cc} 0 & -1 \\ 1 & 0 \end{array} \right]  \\
 & = &  B^{-1} \times A \times B^{-1} = \left[\begin{array}{cc} 1 & 0 \\ -1 & 1 \end{array} \right]\left[\begin{array}{cc} 1 & 1 \\ 0 & 1 \end{array} \right]\left[\begin{array}{cc} 1 & 0 \\ -1 & 1 \end{array} \right] = \left[\begin{array}{cc} 0 & 1 \\ -1 & 0 \end{array} \right]
\end{eqnarray*}

Similarily, since $\rho = f_{A}\circ f_{B}^{-1}$ and $\rho^{3} = \id$ one must have 

 \begin{eqnarray*}  \mt(\rho)  =  A\times B^{-1} = \left[\begin{array}{cc} 1 & 1 \\ 0 & 1 \end{array} \right]\left[\begin{array}{cc} 1 & 0 \\ -1 & 1 \end{array} \right] & = & \left[\begin{array}{cc} 0 & 1 \\ -1 & 1 \end{array} \right]\, , \\
  \mt(\rho^{3})  =  (A\times B^{-1})^{3} = \left[\begin{array}{cc} 0 & 1 \\ -1 & 1 \end{array} \right]^{3} = \left[\begin{array}{cc} -1 & 0 \\ 0 & -1 \end{array} \right] & =&  \left[\begin{array}{cc} ~1 & 0 \\ 0 & ~1 \end{array} \right] \, .
 \end{eqnarray*}
 
The matrice $\mt(f) \in \mathrm{SL}_{2}(\Z)$ with positive coefficients associated to a homeomorphism $f\in \mathcal{T}^{+}$, or equivalently to its topological snail, must in fact be considered as representative of its class modulo the multiplication by $-\id$ in $\mathrm{SL}_{2}(\Z)$. In other words, the map $\Phi$ must not be defined as a map into $\mathrm{SL}_{2}(\Z)$ but into $\mathrm{PSL}_{2}(\Z)$. When we look at a turbulence matrice $M_{T}(f)$, we must simultaneaously associate its opposite $-M_{T}(f)$. In a similar fashion, the homeomorphism $f \in \homplus$ has a natural conjugated one $f\circ \sigma \in \hom$.

Since $Z=\mt(\varsigma)$, for any $f\in \mathcal{T}^{+}$ one has therefore to set~:
\begin{eqnarray*} \Phi(\varsigma \circ f) =  Z \times \mt(f) = \left[\begin{array}{cc} 0 & 1 \\ -1 & 0 \end{array} \right] \left[\begin{array}{cc} g_{1} & r_{1} \\ g_{2} & r_{2} \end{array} \right] & = & \left[\begin{array}{cc} g_{2} & r_{2} \\ -g_{1} & -r_{1} \end{array} \right] \\
\Phi(f \circ \varsigma)  =  \mt(f) \times Z = \left[\begin{array}{cc} g_{1} & r_{1} \\ g_{2} & r_{2} \end{array} \right] \left[\begin{array}{cc} 0 & 1 \\ -1 & 0 \end{array} \right] & = & \left[\begin{array}{cc} -r_{1} & \,~g_{1}~ \\ -r_{2} & \,~g_{2}~ \end{array} \right] \\
 \Phi(\varsigma \circ f \circ \varsigma) = Z  \times \mt(f) \times Z=   \left[\begin{array}{cc} 0 & 1 \\ -1 & 0 \end{array} \right] \left[\begin{array}{cc} g_{1} & r_{1} \\ g_{2} & r_{2} \end{array} \right] \left[\begin{array}{cc} 0 & 1 \\ -1 & 0 \end{array} \right] & = & \left[\begin{array}{cc} r_{2} & -g_{2} \\ -r_{1} & g_{1} \end{array} \right] 
\end{eqnarray*}

Since $\Phi(\varsigma) = Z$ and $\sigma(\sn(n,p)) = \sigma(\sn^{-}(n;p)$, we are now conducted to define the general topological snail for $n, p \in \Z$ are relatively prime, or $n=\pm 1$ with $p=0$;  $p=\pm1$ with $n=0$, by
$$\begin{array}{rcccrcl} \sn(n;p)  & = & \sn(n;p) & \iff & np& \geqslant &0 \\
 \sn(n;p)& = &\sigma\left( \sn(n;p) \right) & \iff  & np &\leqslant &0 
 \end{array}$$

One has in general $\sn(-n;p) = \sigma\left(\sn(n;p)\right)$: the action of $\sigma$ on $\sn(n;p)$ reduces to the multiplication of 
$\left[\begin{array}{c} n \\ p \end{array} \right]$ by $Y$. One must therefore finally set 

$$\Phi(\sigma) = \left[\begin{array}{cc} -1 & 0 \\ 0 & 1 \end{array} \right] =Y
$$

On has then the following canonical form for a homeomorphism $f\in \hom$~:

\begin{theo} \label{canonicalform}
Let $\Delta^{+}\subset \C$ be the real line, $X = \lbrace -1; 0 ; 1 \rbrace \subset \Delta^{+}$,  $\sigma$ be the orthogonal symetry about $\Delta^{+}$ defined by $\sigma(z) = \overline{z}$, $\varsigma$ the central symetry about $0$ defined by $\varsigma(z) = -z$. Then for any homeomorphism $h$ of the complex plane $\C$ such that $h(X)=X$, there is a unique $f\in \mathcal{T}^{+}$, a positively turbulent homeomorphism $f$ relatively to $X$ and $\Delta^{+}$ such that one has either and exclusively one of the following canonical forms~: 
$$\begin{array}{rclcrclcrclcrcl}
h & = & f &~~ & h & = & f \circ \varsigma & ~~& h & = & \varsigma \circ f &~~ &h  & =& \varsigma \circ f \circ \varsigma \\
h & =&  f \circ \sigma & & h & =&  f \circ \varsigma \circ \sigma & & h & =&  \varsigma \circ f \circ \sigma & & h & =& \varsigma \circ f \circ \varsigma \circ \sigma
\end{array}$$
Furthermore, if one defines $\Phi$ as a group homomorphism according to the preceding equalities and using 
\begin{eqnarray*}
 \Phi(f) & = & \mt(f) =\left[\begin{array}{cc} g_{1} & r_{1} \\ g_{2} & r_{2} \end{array} \right] = M   \\
\Phi(\varsigma) & = &  \mt(\varsigma) = \left[\begin{array}{cc} 0 & 1 \\ -1 & 0 \end{array} \right]   =  Z \\
 \Phi(\sigma) & = & \mt(\varsigma)  =\left[\begin{array}{cc} -1 & 0 \\ 0 & 1 \end{array} \right]  = Y
\end{eqnarray*}
then one obtains an isomorphism between the mapping class group of $\R^{2}\setminus X$ and $\mathrm{PSL}_{2}(\Z)$. Similarily if one systematically identifies $(n;p)$ with $(-n;-p)$ and generaly defines and writes $\mt(f)= \Phi(f)$, then for any $(n;p) \in \Z^{2}$, one has~:
\begin{eqnarray*}
\sn(n';p') \equiv f\left(\sn(n;p)\right) & \iff & \left[\begin{array}{c} n' \\ p' \end{array}\right] \equiv \mt(f) \left[\begin{array}{c} n \\ p \end{array}\right]  
\end{eqnarray*} 
 \end{theo}

\begin{proof}
Let $X= \lbrace -1; 0 1 \rbrace$ and $h$ be an homeomorphism such that $h(X)=X$.

Let us first define the fondamental circle and write $\mathrm{C}$ the circle of center 0 and radius 1. It is a junction of two simple arcs doubly branched on each other,  
$$C = \Gamma_{g} \sqcup \Gamma_{r}\, , $$
where $\Gamma_{g} = \mathrm{C} \cap \H^{+}$ and $\Gamma_{r} = \mathrm{C} \cap \H^{-}$. Let us all $\Gamma_{g}$ the green fundamental curve and $\Gamma_{r}$ the red fundamental curve. 

The green fundamental curve $\Gamma_{g} =  \sn(1,1) = \sn(1;1)$ is a half-circle in the top half plane doubly branched on $\big\{ p_{1}; p_{3} \big\}$.  Its image by the symetry $\sigma$ is the red fundamental curve $\Gamma_{r} =  \sigma(\Gamma_{g}) = \sigma(\sn(1,1)) = \sn(1;-1)$. The image $h(C)$ of the fundamental curve splits in two disjoint curves $h(\Gamma_{g})$ and $h(\Gamma_{r})$, doubly branched on $\lbrace  h(p_{1}); h(p_{3}) \rbrace$. Their spinning skeletons $\Psi\left(h(\Gamma_{1})\right)$ and $\Psi\left(h(\Gamma_{2})\right)$  define disjoint topological snails that are doubly branched on $\big\lbrace  h(p_{1}); h(p_{3}) \big\rbrace$ and only depend on the mapping class of $h$ relative to $X$. Call those snails respectively the green and the red snail, and their junction the closed snail of $h$~:
$$\sn(h)\equiv \sn(\Gamma_{g}) \sqcup \sn(\Gamma_{r}) \equiv h(\sn(1;1)) \sqcup h(\sn(1;-1))$$

Suppose first that both the red and the green snail have their emerging arc on different sides of $\Delta^{+}$. Then those arcs must be both doubly branched on each others and on $\lbrace p_{1}; p_{3} \rbrace$, which corresponds to only four possible isotopy classes~: 
\begin{itemize}
\item The identity $\id$ and $\varsigma$, the symetry around $p_{2}=0$. They preserve the orientation. 
\item The symetry $\sigma$ around the horizontal axis $\Delta^{+}$ and the symetry $\sigma \circ \varsigma = \varsigma \circ \sigma$ around the vertical axis whose equation is $x=0$. They reverse the orientation. 
\end{itemize}
Furthermore, those classes are uniquely determined by their action on the orientation and on the red and the green arc. Any homeorphism that preserves both arcs $\sn(1;1)$ and $\sn(1;-1)$ oriented according to $\Delta^{+}$ is effectively homotopic to the identity. It is also very natural to consider the subgroup formed by the isotopy classes in $\hom$ that globaly preserve the fundamental circle $\sn(1;1) \sqcup \sn(1;-1)$. It contains of course the classes of $\id$, $\sigma$, $\varsigma$ and $\sigma\circ \varsigma$. But no other class. Another important sub-group is formed by the homeomorphisms $f$ that preserve $\sn(1;1)$ and possibly act non trivially on $\sn(1;-1)$. They are easy to identify simply analyzing the red snail from its ends $p_{1}$ and $p_{3}$. Suppose that $f$ is such an homeomorphism and let us call $\gamma_{1}$ the half circle in $f(\sn(1;-1))$ that has an end in $p_{1}$ and $\gamma_{2}$ the half  in $f(\sn(1;-1))$ that has an end in $p_{3}$. If one of them is emerging in $\H^{-}$, then necessarily both coincide with $\sn(1;-1)$ and $f\equiv \id$. Else they are both emerging in the top half-plane and necessarily on different sides of $\sn(1;1)$. It is then an exercise to check that $f(\sn(1;-1)) = \sn(2n+1;2n+3)$ or $f(\sn(1;-1)) = \sn(2n+3;2n+1)$ for some $n\geqslant 0$.

When the red and the green snail have their visible emerging arc on a common definite side of $\Delta^{+}$, which simply  defines the emerging side of the closed snail $\sn(h)$, there is a definite dominant snail, of a given color, surrounding the other one on this side of $\Delta^{+}$ and on the other one (see for example figure \ref{figure36}). Let us say that this colour is the dominant color of the snail.

For a positively turbulent  $f\in \mathcal{T}^{+}$ the rule (see figure \ref{figure30})
\begin{eqnarray*}
\sn(n';p') = f(\sn(n;p)) & \iff & \left[\begin{array}{c} n' \\ p' \end{array}\right] = \mt(f) \left[\begin{array}{c} n \\ p \end{array}\right]  
\end{eqnarray*} 
extends to $\sn(n;p)$ for $n\geqslant -1$ and $p\geqslant -1$, with an exactly similar proof (see theorem \ref{decomposition} and its proof). And this implies that the color of the closed snail remains the same from the first step and at each step of the recurrence: when the snail is in top position, its color is invariant by the action of positively turbulent homeomorphism. 

\input{figure36}

Let us suppose that $\sn(h)$ is green and emerging in the top half-plane $\H^{+}$. Then 
$$h(\sn(1;1))\equiv \sn(n;p)$$
 for some integers $n\geqslant 1$ and $p\geqslant 1$. According to the theorem \ref{decomposition}, let $f\in \mathcal{T}^{+}$ be such that $f(\sn(1;1)) \equiv \sn(n;p) $. Then $f^{-1} \circ h (\sn(1;1))\equiv \sn(1;1)$ and if $f^{-1} \circ h (\sn(-1;1))\nequiv \sn(-1;1)$, then $\sn(f^{-1} \circ h)$ would be in top position and red, but this is impossible for in this case $\sn(f\circ f^{-1} \circ h)= \sn(h)$ would also be red, since the color of a snail's homeomorphism in top position is invariant by the action of positively turbulent homeomorphisms. Therefore $f^{-1} \circ h \left(\sn(1;-1)\right) = \sn(1;-1)$ and $f^{-1} \circ h\equiv \id$ in $\R^{2} \setminus X$ or equivalenly $f\equiv h$ and $h\in \mathcal{T}^{+}$.  

This means that the positively turbulent homeomorphisms are the homeomorphisms that preserve the orientation and admit a closed snail in top position of the green color. But since $\varsigma$ and $\sigma$ permute $\Gamma_{g}$ and $\Gamma_{r}$, the snails  $\sn(h)$ and $\sn(h \circ \sigma)$  always have opposite colors, as have $\sn(h)$ and $\sn (h \circ \varsigma)$. And the snails $\sn(h)$ and $\sn(h \circ \sigma)$ have also opposite positions relative to $\Delta^{+}$, as have $\sn(h)$ and $\sn(\varsigma \circ h)$. Since $\sigma \circ \varsigma = \varsigma \circ \sigma$ (this map is a symetry around the vertical axis whose equation is $x=0$) and $\sigma$ reverses the orientation, one obtains the first part of the theorem and the definition of $\Phi(h)$. The color of the closed snail $\sn(h)$, red or green ; its position, top or bottom; and the orientation of $h$, positive or negative uniquely determine one of the eight possible forms for $h$~: $f$, $f\circ \varsigma$, $\varsigma \circ f$, $\varsigma \circ f \circ \varsigma$,  $f\circ \sigma$, $f\circ \varsigma\circ \sigma$, $\varsigma \circ f \circ \sigma$, $\varsigma \circ f \circ \varsigma\circ \sigma$, with a unique positively turbulent homeomorphism $f\in \mathcal{T}^{+}$. Let us call this map the positively turbulent part of $h$ relatively to $\Delta$. Setting 
$$\Phi(\varsigma) =  \mt(\varsigma) = \left[\begin{array}{cc} 0 & 1 \\ -1 & 0 \end{array} \right]   =  Z ~~\mathrm{and}~~\Phi(\sigma) = \mt(\sigma)  =\left[\begin{array}{cc} -1 & 0 \\ 0 & 1 \end{array} \right]  = Y\, $$ and extending $f\mapsto \mt(f)$ defined on $\mathcal{T}^{+}$ to a map that satisfy the law of group morphisms for all those formulas, one can therefore uniquely define a map
$$\begin{array}{cccl} \Phi~:& \mathrm{Hom}\left(\R^{2} \setminus X\right)& \longrightarrow & \mathrm{PSL}_{2}(\Z) \\ & f & \longmapsto &  \mt(f) \end{array}\, $$
 
It is defined using the closed topological snail and the proof that it is an isomorphism between $\mathrm{Hom}(\R^{2}\setminus X)$ and $\mathrm{PSL}_{2}(\Z)$ will on the contrary be based on their common algebraic properties.

Some very important nice and useful relations, the zip-relations, will retain our attention. One has first (see figure \ref{figure35})
$$ Z  =  A^{-1}BA^{-1}= B^{-1}AB^{-1}= BA^{-1}B = A B^{-1} A  \, , $$
and then
$$ AZA  = B\, ;~~~AZB= BZA = Z\,;~~~~~BZB= A\,.$$
and similarily 
$$ AYA = Y\, ;~~~BYB=Y ;~~~ZYZ=Y ;~~~YZY=Z\, .$$
Let us define $\tau= \sigma \circ \varsigma= \varsigma \circ \sigma$ and  $T=YZ=ZY$. Then of course $\Phi(\tau)= \mt(\tau) = T$. And one can write the zip relations under their canonical form~:
$$\left\lbrace 
\begin{array}{rcl} 
AT& =  &TB \\
AY & =&  YA^{-1} \\
AZ & = & ZB^{-1} \end{array}\right.~~~~~\left\lbrace 
\begin{array}{rcl} 
TA& =  &BT \\
YA & =&  A^{-1}Y \\
ZA & = & B^{-1}Z \end{array}\right.~~~~~\left\lbrace 
\begin{array}{rcl} 
BT& = & TA \\
BY & = & YB^{-1} \\
BZ & = & ZA^{-1} \end{array}\right. ~~~~~\left\lbrace 
\begin{array}{rcl} 
TB& =  &AT \\
YB & =&  B^{-1}Y \\
ZB & = & ZA^{-1} \end{array}\right.~~~~~
 $$

They are very useful for several reasons. In the proof of the theorem, they allow to define an algorithm to compute the canonical form of a product from  the canonical forms of its factors by zipping thrue any expression the special transformations $Z$, $Y$ or $T$ according to the needs, and by introducing them if necessary to simplify a formula. One can move a letter $T$, $Y$ or $Z$ inside an expression in the relative powers of $A$ and $B$, like for exemple~:
\begin{eqnarray*}
Y\left(AB^{-3}AB^{5}\right) & = & \left(A^{-1}B^{3}A^{-1}\right)Y\left(B^{5}\right) \\
Z\left(AB^{-3}AB^{5}\right) & = & \left(B^{-1}A^{3}B^{-1}A^{-5}\right)Z \\
T\left(AB^{-3}AB^{5}\right) & = & \left(BA^{-3}\right)T\left(AB^{5}\right) 
\end{eqnarray*} 

Let us suppose first that $h_{2}$ and $h_{1}$ are both orientation-preserving homeomorphisms and that $h_{2} = f_{2}$ is positively turbulent relative to $\Delta^{+}$. If $h_{1} = f_{1}$ is also positively turbulent, or if $h_{2} = f_{2} \circ \varsigma$, then the positively turbulent part of  $h_{2} \circ h_{1}$ is immediatly $f_{2} \circ f_{1}$ and one can check the morphism rule directly writing for example~:
$$\Phi(h_{2} \circ h_{1})= \Phi(f_{2} \circ f_{1} \circ \varsigma)= \Phi(f_{2} \circ f_{1}) \times \Phi(\varsigma)= \Phi(f_{2}) \times \Phi(f_{1}) \times \Phi(\varsigma)= \Phi(f_{2}) \times \Phi(f_{1}\circ\varsigma) = \Phi(h_{2}) \times \Phi(h_{1})$$

But if on the contrary the canonical form of $h_{1}$ is $\varsigma \circ f_{1}$ or $\varsigma \circ f_{1} \circ \varsigma$, then we will use the zip-relations to reduce the product to a canonical form. Consider for exemple the isotopy classes $h_{2} =  f_{B}\circ f_{A}^{2}$ and $h_{1} = \varsigma \circ f_{B} \circ f_{A}^{3} \circ \varsigma$. One first work in $\mathrm{PSL}_{2}(\Z)$, using the zip relations to write~: 
\begin{eqnarray*}
\mt(h_{2}) \times \mt(h_{1}) & = & BA^{2} Z B A^{3} Z \\
& = &BA^{2}A^{-1}B^{-3} \\
& = &  BAB^{-3} \\
& = & B^{2}(B^{-1}AB^{-1})B^{-2} \\
& = & B^{2} Z B^{-2} \\
& = & B^{2} A^{2}Z \\
\end{eqnarray*}

One therefore considers the unique antecedent of $B^{2}A^{2}Z$ in canonical form, $f_{B}^{2} \circ f_{A}^{2} \circ \varsigma$. The zip-relations are valuable for the homeomorphisms $f_{A}$, $f_{B}$, $\tau$, $\sigma$, $\varsigma$, and they can be used reversely to show that 
$$  f_{B}^{2} \circ f_{A}^{2} \circ \varsigma = f_{B} \circ f_{A}^{2} \circ \varsigma \circ f_{B} \circ f_{A}^{3} \circ \varsigma = h_{2} \circ h_{1} $$
And therefore one obtains that $\mt(h_{2}) \times \mt(h_{1})= \mt(f_{B}^{2} \circ f_{A}^{2} \circ \varsigma)= \mt(h_{2} \circ h_{1})$.

The idea of the preceeding computation is indeed very simple. 

The first step is to obtain a canonical expression for $\Phi(h_{2}) \times \Phi(h_{1})= \mt(h_{2}) \times \mt(h_{1})$. One considers two homeomorphisms $h_{2}$ in and $h_{1}$ in canonical form, and their expressions using $\Phi$, uniquely decomposed as positive powers of $A$ and $B$, with the corresponding letters $Y,Z,T$. One uses the zip-relations to put all the letters $Y,Z,T$ in the product at the right end of the expression. In general, one obtains a word in all positive or all negative powers of $A$ and $B$, followed by a word  of the same nature. If the two words are at the positive powers of $A$ and $B$ the work is done and the form is canonical. Else, there is a unique place where one can read one of the expressions $A^{-1}B$, $AB^{-1}$, $B^{-1}A$, $BA^{-1}$ and introduce the letter $Z$ to replace those combinations respectively by $ZA$, $BZ$, $ZB$ and $AZ$. Zipping this additional $Z$ to the right end of the expression, one obtains a single word in the positive or the negative powers of $A$ and $B$. If the powers are positive then the expression is canonical. Else one can introduce $Z^{2} = \id$ at the left end of the total expression and zip one of the letter $Z$ to the right end : the result is therefore a canonical expression.  

Once this first step is donne, one can consider the unique homeomorphism $h_{3}$ in canonical form whose image by $\Phi$ has the obtained canonical expression. On this expression, all the preceeding operations have a traduction in the mapping-class group that can be applied retrospectively to come back to the product of the canonical expressions of $h_{2}$ and $h_{1}$. Once has therefore that 
$$\Phi(h_{2} \circ h_{1}) = \Phi(h_{2}) \times \Phi(h_{1})\, . $$
And since all this can be done for any homeomorphism $h_{2}$ and $h_{1}$, the fundamental morphism rule for $\Phi$ is proved.

Furthermore, the turbulence matrix $\mt(f)$ of any homeomorphism $f\in \mathrm{Hom}\left(R^{2} \setminus X\right)$ can be uniquely determined from its topological properties. Let $f(\sn(1;0)$ and $f(\sn(0;1)$ be the respective red and green snail of $f$. Suppose first that $f\nequiv \id$ with a closed snail $\sn(f)$ positively emerging (in the top half-plane). Let $a,b,c,d$ be the positive integers such that 
$$ \left\lbrace \begin{array}{c}f(\sn(1;0)) \equiv \sn(a;b) \equiv \sn(-a;-b) \\   f(\sn(0;1)) \equiv \sn(c;d) \equiv \sn(-c;-d)\end{array}\right. $$
The first row of $\mt(f)$ must be $\left[\begin{array}{c} a \\ b \end{array}\right]$ or $\left[\begin{array}{c} -a \\ -b \end{array}\right]$ and its second raw $\left[\begin{array}{c} c \\ d \end{array}\right]$ or $\left[\begin{array}{c} -c \\ -d \end{array}\right]$. 

One always has 
$$\left| \begin{array}{cc} a & c \\ b & d \end{array} \right| = ad - bc = \pm1 $$ and the determinant $\mathrm{det}(\mt(f))$ must be equal to 1 if $f$ preserves the orientation or else $-1$. And one can always take $\left[\begin{array}{c} a \\ b \end{array}\right]$ as a first row for $\mt(f)$ and choose $\left[\begin{array}{c} c \\ d \end{array}\right]$ or $\left[\begin{array}{c} -c \\ -d \end{array}\right]$ in order to obtain the determinant $+1$ if $f$ preserves the orientation or else $-1$. This determines a unique matrix $\mt(f)$ in $\mathrm{PSL}_{2}(\Z)$.
Similarily, if $f\nequiv \id$ and $f$ is negatively emerging, let $a,b,c,d$ be the positive integers such that 
$$ \left\lbrace \begin{array}{c}f(\sn(1;0)) \equiv \sn(a;-b) \equiv \sn(-a;b) \\   f(\sn(0;1)) \equiv \sn(c;-d) \equiv \sn(-c;d)\end{array}\right. $$
Then chosing $\left[\begin{array}{c} a \\ -b \end{array}\right]$ as a first row for $\mt(f)$ and $\left[\begin{array}{c} c \\ -d \end{array}\right]$ or $\left[\begin{array}{c} -c \\ d \end{array}\right]$ as a second row in order to obtain the determinant $+1$ if $f$ preserves the orientation or else $-1$ is the unique manner to define $\mt(f)$ in $\mathrm{PSL}_{2}(f)$. This shows that $\Phi$ is an isomorphism between the mapping-class group of $\R^{2} \setminus X$ and $\mathrm{PSL}_{2}(\Z)$.

The equivalence  
\begin{eqnarray*}
\sn(n';p') \equiv f\left(\sn(n;p)\right) & \iff & \left[\begin{array}{c} n' \\ p' \end{array}\right] \equiv \mt(f) \left[\begin{array}{c} n \\ p \end{array}\right]  
\end{eqnarray*} 
 for any $(n;p) \in \Z^{2}$ identified with $(-n;-p)$ is a simple consequence of the fundamental canonical decomposition and the fact that $h=f_{A}$, $h=f_{B}$, $h=\varsigma$ and $h=\sigma$ act on the parameter of $\sn(n;p)$ as their respective associated turbulence matrix on $\left[\begin{array}{c} n \\ p \end{array}\right]$ identified with $\left[\begin{array}{c} -n \\ -p \end{array}\right]$.

\end{proof}

\section{The trace of an homeomorphism.}

The canonical form (theorem \ref{canonicalform}) has of course a general meaning, analogous to the decomposition of a natural number in the product of its prime factors~: any homeomorphism $f$ that preserve $X= \Big\{ p_{1}; p_{2}; p_{3} \Big\}$ can be tought of as a composition of elementary essential moves that determine its topological properties. Those properties can be read directly on the matrix, on the snail itself, or on the types, $A$, $B$, $A^{-1}$ and $B^{-1}$ of those elementary homeomorphims. And it is important  to understand the role played by the line $\Delta$ in the expression of the homeomorphisms in $\homplus$ and the deep interest of the zip relations to exhibit their general properties. 

The orientation is perhaps the simplest of thoses caracteristics. Since the application $\sigma$ doesn't change neither the green $\sn(1;0)$ nor the red $\sn(0;1)$, the composition on the right by $\sigma$ keeps the physical topological snail unchanged. This means that the homeomorphisms $f$ and $f\circ \sigma$ can be thought of as in some sens conjugated relatively to the line $\Delta$. And it is therefore natural to add an orientation on the snail, with the following meaning : one sets that the positive orientation on $\sn(f)$ always keeps on its left the image by $f$ of the half-plane $\H^{+}$. It is therefore directed from the green to the red if and only if $f$ preserves the orientation, and equivalently if and only is $\mathrm{det}(\mt(f))=+1$.

Let us now suppose that $f$ preserves the orientation. Since $\varsigma$ is an involution, the forms $f\in \mathcal{T}^{+}$ and $\varsigma \circ f \circ \varsigma$ are simply conjugated : the homeomorphisms of the form $\varsigma \circ f \circ \varsigma$ relatively to $\Delta^{+}$ are purely turbulent relatively to $\Delta^{-}$. And if $f_{A}$ is of the type $A$ relatively to $\Delta^{+}$, its conjugate $\varsigma \circ f_{A} \circ \varsigma = f_{B}^{-1}$ is also of the type $A$ relatively to $\Delta^{-}$. Similarily, the forms $\varsigma \circ f$ and $\varsigma \circ f$ are conjugated. And using the zip-relations, we are able to understand the dynamical meaning of the presence of this factor $\varsigma$. Consider for example $f= f_{B} \circ f_{A}^{3} \circ f_{B}^{2} \circ f_{A} \circ f_{B} \circ f_{A} \circ \varsigma$. Use the zip-relations to write 
\begin{eqnarray}
\mt(f) & = &AB^{3}A^{2}BABZ \\
& = & \left(AB\right)\left(B^{2}A^{2}BZ\right)\left(B^{-1}A^{-1}\right) \\
& = & \left(AB\right)\left(B^{2}A^{2}B\left(B^{-1}AB^{-1} \right)\right)\left(B^{-1}A^{-1}\right) \\
& = & \left(AB^{2}\right)\left(BA^{2}\right)\left(B^{-2}A^{-1}\right) \\
& = & \left(AB^{2}\right)\left(BA^{2}\right)\left(AB^{2}\right)^{-1}
 \end{eqnarray}
If $\varphi =  f_{A} \circ f_{B}^{2}$, then $f = \varphi \circ \left( f_{B} \circ f_{A}^{2} \right)\circ \varphi^{-1}$. The homeomorphisms $f$ and $ f_{B} \circ f_{A}^{2}$ are therefore  conjugated via the homeomorphism $\varphi =  f_{A} \circ f_{B}^{2} \in \mathcal{T}^{+}$ and $f$ is a turbulent homeomorphism, positively turbulent relatively to the topological line $\varphi(\Delta)$. The $Z$ factor on the right is therefore the sign that the line $\Delta$ is not optimal to read the topological and dynamical properties of $f$. On the contrary, it has to be replaced in order to obtain a coincidence between the canonical form for $f$ and its intrisic properties.  The preceeding method is nevertheless universal : if one defines the equivalence relation on isotopy classes of homeomorphisms
$$f\equiv g \iff \exists \varphi \in \hom ~~|~~f \equiv \varphi \circ g \varphi^{-1}\, ,$$
then one can also naturally call turbulent an homeomorphism $g$  that is conjugated to some homeomorphism $f \in \mathcal{T}^{+}$ positively turbulent relatively to $\Delta$. If furthermore one denotes by $\mathcal{T}$ the set of turbulent homeomorphisms,  we obtain the following very nice classification theorem~:

\begin{theo}\label{classification}
Let $f \in \homplus$ be an orientation-preserving homeomorphism with an oriented line $\Delta^{+}$ such that $X=\lbrace p_{1};p_{2};p_{3} \rbrace\subset \Delta^{+}$. Then the absolute value $\vert \tr(\mt) \vert$ of the trace of the turbulence matrix $\mt(f)$ relatively to $\Delta^{+}$ defines a caracteristic number $\tr(f)$, constant on the isotopy classes of $f$ in $\hom$, and such that one has~:
\begin{eqnarray*}
\tr(f) = 0 & \iff & f^{2} \equiv \id  \\
\tr(f) = 1 & \iff & f^{3} \equiv \id \\
\tr(f) = 2 & \iff & f \equiv f_{A}^{n} \equiv f_{B}^{n}~~\mathrm{for~some}~n\in \Z \\
\tr(f) \geqslant 3 & \iff & f \equiv g \in \mathcal{T}.
\end{eqnarray*}

\end{theo}

\begin{proof}
Let us first prove that the trace $\tr(f)$ is well defined and constant on the isotopy classes of $f$ in $\homplus$. Let us first consider $f\in\homplus$ and two lines $\Delta_{1}^{+}$ and $\Delta_{2}^{+}$ with $X=\lbrace p_{1};p_{2};p_{3} \rbrace\subset \Delta_{1}^{+} \cap \Delta_{2}^{+}$ and associated respective turbulence matrices $M^{1}_{T}(f), M^{2}_{T}(f) \in \mathrm{PSL}_{2}(\Z)$. Suppose that $\Delta_{1}^{+}$ is the horizontal axis and 

$$p_{1} <_{\Delta_{1}^{+}}p_{2} <_{\Delta_{1}^{+}} p_{3}$$ 

Consider the snails on this axis and define a permutation $\varphi$ of the set $X$ such that 
$$\varphi(p_{1}) <_{\Delta_{2}^{+}} \varphi(p_{2}) <_{\Delta_{2}^{+}} \varphi(p_{3})$$
The spinning skeleton of the curve $[\varphi(p_{1}) ;\varphi(p_{3})]_{\Delta_{2}^{+}}$ relatively to $\Delta_{1}^{+}$ defines a topological snail, image of $[p_{1} ;p_{3}]_{\Delta_{1}^{+}}$ by a unique homeomorphism $\varphi \in \homplus$. In the isotopy class of $\varphi$ inside $\homplus$, choose\footnote{This is a fundamental topological result but its proof is out of the scope of the present paper.} an element such that $\varphi(\Delta_{1}^{+}) = \Delta_{2}^{+}$ and set 
$g = \varphi^{-1} \circ f \circ \varphi$. Then $f= \varphi \circ g \circ \varphi^{-1}$ and by theorem \ref{canonicalform} one has
$$\mt^{1}(f) =  \mt^{1}(\varphi \circ g \circ \varphi^{-1}) =\mt^{1}(\varphi) \times \mt^{1}( g ) \times \left( \mt^{1}(\varphi)\right)^{-1}$$
Therefore $\tr(\mt^{1}(f))=\tr(\mt^{1}(g))$. But it is also clear since $\varphi\in \homplus$ and $\varphi(\Delta_{1}^{+}) = \Delta_{2}^{+}$ that the homeomorphisms of the respective types $A$ and $B$ relatively to $\Delta_{2}^{+}$ and $\Delta_{1}^{+}$ are respectively conjugated via $\varphi$. This therefore implies that $\mt^{2}(f) = \mt^{1}(g)$ and $\tr(\mt^{1}(f))=\tr(\mt^{2}(f))$. The trace is therefore a well defined caracterisic number on $\homplus$. And since the map $f\rightarrow \mt(f)$ is an isomorphism, it is of course constant on the hole conjugacy class of $f$.

It is also very easy to caracterize the nature of $f$ using the zip-relations and the canonical form (theorem \ref{canonicalform}) for $f\in \homplus$. 

The canonical form relatively to $\Delta^{+}$ appears as a word $w$ in the letters $A$ and $B$ and the additional letter $Z$ once or twice.  
The fact that $f\in \mathcal{T}^{+}$ simply means that $f$ can by uniquely decomposed into a product of homeomorphisms of the respective types $A$ and $B$, with at least a letter $A$ and a letter $B$. Since $\tr(AB) = \tr(BA)=3$ and since the trace is strictly increasing after any mutliplication by a finite number of factors $A$ and $B$, one has $\tr(f) \geqslant 3$ if $f\in \mathcal{T}^{+}$ and also $\tr(f) =2$ is $f=f_{A}^{n}$ of $f=f_{B}^{n}$ for some $n\in \Z$. If there are two $Z$ factors, it is sufficient to reverse the orientation on $\Delta^{+}$ to obtain the preceeding form. But the canonical form can always be reduced in other cases and this way the result of the theorem becomes obvious. 

When unique, the factor $Z$ can be supposed to be at the first position after the word $w$ in the letters $A$ and $B$. Let $n(f) = n(w)$ be the total number of letters with their multiplicity in this word $w$.  If $n(w)\geqslant 2$, then $f$ can always be conjugated to an homeomorphism $f'$ whose canonical form has only $n-1$ letters $A$ or $B$ or $n-2$ letters and the final letter $Z$ : 

\begin{itemize}
\item If  $w=Aw'A$, 
where $w'$ is a possibly trivial word in $A$ an $B$, change $Z$ to $A^{-1}BA^{-1}$ to write  
\begin{eqnarray*} wZ & = & Aw'AZ=Aw'A\left(A^{-1}BA^{-1}\right)= A\left(w'B\right)A^{-1}
\end{eqnarray*}
\item If $w=Bw'B$ write similarily~:
\begin{eqnarray*} wZ=Bw'BZ=Bw'B\left(B^{-1}AB^{-1}\right)= B\left(w'A\right)B^{-1}\,. 
\end{eqnarray*}
\item If $w=Aw'B$, simply zip Z to obtain~:
\begin{eqnarray*} wZ & = & Aw'BZ=A\left(w'Z\right)A^{-1} 
\end{eqnarray*}

\item If $w=Bw'A$ then write 
\begin{eqnarray*}  wZ &= &Bw'AZ=B\left(w'Z\right)B^{-1}
\end{eqnarray*}

\end{itemize}

This process progressively conjugates $f$ either to a positive power of $f_{A}$, a positive power of $f_{B}$,  the homeomorphisms $f_{A} \circ \varsigma$ or $f_{B} \circ \varsigma$ or the homeomorphism $\varsigma$. Therefore, any homeomorphism $f\in \homplus$ can be conjugated to a homeomorphism with the following canoncial forms~: 
\begin{itemize}
\item A word $w$ with at least one $A$ and one $B$, when $\tr(f) \geqslant 3$ and$f\in \mathcal{T}$.
\item A power $A^{n}$ or $B^{n}$ for $n\in \Z$ when $\tr(f)=2$.
\item One of the forms $AZ$, $BZ$, $ZA$, $ZB$ when $\tr(f)=1$
\item The form $Z$ when $\tr(f)=0$.
\end{itemize}

\end{proof}

\begin{remarks}
\item The homeomorphisms $f_{A}$ and $f_{B}$ are conjugated via $\tau = \varsigma \circ \sigma$ wich doesn't preserve the orientation. They are not conjugated via an element of $\homplus$ for their respective periodic orbit in $X$ don't circulate around their centers in the same direction.  

\item The homeomorphism $BZ$ is isotopic to a rotation of an angle $+\frac{2\pi}{3}$ around a center located inside a topogical circle whose direct order is such that $p_{1} <p_{2} <p_{3}$. Its inverse is $AZ$ since $AZBZ\equiv AA^{-1} \equiv \id$. The homeomorphism $ZA$  is isotopic to a rotation of an angle $+\frac{2\pi}{3}$ around a center located inside a topogical circle whose indirect order is such that $p_{1} <p_{2} <p_{3}$. Its inverse is $ZB$.

\item The result of the theorem can also be very nicely obtained using the classical Cayley-Hamilton and Jordan decomposition theorem. Chose the representant of the matrix $\mt(f)\in \mathrm{PSL}_{2}(\Z)$ whose trace is positive. Then its characteristic polynomial is 
$$\chi_{f}(X) = X^{2} - \tr(f) X + \id $$
and always satisfies 
$$ \chi_{f}(\mt(f)) = \left(\mt(f)\right)^{2} - \tr(f) \left(\mt(f)\right) + \id = 0$$
\begin{itemize}
\item If $\tr(f)=0$ then in $\mathrm{PSL}_{2}(\Z)$ one has $\left(\mt(f)\right)^{2} = -\id = \id$. Wich implies that $f^{2} \equiv \id$ by the theorem \ref{canonicalform}.
\item If $\tr(f)=1$ then 
$$\left(\mt(f)\right)^{2} - \left(\mt(f)\right) + \id = 0$$
and since $X^{3} -1 = (X-1)(X^{2} - X + 1)$ one obtains 
$$\left(\mt(f)\right)^{3} =\id\, .$$
But as before by theorem \ref{canonicalform} this implies that $f^{3} \equiv \id$.
\item If  $\tr(f)=2$ then $1$ is a proper value of order $2$ of $\mt(f)$, wich is conjugate in $\mathrm{SL}_{2}(\Z)$ to its Jordan form that is a relative power of $A$. And by the theorem \ref{canonicalform} in this case $f$ is conjugated to a relative power of $f_{A}$.

\item One can finally remark that $\tr(f)\geqslant 3$ if and only if $\mt(f)$ has a largest proper value equal to  
$$\lambda = \frac{\tr(f) + \sqrt{\tr(f)^{2} - 4}}{2} >1\, .$$
Since $\mt(f^{n})= \mt(f)^{n}$, $f$ can no longer be periodic nor conjugated to a power of $A$. It is therefore a turbulent homeomorphism and this finishes to establish the classification.

\end{itemize}

\item The homeomorphisms $f_{A}$ and $f_{B}$ can be thought of as fixing a point, respectively $p_{1}$ and $p_{3}$, and rotating the two others of a respectively positive and negative half-turn. A correponding homeomorphism fixing $p_{2}$ is $A^{2}Z = ABA^{-1}= ZB^{-2}= BA^{-1}B^{-1}$. One can also consider three positive half-turns $A$, $B^{-1}= ZAZ$, $AB^{-1}A^{-1} = A^{2}Z$.

\item If we consider a process during wich a homeomorphism evolves via successive compositions of homeomorphisms of the type $f_{A}$, $f_{B}$, $f_{A}^{-1}$, $f_{B}^{-1}$ and measure its complexity by its trace (recall that according to the annoucement of this paper it is a good idea to do so), then it can increase constantly and suddenly happen to break down. For example, if one computes 
$$A^{3}B^{3}AB^{4}A^{3}B^{2}A = \left[ \begin{array}{cc} 460 & 659 \\ 141 & 202 \end{array} \right]  $$
the trace of the associated homeomorphism is $\tr(f) = 460 + 202 = 662$. But this trace collapses when you compose $f$ by $f_{B}^{-1}$, since
$$ \tr\left(\left[ \begin{array}{cc} 1 & 0 \\ -1 & 1 \end{array} \right]\left[ \begin{array}{cc} 460 & 659 \\ 141 & 202 \end{array} \right] \right) = \tr\left(\left[ \begin{array}{cc} 460 & 659 \\ -319 & -457 \end{array} \right]\right)=3  $$
The homeomorphism $f$ was positively turbulent relatively to $\Delta$, but $f_{B}^{-1} \circ f$ happens to be conjugated to $f_{A}f_{B}$ via a complicated homeomorphism and this explain the breakdown of its complexity. To obtain such an example, simply use the zip relations to write 
\begin{eqnarray*}
\mt(f) & = & \left(B^{2}A^{3}B^{2}A\right)^{-1} (AB)\left(B^{2}A^{3}B^{2}A\right) \\
& = &A^{-1} B^{-2} A^{-3} B^{-2} (AB)\left(B^{2}A^{3}B^{2}A\right) \\
& = & \left( Z B A^{2} B^{3} A^{2} Z \right) (AB) \left(B^{2}A^{3}B^{2}A\right) \\
& = & ZBA^{2} B^{3} AB^{4}A^{3}B^{2}A \\
& = & B^{-1}A^{3} B^{3} AB^{4}A^{3}B^{2}A 
\end{eqnarray*}

\end{remarks}

\section{Action on $X$ and reduction modulo 2 of $\mt(f)$.}

From a general point of view, the group $\hom$ naturaly acts by restriction on $X$ and a specialy important of its subgroups is the stabilizator for this action. It is made of those homeomorphisms who fix each of the three points $p_{1}$, $p_{2}$, $p_{3}$.

First consider 
the application $\displaystyle \varphi~: X \cup \lbrace \infty \rbrace \longrightarrow \left(\frac{\Z}{2\Z}\right)^{2}$ defined by 
$$\varphi(p_{\infty}) = \left[ \begin{array}{c} 0 \\ 0 \end{array}\right] ~~~~~\varphi(p_{1}) = \left[ \begin{array}{c} 1 \\ 0 \end{array}\right] ~~~~~  \varphi(p_{2}) = \left[ \begin{array}{c} 1 \\ 1 \end{array}\right] ~~~~~  \varphi(p_{1}) = \left[ \begin{array}{c} 0 \\ 1 \end{array}\right].$$

For $f\in \hom$, define $\widetilde{f} = \varphi \circ f \circ \varphi^{-1}$. Then $\widetilde{f}$ is a map from $\left(\frac{\Z}{2\Z}\right)^{2}$ to itself and one has the following very nice application of the theorem \ref{canonicalform}~:

\begin{theo} \label{modulo} The map $\widetilde{f}$ associated to any homeomorphism $f\in\homplus$ is a linear isomorphism whose canonical matrix $\rdeux(f)$ is the reduction modulo 2 of the turbulence matrix $\mt(f)$.
\end{theo}

\begin{proof}
With the matrix $A=\left[ \begin{array}{cc} 1 & 1 \\ 0 & 1 \end{array}\right]$, compute modulo 2~:
$$\begin{array}{cc} \left[ \begin{array}{cc} 1 & 1 \\ 0 & 1 \end{array}\right] \left[ \begin{array}{c} 0 \\ 0 \end{array}\right]= \left[ \begin{array}{c} 0 \\ 0 \end{array}\right]& \left[ \begin{array}{cc} 1 & 1 \\ 0 & 1 \end{array}\right] \left[ \begin{array}{c} 0 \\ 1 \end{array}\right]= \left[ \begin{array}{c} 1 \\ 1 \end{array}\right] \\ 
& \\
\left[ \begin{array}{cc} 1 & 1 \\ 0 & 1 \end{array}\right] \left[ \begin{array}{c} 1 \\ 1 \end{array}\right]= \left[ \begin{array}{c} 0 \\ 1 \end{array}\right] & \left[ \begin{array}{cc} 1 & 1 \\ 0 & 1 \end{array}\right] \left[ \begin{array}{c} 1 \\ 0 \end{array}\right]= \left[ \begin{array}{c} 1 \\ 0 \end{array}\right] \end{array}$$
Check that $A$ acts on $\varphi \left(X \cup \lbrace \infty \rbrace \right)$ as $f_{A}$ on $X \cup \lbrace \infty \rbrace$. But $B$ also acts on $\varphi \left(X \cup \lbrace \infty \rbrace \right)$ as $f_{B}$ on $X$ and this is also the case for $Y$ with $\sigma$ and $Z$ with $\varsigma$. This proves the result using theorem \ref{canonicalform}.
\end{proof}

This result also shows that the homeomorphism $f$ induces a permutation on $X$ in a certain natural order, corresponding to a certain complexity, obtained by taking successively the applications $f_{A}$ and $f_{B}$, according to the following sequences of permutations~:
$$\begin{array}{|r|c|l|c|} \hline  A & (1;2;3) \mapsto (2; 1; 3) & BABAB & A\\
BA & (1;2;3) \mapsto (2; 3; 1) & ABAB & AZ\\
ABA & (1;2;3) \mapsto (3; 2; 1) & BAB & Z\\
BABA & (1;2;3) \mapsto (3; 1; 2) & AB & ZA\\
ABABA & (1;2;3) \mapsto (1; 3; 2) & B & ZAZ\\
 BABABA & (1;2;3) \mapsto (1; 2; 3) & \id & \id\\ \hline \end{array}$$

\section{Circulation of the middle point. \label{relative}}

Let $(f_{t})_{t\in [0;1]}$ be an isotopy from $f_{0} = \id$ to $f_{1} = f$ and suppose that $\rdeux(f)\equiv \id$, for fixed points $p_{1}= -1$, $p_{2} = 0$ and $p_{3}=1$. 

With $z=x+iy \in \C$, let us define the function $h:\C \rightarrow [0;+\infty[$ by
$$h(z) = \vert z \vert~~\mathrm{if}~~ x\leqslant 0~~~~~\mathrm{and}~~~~~~h(z) = \max(\vert x \vert ; \vert y \vert)~~\mathrm{if}~~x\geqslant 0$$  
The function $h$ is therefore constant on each curve obtained by joinning half of a circle of radius $r\geqslant 0$ and center $O$ in the half plane $x\leqslant 0$ and a half of a square of width $h$ and symetric relatively to $y=0$ in the half plane $x\geqslant 0$. 

In a similar manner, for $z=x+iy$ let us define the function $k: \C \mapsto [0,+\infty[$ by 
$$k(z) = \vert z -1 \vert~~\mathrm{if}~~ x\leqslant -1\, ,~~~k(z) =\vert y \vert~~\mathrm{if}~~-1 \leqslant x \leqslant 1~~~\mathrm{and}~~~k(z) = \vert z -1\vert~~\mathrm{if}~~ x\geqslant 1$$ 
Then $k$ is constant on the curves obtained by joining half of a disc of radius $r$ and center $-1$ in the half plane $x\leqslant -1$, two horizontal segments, and half of a disc of radius $r$ and center $1$ in the half plane $x\geqslant 1$. This function is also a measure of the distance between the point of affix $z$ and the segment $[-1;1]$. For each $z\in \C$, let $\omega(z)$ be the unique point $\omega(z) \in [-1;1]$ that satisfies $\vert z - \omega(z) \vert = k(z)$.

Let $N_{1}(0;-1)$, $N_{2}(0;1)$, $E_{1}(_2;0)$, $I(0;0)$ and $E_{3}(2;0)$ (see figure \ref{figure37}) and  let us define the set $\Omega=\Gamma_{1} \sqcup \Gamma_{2} \sqcup \Gamma_{2}$ with~:
\begin{itemize}
\item $\Gamma_{1} = \lbrace z=x+iy \in \C~~|~~h_{1}(z)=k(z) = 1 \rbrace$
\item  $\Gamma_{2} = \lbrace z=x+iy \in \C~~|~~h_{1}(z)=h_{2}(z) = 1 \rbrace$
\item $\Gamma_{3} = \lbrace z=x+iy \in \C~~|~~h_{2}(z)=k(z) = 1 \rbrace$
\end{itemize}
Then $\Gamma_{1}$, $\Gamma_{2}$ and $\gamma_{3}$ are simple arcs doubly branched on the set $\lbrace N_{1} ; N_{2} \rbrace$.

For $t\in ]0;2[$ and $x\in [0;+\infty]$ consider the coefficient defined by 
$$C(t;x) = \frac{x}{t} ~~\mathrm{if}~~x\in [0;t]\, ,C(t;x) = \frac{x+2-2t}{2-t} ~~\mathrm{if}~~x\in [t;2]\, ~~~\mathrm{and}~~~C(t;x) = x ~~\mathrm{if}~~x\geqslant 2$$ 

And for $m\in \C \setminus \lbrace p_{1}; p_{3} \rbrace$, let us define $\Psi_{m}(z)$ :
\begin{itemize}
\item If $h(m+1) \leqslant 1$ set $\Psi_{m}(-1)=-1$ and for $z\neq-1$:
\begin{eqnarray*} \Psi_{m}(z) & = & -1 + C\left(h(m+1); h(z+1 )\right) \times \frac{z +1}{\vert z+1 \vert}
\end{eqnarray*}

\item If $k(m) \geqslant 1$ set  $\Psi_{m}(z)=z$ for $k(z)=0$ 

\begin{eqnarray*} \Psi_{m}(z) = & \omega(z) +  \frac{\displaystyle z - \omega(z)}{\displaystyle k(m)}\end{eqnarray*}

\item if $h(1-m) \leqslant 1$ set $\Psi_{m}(1) =1$ and for $z\notin [-1;1]~:$

\begin{eqnarray*} \Psi_{m}(z) & = & 1 +  C(h(1-m);h(1-z)) \times \frac{z - 1}{\vert z -1 \vert}   
\end{eqnarray*}
\end{itemize}

\setlength{\unitlength}{0.7mm}
\begin{figure}
\begin{picture}(100,120)(-120,-65)

\color{blue}
\put(-40,-0.2){\line(1,0){80}}
\put(-40,0.2){\line(1,0){80}}
\put(-40,0){\line(1,0){80}}

\color{black}
\put(-40,0){\circle*{1.5}}
\put(-37,9){\circle*{0.5}}

\put(-36,8){\makebox(0,0)[tl]{$\scriptstyle m_{1}$}}

\put(-26,39){\makebox(0,0)[tl]{$\scriptstyle \Psi_{m_{1}}(m_{1})$}}

\put(-26.7,40){\circle*{0.5}}

\put(-40,-2){\makebox(0,0)[tc]{$p_{1}$}}

\put(-37,9){\vector(1,3){10.3}}

\put(-40,0){\qbezier[15](0,0)(1.5,4.5)(3,9)}

\put(55,-15){\vector(1,-1){13.28}}
\put(55,-15){\circle*{0.5}}
\put(68.28,-28.28){\circle*{0.5}}
\put(56,-14){\makebox(0,0)[bl]{$\scriptstyle m_{3}$}}
\put(70,-30){\makebox(0,0)[tl]{$\scriptstyle \Psi_{m_{3}}(m_{3})$}}

\put(40,0){\qbezier[25](0,0)(10,-10)(20,-20)}

\color{black}
\put(-80,0){\circle*{1.5}}
\put(-82,0){\makebox(0,0)[cr]{$E_{1}$}}

\put(80,0){\circle*{1.5}}
\put(82,0){\makebox(0,0)[cl]{$E_{2}$}}
\put(20,0){\circle*{1}}
\put(20,-2){\makebox(0,0)[tc]{$\scriptstyle \omega(z)$}}

\put(20,0){\qbezier[25](0,0)(0,20)(0,40)}
\put(20,60){\vector(0,-1){20}}
\put(20,60){\circle*{0.5}}
\put(20,40){\circle*{0.5}}
\put(21,60){\makebox(0,0)[lc]{$\scriptstyle m_{2}$}}
\put(21,39){\makebox(0,0)[tl]{$\scriptstyle \Psi_{m_{2}}(m_{2})$}}

\put(0,0){\circle*{1.5}}
\put(-2,-1){\makebox(0,0)[tr]{$ I$}}

\put(-75,-20){\makebox(0,0)[tr]{$\Gamma_{1}$}}
\put(2,-20){\makebox(0,0)[tl]{$\Gamma_{2}$}}
\put(75,23){\makebox(0,0)[bl]{$\Gamma_{3}$}}

\color{blue}

\color{blue}

\color{black}

\put(40,0){\circle*{1.5}}
\put(42,0){\makebox(0,0)[cl]{$p_{3}$}}

\put(-40,40){\line(1,0){80}}
\put(-40,-40){\line(1,0){80}}
\put(0,-40){\line(0,1){80}}

\put(40,0){\qbezier[300](40,0)(40,-16.56)(28.28,-28.28)}
\put(40,0){\qbezier[300](28.28,-28.28)(16.56,-40)(0,-40)}
\put(40,0){\qbezier[300](0,40)(16.56,40)(28.28,28.28)}
\put(40,0){\qbezier[300](28.28,28.28)(40,16.56)(40,0)}

\put(-40,0){\qbezier[300](0,-40)(-16.56,-40)(-28.28,-28.28)}
\put(-40,0){\qbezier[300](-28.28,-28.28)(-40,-16.5)(-40,0)}
\put(-40,0){\qbezier[300](-40,0)(-40,16.56)(-28.28,28.28)}
\put(-40,0){\qbezier[300](-28.28,28.28)(-16.56,40)(0,40)}


\color{black}

\color{purple}
\put(0,40){\circle*{1.5}}
\put(0,42){\makebox(0,0)[bc]{$N_{2}$}}

\put(0,-40){\circle*{1.5}}
\put(0,-42){\makebox(0,0)[tc]{$ N_{1}$}}

\end{picture}
\caption{Action on $m$ of the family of homeomorphisms $\Psi_{m}$. \label{figure37}}

\end{figure}
\setlength{\unitlength}{0.7mm}

For each $m\in \C$ such that $m\notin \lbrace p_{1}; p_{3}\rbrace$, the map $z\mapsto \Psi_{m}(z)$ is a homeomorphism, it depends continuously on $m$ and fixes $p_{1}$ and $p_{3}$. It sends $m$ on $\Omega$. 

If $h(m+1)~\leqslant~1$, then $f$ respects the foliation of $\C$ by the curves $h(z+1)=\lambda \geqslant 0$, and by the half-lines of origin $-1$, fixes the points $z$ such that $h(z+1)\geqslant 2$ and sends $m$ on the curve $\Gamma_{1} \sqcup \Gamma_{2}$. If $m$ is on this curve, then $z\mapsto \Psi_{m}(z)$ fixes all points of $\C$. The situation is symetric if $h(m-1)=h(1-m) \leqslant 1$ and similar if $k(m)\geqslant 1$ (see figure \ref{figure37}). 

\setlength{\unitlength}{0.7mm}
\begin{figure}
\begin{picture}(100,120)(-120,-65)

\color{black}

\put(-40,0){\circle*{1.5}}
\put(-40,-2){\makebox(0,0)[tc]{$p_{1}$}}

\color{black}

\put(-80,0){\circle*{1.5}}
\put(-82,0){\makebox(0,0)[cr]{$E_{1}$}}

\put(80,0){\circle*{1.5}}
\put(82,0){\makebox(0,0)[cl]{$E_{2}$}}

\put(0,0){\circle*{1.5}}
\put(-2,0){\makebox(0,0)[cr]{$ I$}}

\color{blue}

\put(-35,40){\circle*{2}}
\put(-35,40){\vector(-1,0){20}}
\put(-35,40.2){\line(-1,0){18}}
\put(-35,40.4){\line(-1,0){18}}
\put(-35,39.8){\line(-1,0){18}}
\put(-35,39.6){\line(-1,0){18}}

\color{black}

\put(-35,43){\makebox(0,0)[bc]{$M_{t}$}}

\color{black}

\put(40,0){\circle*{1.5}}
\put(40,-2){\makebox(0,0)[tc]{$p_{3}$}}

\put(-40,40){\line(1,0){80}}
\put(-40,-40){\line(1,0){80}}
\put(0,-40){\line(0,1){80}}

\put(40,-2){\qbezier[300](-22,0)(-22,-9.1)(-15.55,-15.55)}
\put(40,-2){\qbezier[300](-15.55,-15.55)(-9.1,-22)(0,-22)}
\put(40,-2){\qbezier[300](0,-22)(9.1,-22)(15.55,-15.55)}
\put(40,-2){\qbezier[300](15.55,-15.55)(22,-9.1)(22,0)}

\put(40,2){\qbezier[300](-22,0)(-22,9.1)(-15.55,15.55)}
\put(40,2){\qbezier[300](-15.55,15.55)(-9.1,22)(0,22)}
\put(40,2){\qbezier[300](0,22)(9.1,22)(15.55,15.55)}
\put(40,2){\qbezier[300](15.55,15.55)(22,9.1)(22,0)}

\put(-40,-2){\qbezier[100](-18,0)(-18,-7.45)(-12.72,-12.72)}
\put(-40,-2){\qbezier[100](-12.72,-12.72)(-7.45,-18)(0,-18)}
\put(-40,-2){\qbezier[100](0,-18)(7.45,-18)(12.72,-12.72)}
\put(-40,-2){\qbezier[100](12.72,-12.72)(18,-7.45)(18,0)}

\put(-40,2){\qbezier[100](-18,0)(-18,7.45)(-12.72,12.72)}
\put(-40,2){\qbezier[100](-12.72,12.72)(-7.45,18)(0,18)}
\put(-40,2){\qbezier[100](0,18)(7.45,18)(12.72,12.72)}
\put(-40,2){\qbezier[100](12.72,12.72)(18,7.45)(18,0)}

\put(40,-2){\qbezier[100](-18,0)(-18,-7.45)(-12.72,-12.72)}
\put(40,-2){\qbezier[100](-12.72,-12.72)(-7.45,-18)(0,-18)}
\put(40,-2){\qbezier[100](0,-18)(7.45,-18)(12.72,-12.72)}
\put(40,-2){\qbezier[100](12.72,-12.72)(18,-7.45)(18,0)}

\put(40,2){\qbezier[100](-18,0)(-18,7.45)(-12.72,12.72)}
\put(40,2){\qbezier[100](-12.72,12.72)(-7.45,18)(0,18)}
\put(40,2){\qbezier[100](0,18)(7.45,18)(12.72,12.72)}
\put(40,2){\qbezier[100](12.72,12.72)(18,7.45)(18,0)}

\put(62,2){\vector(0,-1){0}}
\put(22,2){\vector(0,-1){0}}
\put(58,-2){\vector(0,1){0}}
\put(18,-2){\vector(0,1){0}}

\put(-62,2){\vector(0,-1){0}}
\put(-22,2){\vector(0,-1){0}}
\put(-58,-2){\vector(0,1){0}}
\put(-18,-2){\vector(0,1){0}}

\put(-40,26){\makebox(0,0)[bc]{$\scriptstyle A$}}
\put(-40,19){\makebox(0,0)[tc]{$\scriptstyle A^{-1}$}}
\put(-40,-26){\makebox(0,0)[tc]{$\scriptstyle A$}}
\put(-40,-19){\makebox(0,0)[bc]{$\scriptstyle A^{-1}$}}

\put(40,26){\makebox(0,0)[bc]{$\scriptstyle B$}}
\put(40,19){\makebox(0,0)[tc]{$\scriptstyle B^{-1}$}}
\put(40,-26){\makebox(0,0)[tc]{$\scriptstyle B$}}
\put(40,-19){\makebox(0,0)[bc]{$\scriptstyle B^{-1}$}}

\put(-40,-2){\qbezier[300](-22,0)(-22,-9.1)(-15.55,-15.55)}
\put(-40,-2){\qbezier[300](-15.55,-15.55)(-9.1,-22)(0,-22)}
\put(-40,-2){\qbezier[300](0,-22)(9.1,-22)(15.55,-15.55)}
\put(-40,-2){\qbezier[300](15.55,-15.55)(22,-9.1)(22,0)}

\put(-40,2){\qbezier[300](-22,0)(-22,9.1)(-15.55,15.55)}
\put(-40,2){\qbezier[300](-15.55,15.55)(-9.1,22)(0,22)}
\put(-40,2){\qbezier[300](0,22)(9.1,22)(15.55,15.55)}
\put(-40,2){\qbezier[300](15.55,15.55)(22,9.1)(22,0)}

\put(40,0){\qbezier[300](40,0)(40,-16.56)(28.28,-28.28)}
\put(40,0){\qbezier[300](28.28,-28.28)(16.56,-40)(0,-40)}
\put(40,0){\qbezier[300](0,40)(16.56,40)(28.28,28.28)}
\put(40,0){\qbezier[300](28.28,28.28)(40,16.56)(40,0)}

\put(-40,0){\qbezier[300](0,-40)(-16.56,-40)(-28.28,-28.28)}
\put(-40,0){\qbezier[300](-28.28,-28.28)(-40,-16.5)(-40,0)}
\put(-40,0){\qbezier[300](-40,0)(-40,16.56)(-28.28,28.28)}
\put(-40,0){\qbezier[300](-28.28,28.28)(-16.56,40)(0,40)}

\color{red}

\put(-80,1){\vector(0,1){20}}
\put(-79.8,1){\line(0,1){18}}
\put(-80.2,1){\line(0,1){18}}
\put(-79.6,1){\line(0,1){18}}
\put(-80.4,1){\line(0,1){18}}

\put(80,1){\vector(0,1){20}}
\put(79.8,1){\line(0,1){18}}
\put(80.2,1){\line(0,1){18}}
\put(79.6,1){\line(0,1){18}}
\put(80.4,1){\line(0,1){18}}

\put(0,1){\vector(0,-1){20}}
\put(-0.2,1){\line(0,-1){18}}
\put(0.2,1){\line(0,-1){18}}
\put(-0.4,1){\line(0,-1){18}}
\put(0.4,1){\line(0,-1){18}}

\color{black}
\put(-2,20){\line(0,1){18}}
\put(-2,38){\vector(-1,0){18}}
\put(-4,36){\makebox(0,0)[cc]{$\scriptstyle l$}}
\put(-8,32){\vector(0,-1){12}}
\put(-8,32){\line(-1,0){12}}
\put(-10,30){\makebox(0,0)[cc]{$\scriptstyle r$}}


\put(2,20){\line(0,1){18}}
\put(2,38){\vector(1,0){18}}
\put(4,36){\makebox(0,0)[cc]{$\scriptstyle r$}}
\put(8,32){\vector(0,-1){12}}
\put(8,32){\line(1,0){12}}
\put(10,30){\makebox(0,0)[cc]{$\scriptstyle l$}}

\put(-2,-20){\line(0,-1){18}}
\put(-2,-38){\vector(-1,0){18}}
\put(-4,-36){\makebox(0,0)[cc]{$\scriptstyle r$}}
\put(-8,-32){\vector(0,1){12}}
\put(-8,-32){\line(-1,0){12}}
\put(-10,-30){\makebox(0,0)[cc]{$\scriptstyle l$}}


\put(2,-20){\line(0,-1){18}}
\put(2,-38){\vector(1,0){18}}
\put(4,-36){\makebox(0,0)[cc]{$\scriptstyle l$}}
\put(8,-32){\vector(0,1){12}}
\put(8,-32){\line(1,0){12}}
\put(10,-30){\makebox(0,0)[cc]{$\scriptstyle r$}}

\put(-20,42){\vector(1,0){40}}
\put(-8,44){\makebox(0,0)[cc]{$\scriptstyle l$}}

\put(20,48){\vector(-1,0){40}}
\put(8,49){\makebox(0,0)[bc]{$\scriptstyle r$}}

\put(20,-42){\vector(-1,0){40}}
\put(-8,-44){\makebox(0,0)[cc]{$\scriptstyle l$}}

\put(-20,-48){\vector(1,0){40}}
\put(8,-49){\makebox(0,0)[tc]{$\scriptstyle r$}}

\color{mygreen}

\put(-80,-1){\vector(0,-1){20}}
\put(-79.8,-1){\line(0,-1){18}}
\put(-80.2,-1){\line(0,-1){18}}
\put(-79.6,-1){\line(0,-1){18}}
\put(-80.4,-1){\line(0,-1){18}}

\put(80,-1){\vector(0,-1){20}}
\put(79.8,-1){\line(0,-1){18}}
\put(80.2,-1){\line(0,-1){18}}
\put(79.6,-1){\line(0,-1){18}}
\put(80.4,-1){\line(0,-1){18}}

\put(0,1){\vector(0,1){20}}
\put(-0.2,1){\line(0,1){18}}
\put(0.2,1){\line(0,1){18}}
\put(-0.4,1){\line(0,1){18}}
\put(0.4,1){\line(0,1){18}}

\color{black}

\put(-82,10){\makebox(0,0)[cr]{$\scriptstyle Z$}}
\put(-82,-10){\makebox(0,0)[cr]{$\scriptstyle \id$}}
\put(82,10){\makebox(0,0)[cl]{$\scriptstyle Z$}}
\put(82,-10){\makebox(0,0)[cl]{$\scriptstyle \id$}}
\put(-2,10){\makebox(0,0)[cr]{$\scriptstyle \id$}}
\put(2,-10){\makebox(0,0)[cl]{$\scriptstyle Z$}}

\color{purple}
\put(0,40){\circle*{1.5}}

\put(40,50){\vector(-4,-1){28}}
\put(41,50){\makebox(0,0)[lc]{$N_{2}$}}
\put(0,-40){\circle*{1.5}}
\put(0,-40){\circle*{1.5}}
\put(-40,-50){\vector(4,1){28}}
\put(-41,-50){\makebox(0,0)[cr]{$N_{1}$}}

\end{picture}
\caption{The graph $\Omega$ for the circulating point $M_{t} = f_{t}(0)$. \label{figure38}}

\end{figure}
\setlength{\unitlength}{0.7mm}

One can first suppose that for all $t\in [0;1]$, $f_{t}(-1) = -1$ and $f_{t}(1) = 1$~: if necessary, replace $(f_{t})$ by the isotopy
$$\tilde{f}_{t}(z) = \frac{2f_{t}(z) - f_{t}(-1) - f_{t}(1)}{f_{t}(1) - f_{t}(-1)}\, .$$    
Then define, for $m(t)=f_{t}(0)$, the isotopy $\varphi_{t} = \Psi_{m(t)} \circ f_{t}$. The new point $M_{t} = \varphi_{t}(0)$ keeps moving in $\Omega$, and one can then use further isotopies on $\varphi_{t}$ in order to suppress any step backward of the point $M_{t}$ in $\Omega$, and even suppose that the speed of $M_{t}$ is constant between the stations $N_{1}$ and $N_{2}$. In this last case, the movement of $M_{t}$ can be described from the starting point $I$ as an initial direction, to $N_{2}$ or to $N_{1}$, and a finite serie of turns on the left or on the right at the successive stations $N_{1}$ and $N_{2}$ (see figure \ref{figure38}), until the point returns back to the final point $I= f_{1}(I)$. 

Define a code $C(f)$ for this movement in the following manner~: take the succession of turns on left $l$ and right $r$ in their order of appearence from the right to the left (as for the composition of maps). Add a $Z$ on the right if the movement begins downward from $I$ to $N_{1}$, and add a $Z$ on the left if  the movement finishes downward from $N_{2}$ to $I$ (in the direction if the red arrow on figure \ref{figure38}). For example, the code for an eight inside $\Omega$, beginning in the direction of $N_{1}$ and turning first around $p_{1}$ left on the right and then around $p_{2}$ left on the right is $Z(l^{2}r^{2})Z$. This way, for any $f \in \hom$, one can define two unique caracteristic sequences of carcteristic numbers (see theorem \ref{caracteristic} page \pageref{caracteristic}) and two numbers $\varepsilon_{i}$ and $\varepsilon_{f}$ in  $\big\{ 0;1 \big\} $.

One has then the following result~: 

\begin{theo}\label{code}
Let $(f_{t})_{t\in [0;1]}$ be an isotopy whose code in $\Omega$ is $C(f)$. Then 
$$C(f) =Z^{\varepsilon_{f}} \left(\prod_{i=1}^{n} l^{\alpha_{i}}r^{\beta_{i}} \right)Z^{\varepsilon_{i}}  \iff  \Phi(f) = Z^{\varepsilon_{f}} \left(\prod_{i=1}^{n} A^{\alpha_{i}}B^{\beta_{i}} \right) Z^{\varepsilon_{i}} $$
\end{theo}

\begin{proof}
Consider the movement of $M_{t}$ from its starting point $I$ but change the description of the movement. Instead of the stations $N_{1}$ and $N_{2}$, use the stations $E_{1}$, $I$ and $E_{2}$ to decompose the movement into elementary moves in $\Omega$ between those stations. Consider an horizontal\footnote{An isotopy is horizontal if it keeps the second coordinate unchanged.} isotopy $(\varphi_{t})$ between $\varphi_{0}=\id$ and $\varphi = \varphi_{1}$ such that 
$\varphi(E_{1})=p_{1}$, $\varphi(p_{1})=I$, $\varphi(p_{3})=p_{3}$ and 
an horizontal isotopy $(\psi_{t})$ between $\psi_{0}=\id$ and $\psi=\psi_{1}$ such that 
$\psi(p_{1})=p_{1}$, $\psi(p_{3})=I$, $\psi(E_{2})=p_{3}$.
At each station of the moving point $M$ in $E_{1}$ or $E_{2}$, introduce respectively an additional round trip along $\varphi$, $\varphi^{-1}$ or $\psi$, $\psi^{-1}$. This way, a decomposition of the hole movement into elementary isotopies between homeorphisms in $\homplus$ appears.

Those elementary moves can easily be analysed on figure \ref{figure38} and traduced in terms of the composition of at most two of the applications $f_{A}$, $f_{B}$, $f_{A}^{-1}$, $f_{B}^{-1}$. Furthermore, if one changes the direction for the letter $Z$ in $E_{1}$ and $E_{2}$, the formulas for those applications supports the precise rule of the theorem : for a left turn $l$ the letter $A$, for a right turn $r$ the letter $B$, and a possible initial with a possible final $Z$ depending on the orientation upward or downward of the movement and according to figure \ref{figure38}. And since the movement of $M$ never goes back, the final position at one of those moves is the initial position at the following one : since $Z^{2} = \id$, only the first and the last possible letters $Z$ are conserved in the final formula for any possible composition, and all the letters $l$ and $r$ have been changed into the respective letters $A$ and $B$.

Let us check for exemple the move from $E_{1}$ passing through $N_{2}$ either to $I$ or $E_{2}$. The code for the first case is $ZrZ$ for in this case the point moves upward in $E_{1}$, turns on its right in $N_{2}$ and finishes in $I$ downward. But according to the zip rules, this move also corresponds to $A^{-1} = ZBZ$ (see figure \ref{figure38}). In the second case, one obtains the identity $AZ= AA^{-1}BA^{-1}= BA^{-1}$.    
\end{proof}

\begin{remarks}

\item A natural representation of $\mathrm{S}_{3}$ as a quotient of the free group generated by $A$ and $B$ and the relations $A^{2}=\id$; $B^{2}=\id$, $(AB)^{3}=\id$ naturally appears.

\item One can also check that for any initial position and direction in $\Omega$, twice the same turn on the left or on the right brings to the initial position. This corresponds of course to the reduction modulo $2$.

\item When $\rdeux(f) = \id$, one can consider partial topological numbers usually called the linking numbers : they count the total number of turns of a fixed point around another during an isotopy (see P. Le Calvez \cite{PL1}, \cite{PL2}, \cite{PL3}, \cite{PL4}, \cite{PL5}, and J.M. Gambaudo \cite{JMG1}, \cite{PL4}). Since it is always possible to consider an isotopy $\ft$ with two given fixed points immobilized during the isotopy, there cannot be any meaning for a concept of linking number between two fixed points of a homeomorphism of the plane $\R^{2}$. Those numbers depend on the choice of a particular isotopy and can even be arbitrarily fixed. On the contrary, once two fixed points $p_{1}$ and $p_{3}$ are given, the topological trajectory $f_{t}(p_{2})$ of a third fixed point or periodic orbit between the two others is a very interesting intrinsic caracteristic with a deep dynamical signification : this object defines the topological snail and is the basic tool to exhibit new fixed points induced by the turbulence phenomenon. It can by no mean be reduced to nor replaced by any of the partial linking numbers associated to an isotopy.

\item Given a fixed point $p$, one can naturally define an equivalence relation on the set of the other fixed points of $f$, by relating two fixed points $x$ and $y$ such that the trajectory of $p$ during an isotopy that keeps $x$ and $y$ immobilized doesn't separate them (the curve is homotopic to a point in $\mathbb{S}^{2} \setminus \lbrace x; y \rbrace$). In terms of the present theory, this relation means that if $p_{1}=x$, $p_{3}= y$, and if the code for $f$ is $\mathrm{C}(f) = Zgd^{2n+1}g$, or $\mathrm{C}(f) = gd^{2n+1}gZ$, or $\mathrm{C}(f) = Zdg^{2n+1}d$, or $\mathrm{C}(f) = dg^{2n+1}dZ$, then we say that $x$ and $y$ are related. And this is particularily interesting since this can of course be considered for an homeomorphism $f=F^{n}$ that is an iterate of another one, with a periodic point $p$ of order $n$ and precisely means that $x$ and $y$ are in the same Nielsen class relatively to $p$.

\item To compute the linking numbers for $f$ relatively to $\lbrace p_{1}; p_{3} \rbrace$, simply consider another decomposition of the movement of $f_{t}(p_{2})$ (see section \ref{relative} proof of theorem \ref{code}), with stations $E_{1}$; $I$ and $E_{2}$. When $M_{t}$ passes from $E_{1}$ to $E_{2}$ without visiting $I$, introduce an intermediary round trip on $[N_{1}; I]$ or $[I;N_{2}]$. The movement will therefore be naturally decomposed into successive powers of $\alpha= A^{2}$, $\beta=B^{2}$, $\alpha^{-1} = A^{-2}$, $\beta^{-1}=B^{-2}$, with no possible simplification. The sum of the exponents of $\alpha$ will be the linking number of $p_{2}$ relatively to $p_{1}$, and the sum of the exponents of $\beta$ will be the linking number of $p_{2}$ relatively to $p_{3}$.

\item The zip relations can also be used to compute the linking number from the turbulence code of an homeomorphism. Read the code from right to left, and if you encounter an odd power of a letter $A$ or $B$, increase it to the next even power, correct the formula on the left with the inverse power of this letter, and use the zip relations to introduce a $Z$. Go on until the hole code gets written with only even powers of $A$ and $B$ and some letters $Z$. When $\rdeux(f) = \id$ this algorithm finishes after an even number of letters $Z$, with only even powers of $A$ and $B$. For example~:

\begin{eqnarray*}
BABABA^{2}BAB^{5}AB^{3}A^{2} & = & BABABA^{2}BAB^{5}B\left(B^{-1}AB^{-1}\right)B^{4}A^{2} \\
& = &  BABABA^{2}A\left(A^{-1}BA^{-1}\right)A^{2}B^{6}Z\left(B^{4}A^{2}\right)\\
& = &  BABAA\left(A^{-1}BA^{-1}\right)\left(A^{4}\right)\left(ZA^{2}B^{6}Z\right)\left(B^{4}A^{2}\right)\\
& = &  BB\left(B^{-1}AB^{-1}\right)\left(B^{2}A^{2}Z\right)\left(A^{4}\right)\left(ZA^{2}B^{6}Z\right)\left(B^{4}A^{2}\right) \\
& = & B^{2}\left(ZB^{2}A^{2}Z\right)\left(A^{4}\right)\left(ZA^{2}B^{6}Z\right)\left(B^{4}A^{2}\right) \\
& = & \left(B^{2}\right)\left(A^{-2}B^{-2}\right)\left(A^{4}\right)\left(B^{-2}A^{-6}\right)\left(B^{4}A^{2}\right)
\end{eqnarray*}

In this case, the linking number with $p_{1}$ is $n_{1} = 2-6+4-2=-2$ and the linking number with $p_{3}$ is $n_{2}=4-2-2+2=2.$

 \item An important universal curve appears in the preceeding analysis. It corresponds to a code $gdgdg$, or $(AB)^{3}$. In this case, the trace $\tr(f)=3$ is minimal for a turbulent homeomorphism and the linking number with $p_{1}$ and $p_{3}$ is zero (see figure \ref{figure39}).

\item It will be relevant to consider a turbulent  homeomorphism as progressively composed of successive homeomorphisms of the type $A$ and $B$. In this case, once the homeomorphism becomes turbulent (for example some $f= f_{A} \circ f_{B}$), two fixed points appear that can be taken as base points $\omega_{A}$ and $\omega_{B}$ for a retrospective analysis of the considered homeomorphism : in this case, there are five points to consider~: the invariant set $X$ with the fixed $\omega_{A}$ and $\omega_{B}$. One can for example consider rotary homeomorphisms $f_{A}$ and $f_{B}$ respectively around $\omega_{A}$ and $\omega_{B}$.

\item The circulation around two fixed points follow a system of successive rotations according to inverse directions, as in the case of convection cells in fluid dynamics. And this point of view will generalize below to define general turbulence in the case of more than three points. 

\setlength{\unitlength}{0.7mm}
\begin{figure}
\begin{picture}(100,100)(-120,-55)

\color{black}

\put(-30,0){\circle*{1.5}}
\put(-30,-2){\makebox(0,0)[tc]{$P_{1}$}}

\color{black}

\put(-80,0){\circle*{1.5}}

\put(80,0){\circle*{1.5}}

\put(0,0){\circle*{1.5}}
\put(-2,0){\makebox(0,0)[tr]{$ P_{2}$}}

\color{black}

\put(-40,23){\makebox(0,0)[bc]{$M_{t}= f_{t}(P_{2})$}}

\color{black}

\put(30,0){\circle*{1.5}}
\put(30,-2){\makebox(0,0)[tc]{$P_{3}$}}

\put(0,0){\qbezier[500](-80,0)(-80,20)(-60,30)}
\put(0,0){\qbezier[500](-60,30)(-40,40)(0,40)}
\put(0,0){\qbezier[500](80,-10)(80,20)(60,30)}
\put(0,0){\qbezier[500](60,30)(40,40)(0,40)}


\put(0,0){\qbezier[500](80,0)(80,-40)(40,-40)}
\put(0,0){\qbezier[500](40,-40)(0,-40)(0,0)}
\put(0,0){\qbezier[500](0,0)(0,20)(-30,20)}

\put(0,0){\qbezier[500](-30,20)(-55,20)(-55,0)}
\put(0,0){\qbezier[500](-55,0)(-55,-30)(0,-30)}
\put(0,0){\qbezier[500](0,-30)(55,-30)(55,0)}

\put(0,0){\qbezier[500](55,0)(55,20)(30,20)}
\put(0,0){\qbezier[500](30,20)(0,20)(0,0)}
\put(0,0){\qbezier[500](0,0)(0,-40)(-40,-40)}
\put(0,0){\qbezier[500](-40,-40)(-80,-40)(-80,0)}

\color{black}


\color{mygreen}

\put(-55,-1){\vector(0,-1){20}}
\put(-54.8,-1){\line(0,-1){18}}
\put(-55.2,-1){\line(0,-1){18}}
\put(-54.6,-1){\line(0,-1){18}}
\put(-55.4,-1){\line(0,-1){18}}

\put(80,-1){\vector(0,-1){20}}
\put(79.8,-1){\line(0,-1){18}}
\put(80.2,-1){\line(0,-1){18}}
\put(79.6,-1){\line(0,-1){18}}
\put(80.4,-1){\line(0,-1){18}}

\put(0,1){\vector(0,1){20}}
\put(-0.2,1){\line(0,1){18}}
\put(0.2,1){\line(0,1){18}}
\put(-0.4,1){\line(0,1){18}}
\put(0.4,1){\line(0,1){18}}

\color{black}

\put(-82,10){\makebox(0,0)[cr]{$\scriptstyle Z$}}
\put(-62,-10){\makebox(0,0)[cr]{$\scriptstyle \id$}}
\put(-55,0){\circle*{1.5}}

\put(55,0){\circle*{1.5}}

\put(62,10){\makebox(0,0)[cl]{$\scriptstyle Z$}}
\put(82,-10){\makebox(0,0)[cl]{$\scriptstyle \id$}}
\put(-2,15){\makebox(0,0)[cr]{$\scriptstyle \id$}}
\put(2,-10){\makebox(0,0)[cl]{$\scriptstyle Z$}}

\color{red}

\put(-80,1){\vector(0,1){20}}
\put(-79.8,1){\line(0,1){18}}
\put(-80.2,1){\line(0,1){18}}
\put(-79.6,1){\line(0,1){18}}
\put(-80.4,1){\line(0,1){18}}

\put(55,1){\vector(0,1){20}}
\put(54.8,1){\line(0,1){18}}
\put(55.2,1){\line(0,1){18}}
\put(54.6,1){\line(0,1){18}}
\put(55.4,1){\line(0,1){18}}

\put(0,-1){\vector(0,-1){20}}
\put(-0.2,-1){\line(0,-1){18}}
\put(0.2,-1){\line(0,-1){18}}
\put(0.2,-1){\line(0,-1){18}}
\put(-0.4,-1){\line(0,-1){18}}

\color{blue}

\put(-30,20){\circle*{1}}

\put(-30,20){\vector(-1,0){10}}
\put(-30,20.2){\line(-1,0){9}}
\put(-30,19.8){\line(-1,0){9}}

\put(0,-30){\circle*{1}}

\put(0,-30){\vector(1,0){10}}
\put(0,-30.2){\line(1,0){9}}
\put(0,-29.8){\line(1,0){9}}

\put(30,20){\circle*{1}}

\put(30,20){\vector(-1,0){10}}
\put(30,20.2){\line(-1,0){9}}
\put(30,19.8){\line(-1,0){9}}

\put(-40,-40){\circle*{1}}

\put(-40,-40){\vector(-1,0){10}}
\put(-40,-40.2){\line(-1,0){9}}
\put(-40,-39.8){\line(-1,0){9}}

\put(40,-40){\circle*{1}}
\put(40,-40){\vector(-1,0){10}}
\put(40,-40.2){\line(-1,0){9}}
\put(40,-39.8){\line(-1,0){9}}

\put(0,40){\circle*{1}}
\put(0,40){\vector(1,0){10}}
\put(0,40.2){\line(1,0){9}}
\put(0,39.8){\line(1,0){9}}

\end{picture}
\caption{The important curve $(AB)^{3}$ corresponds to $gdgdg$. The linking number of $P_{2}$ with each of the points $P_{1}$ and $P_{3}$ is zero but the isotopy class is of course not trivial. It is the minimum of the trace $\tr(f)$ on $\mathcal{T}$  \label{figure39}}

\end{figure}
\setlength{\unitlength}{0.7mm}

 \end{remarks}






\begin{figure}
\begin{picture}(100, 160)(-110,-70)




\color{blue}
\put(-110,0){\line(1,0){40}}
\put(-100,0){\vector(1,0){0}}

\color{black}

\put(-70,0){\circle*{1}}
\put(0,0){\circle*{1}}
\put(30,0){\circle*{1}}

\color{blue}

\put(-70,0){\qbezier[300](-25,0)(-25,-9.65)(-17.67,-17.67)}
\put(-87.67,-17.67){\vector(1,-1){0}}
\put(-70,0){\qbezier[300](-17.67,-17.67),(-9.65,-25),(0,-25)}
\put(-70,0){\qbezier[300](0,-25),(9.65,-25)(17.67,-17.67)}
\put(-70,0){\qbezier[300](17.67,-17.67)(25,-9.65)(25,0)}

\put(0,0){\qbezier[300](-45,0)(-45,17.2)(-31.8,31.8)}
\put(0,0){\qbezier[300](-31.8,31.8),(-17.2,45),(0,45)}
\put(0,0){\qbezier[300](0,45),(17.2,45)(31.8,31.8)}
\put(31.8,31.8){\vector(1,-1){0}}
\put(0,0){\qbezier[300](31.8,31.8)(45,17.2)(45,0)}


\put(30,0){\qbezier[300](-15,0)(-15,-5.73)(-10.6,-10.6)}
\put(30,0){\qbezier[300](-10.6,-10.6),(-5.73,-15),(0,-15)}
\put(30,0){\qbezier[300](0,-15),(5.73,-15)(10.6,-10.6)}
\put(19.4,-10.6){\vector(-1,1){0}}
\put(30,0){\qbezier[300](10.6,-10.6)(15,-5.73)(15,0)}

\put(0,0){\qbezier[300](-15,0)(-15,5.73)(-10.6,10.6)}
\put(0,0){\qbezier[300](-10.6,10.6),(-5.73,15),(0,15)}
\put(0,0){\qbezier[300](0,15),(5.73,15)(10.6,10.6)}
\put(-10.6,10.6){\vector(-1,-1){0}}
\put(0,0){\qbezier[300](10.6,10.6)(15,5.73)(15,0)}

\put(30,0){\qbezier[300](-45,0)(-45,-17.2)(-31.8,-31.8)}
\put(30,0){\qbezier[300](-31.8,-31.8),(-17.2,-45),(0,-45)}
\put(30,0){\qbezier[300](0,-45),(17.2,-45)(31.8,-31.8)}
\put(61.8,-31.8){\vector(1,1){0}}
\put(30,0){\qbezier[300](31.8,-31.8)(45,-17.2)(45,0)}

\put(0,0){\qbezier[300](75,0)(75,28.75)(53.02,53.02)}
\put(0,0){\qbezier[300](53.02,53.02),(28.75,75),(0,75)}
\put(0,0){\qbezier[300](0,75),(-28.75,75)(-53.02,53.02)}
\put(0,0){\qbezier[300](-53.02,53.02)(-75,28.75)(-75,0)}
\put(-53.02,53.02){\vector(-1,-1){0}}

\color{mygreen}

\put(-70,0){\line(1,0){70}}
\put(-60,0){\vector(1,0){0}}

\put(0,0){\qbezier[300](-65,0)(-65,24.83)(-45.95,45.95)}
\put(0,0){\qbezier[300](-45.95,45.95),(-24.83,65),(0,65)}
\put(0,0){\qbezier[300](0,65),(24.83,65)(45.95,45.95)}
\put(0,0){\qbezier[300](45.95,45.95)(65,24.83)(65,0)}
\put(45.95,45.95){\vector(1,-1){0}}

\put(30,0){\qbezier[300](35,0)(35,-13.37)(24.74,-24.74)}
\put(30,0){\qbezier[300](24.74,-24.74),(13.37,-35),(0,-35)}
\put(30,0){\qbezier[300](0,-35),(-13.37,-35)(-24.74,-24.74)}
\put(5.26,-24.74){\vector(-1,1){0}}
\put(30,0){\qbezier[300](-24.74,-24.74)(-35,-13.37)(-35,0)}

\color{red}

\put(0,0){\line(1,0){30}}
\put(10,0){\vector(1,0){0}}
\put(30,0){\qbezier[300](-25,0)(-25,-9.55)(-17.67,-17.67)}
\put(30,0){\qbezier[300](-17.67,-17.67),(-9.55,-25),(0,-25)}
\put(30,0){\qbezier[300](0,-25),(9.55,-25)(17.67,-17.67)}
\put(47.67,-17.67){\vector(1,1){0}}
\put(30,0){\qbezier[300](17.67,-17.67)(25,-9.55)(25,0)}


\put(0,0){\qbezier[300](55,0)(55,21.01)(38.88,38.88)}
\put(0,0){\qbezier[300](38.88,38.88),(21.01,55),(0,55)}
\put(0,0){\qbezier[300](0,55)(-21.01,55)(-38.88,38.88)}
\put(-38.88,38.88){\vector(-1,-1){0}}
\put(0,0){\qbezier[300](-38.88,38.88),(-55, 21.01),(-55,0)}

\put(-70,0){\qbezier[300](15,0)(15,-5.73)(10.6,-10.6)}
\put(-70,0){\qbezier[300](10.6,-10.6),(5.73,-15),(0,-15)}
\put(-70,0){\qbezier[300](0,-15),(-5.73,-15)(-10.6,-10.6)}
\put(-80.6,-10.6){\vector(-1,1){0}}
\put(-70,0){\qbezier[300](-10.6,-10.6)(-15,-5.73)(-15,0)}

\put(0,0){\qbezier[300](-85,0)(-85,32.47)(-60.1,60.1)}
\put(0,0){\qbezier[300](-60.1,60.1),(-32.47,85),(0,85)}
\put(0,0){\qbezier[300](0,85),(32.47,85)(60.1,60.1)}
\put(60.1,60.1){\vector(1,-1){0}}
\put(0,0){\qbezier[300](60.1,60.1)(85, 32.47)(85,0)}


\put(30,0){\qbezier[300](55,0)(55,-21.01)(38.88,-38.88)}
\put(30,0){\qbezier[300](38.88,-38.88),(21.01,-55),(0,-55)}
\put(30,0){\qbezier[300](0,-55)(-21.01,-55)(-38.88,-38.88)}
\put(-8.88,-38.88){\vector(-1,1){0}}
\put(30,0){\qbezier[300](-38.88,-38.88),(-55, -21.01),(-55,0)}

\put(0,0){\qbezier[300](-25,0)(-25,9.55)(-17.67,17.67)}
\put(0,0){\qbezier[300](-17.67,17.67),(-9.55,25),(0,25)}
\put(0,0){\qbezier[300](0,25),(9.55,25)(17.67,17.67)}
\put(17.67,17.67){\vector(1,-1){0}}
\put(0,0){\qbezier[300](17.67,17.67)(25,9.55)(25,0)}

\color{black}

\put(30,0){\line(1,0){70}}
\put(100,0){\vector(1,0){0}}

\put(0,0){\qbezier[300](35,0)(35,13.37)(24.74,24.74)}
\put(0,0){\qbezier[300](24.74,24.74),(13.37,35),(0,35)}
\put(0,0){\qbezier[300](0,35),(-13.37,35)(-24.74,24.74)}
\put(-24.74,24.74){\vector(-1,-1){0}}
\put(0,0){\qbezier[300](-24.74,24.74)(-35,13.37)(-35,0)}



\put(30,0){\qbezier[300](-65,0)(-65,-24.83)(-45.95,-45.95)}
\put(30,0){\qbezier[300](-45.95,-45.95)(-24.83,-65)(0,-65)}

\put(30,0){\qbezier[300](0,-65),(24.83,-65)(45.95, -45.95)}
\put(30,0){\qbezier[300](45.95, -45.95)(65,-24.83)(65,0)}

\put(75.95,-45.95){\vector(1,1){0}}




\end{picture}
\caption{A topological snail. \label{figure30}}

\end{figure}

\end{document}